\documentclass[12pt]{article}
\usepackage{a4wide}
\usepackage{amsmath, amsthm, amsfonts, amssymb, bbm}
\usepackage[mathscr]{eucal}
\usepackage{graphics}
\usepackage{xypic}
\usepackage[all]{xy}
\usepackage[english]{babel}
\usepackage[font=small,format=plain,labelfont=bf,up]{caption}
\usepackage[bbgreekl]{mathbbol}
 
\usepackage{hyperref}       
\hypersetup{       
   pdftex,                  
   colorlinks=true,        
   pdfstartview=FitH,      
   allcolors=[rgb]{0,0.3,0.6},        
   bookmarks=true,          
	 bookmarksdepth=section, 
   bookmarksopen=false,     
   bookmarksnumbered=true, 
	 backref, pagebackref   
}

\allowdisplaybreaks

\theoremstyle{plain}

\newtheorem{theorem}{Theorem}
\newtheorem{lemma}[theorem]{Lemma}
\newtheorem{proposition}[theorem]{Proposition}
\newtheorem{corollary}[theorem]{Corollary}

\theoremstyle{definition}
\newtheorem{remark}[theorem]{Remark}
\newtheorem{definition}[theorem]{Definition}

 \numberwithin{equation}{section}
 \numberwithin{theorem}{section}
 
\newcommand\void[1]{}

\newcommand\be            {\begin{equation}}
\newcommand\ee            {\end{equation}}

\DeclareMathOperator{\End}{End}
\DeclareMathOperator{\Hom}{Hom}
\DeclareMathOperator{\Rep}{Rep}
\DeclareMathOperator{\ev}{ev}
\DeclareMathOperator{\coev}{coev}
\DeclareMathOperator{\Gr}{Gr}
\DeclareMathOperator{\Irr}{Irr}

\newcommand{\rmod}[1]{\mathrm{mod}\text{-}#1}
\newcommand{\smap}{\sigma}
\newcommand{\ot}{\otimes}
\newcommand{\coend}{\mathcal{L}}

\newcommand\SF{\mathcal{S}\hspace{-.65pt}\mathcal{F}}
\newcommand\h            {\mathfrak{h}}

\newcommand{\Rad}{\rho}

\newcommand{\Hig}{\mathrm{Hig}}
\newcommand{\Rey}{\mathrm{Rey}}
\newcommand{\Jac}{\mathrm{Jac}}
\newcommand\Proj{\mathcal{P}roj}
\newcommand{\CM}{\mathsf{C}}

\newcommand\coint		{\Lambda^{\mathrm{co}}}

\newcommand\eps           {\varepsilon}
\newcommand\id            {id}
\newcommand\Id            {I\hspace{-1pt}d}
\newcommand\one           {{\bf1}}

\newcommand\Cb            {\mathbb{C}}

\newcommand\Rb            {\mathbb{R}}
\newcommand\Zb            {\mathbb{Z}}

\newcommand\Ac            {\mathcal{A}}
\newcommand\Cc            {\mathcal{C}}

\newcommand\Mc            {\mathcal{M}}
\newcommand\Tc            {\mathcal{T}}

\newcommand{\sVect}{{\it s}{\cal V}{\it ect}}

\newcommand\ldX{{}^*\!X}

\begin{document}

\thispagestyle{empty}
\def\thefootnote{\fnsymbol{footnote}}
\begin{flushright}
ZMP-HH/17-11\\
Hamburger Beitr\"age zur Mathematik 651
\end{flushright}
\vskip 3em
\begin{center}\LARGE
Projective objects and 
the modified trace\\
in factorisable finite tensor categories 
\end{center}

\vskip 2em
\begin{center}
{\large 
Azat M. Gainutdinov\,$^{a,b}$~~and~~Ingo Runkel\,$^b$~~\footnote{Emails: {\tt azat.gainutdinov@lmpt.univ-tours.fr}, {\tt ingo.runkel@uni-hamburg.de}}}
\\[1.5em]
{\sl\small $^a$ Institut Denis Poisson, CNRS, Universit\'e de Tours, Universit\'e d'Orl\'eans,\\ Parc de Grandmont, 37200 Tours, France}\\[0.5em]
{\sl\small $^b$ Fachbereich Mathematik, Universit\"at Hamburg\\
Bundesstra\ss e 55, 20146 Hamburg, Germany}
\end{center}

\vskip 1mm

\begin{abstract}
For $\Cc$ a factorisable and pivotal finite tensor category over an algebraically closed field of characteristic zero we show:
\begin{enumerate}
\item $\Cc$ always contains a simple projective object;
\item if $\Cc$ is in addition ribbon, the internal characters of projective modules span a submodule for the projective $SL(2,\Zb)$-action;
\item 
	the action of the Grothendieck ring 
of $\Cc$ on the span of
	internal characters of projective objects can be diagonalised;
\item the linearised Grothendieck ring of $\Cc$ is semisimple iff $\Cc$ is semisimple.
\end{enumerate}
Results 1--3 remain true in positive characteristic under an extra assumption. 
Result~1 implies that the tensor ideal of projective objects in $\Cc$ carries a unique-up-to-scalars modified trace function.
We express the modified trace of open Hopf links coloured by projectives in terms of  $S$-matrix elements.
Furthermore, we give a Verlinde-like formula for the decomposition of tensor products of projective objects which uses only the modular 
$S$-transformation restricted to
	internal characters of projective objects.

We compute 
the modified trace
 in the example of symplectic fermion categories, and we illustrate how the Verlinde-like formula for projective objects can be applied there.
\end{abstract}

\vskip 1mm
\textit{Keywords:} Braided and pivotal categories, projective objects, modified traces,\\
\mbox{}\qquad\qquad\qquad\; Grothendieck rings, Vertex-Operator Algebras, Verlinde-like formula.\\
\mbox{}\quad\;\,\textit{MSC codes:}  18D10, 18G05, 18F30, 17B69.

\setcounter{footnote}{0}
\def\thefootnote{\arabic{footnote}}

\newpage

\tableofcontents

\section{Introduction}

Let $H$ be a finite-dimensional quasi-triangular Hopf algebra over a field $k$
and denote by $R$ its universal $R$-matrix.
$H$ is called {\em factorisable} if the map 
$H^* \to H:
f\mapsto (f\otimes\id)(R_{21}R)$ 
(the {\em Drinfeld map} \cite{Drinfeld}) is bijective. If $H$ is factorisable and ribbon, the centre $Z(H)$ of $H$ carries a projective representation of $SL(2,\Zb)$, the mapping class group of the torus \cite{Lyubashenko:1994ab}. In this case, the Drinfeld map furthermore allows one to give a faithful representation of the 
	$k$-linear
Grothendieck ring of $\Rep(H)$ on $Z(H)$ \cite{Drinfeld}.

A factorisable Hopf algebra $H$ can be endowed with a central form $\nu : H \to k$ such that it becomes a symmetric Frobenius algebra \cite[Rem.\,3.1]{Cohen:2008}. The central form $\nu$ induces an isomorphism between $Z(H)$ and the space $C(H)$ of all central forms on $H$:
\be\label{eq:intro-Z-C}
	Z(H) \longrightarrow C(H)	
	\quad , \quad
	z \longmapsto \nu(z \cdot -) \ .
\ee
Characters of finite-dimensional $H$-modules are examples of central forms. We are particularly interested in the preimage in $Z(H)$ of the characters of projective modules. This subspace is called the {\em Higman ideal} $\Hig(H)$ 
	(see e.g.\ \cite{Broue:2009}).

Denote by $\Rep(H)$ the category of finite-dimensional representations of $H$. 
In \cite{Cohen:2008}, Cohen and Westreich prove, amongst other things, the following remarkable results.

\begin{theorem}[\cite{Cohen:2008}]\label{thm:cw-hopf-results}
Let $H$ be a factorisable ribbon Hopf algebra over an algebraically closed field of characteristic zero. 
\begin{enumerate}
\item There is at least one simple and projective $H$-module.
\item $\Hig(H)$ is an $SL(2,\Zb)$-submodule of $Z(H)$.
\item $\Hig(H)$ is a submodule for the action of the Grothendieck ring of $\Rep(H)$ on $Z(H)$, and this action can be diagonalised on $\Hig(H)$ (``the fusion rules can be diagonalised on $\Hig(H)$'').
\end{enumerate}
\end{theorem}

Part 2 generalises~\cite[Thm.\,5.5]{Lachowska-center}, 
\cite[Cor.\,4.1]{Lachowska-Verlinde} 
established for small quantum groups associated to simple Lie algebras of ADE type.

\medskip
The motivation of the present paper is to give a generalisation of these results to factorisable finite tensor categories. Let us describe our setting in more detail.

Let $k$ be a field. An abelian $k$-linear category is called {\em finite} if it is equivalent to the category of finite-dimensional representations of a finite-dimensional $k$-algebra. A {\em finite tensor category} is a finite abelian category which is rigid monoidal with bilinear tensor product and simple tensor unit. 
	If the category is equipped with a natural monoidal isomorphism from the identity functor to the endofunctor given by taking double duals, it is called {\em pivotal}.
An important class of examples of finite tensor categories are the categories $\Rep(H)$ for finite-dimensional Hopf algebras $H$. 
	A Hopf algebra $H$ is called pivotal if the squared antipode of $H$ is given by conjugating with a group-like element~\cite{AAGTV}. In this case,  $\Rep(H)$ is pivotal. A ribbon Hopf algebra is automatically pivotal.

For a braided finite tensor category one can write down at least four natural non-degeneracy 
	conditions
for the braiding. These have only recently been shown to all be equivalent \cite{Shimizu:2016}, and we briefly review this result in Section~\ref{sec:conventions}. We refer to braided finite tensor categories satisfying these equivalent conditions as {\em factorisable}. Indeed, for a quasi-triangular finite-dimensional Hopf algebra $H$, $\Rep(H)$ is factorisable iff $H$ is factorisable~\cite{Lyubashenko:1995}.

Our first main result generalises part 1 of Theorem~\ref{thm:cw-hopf-results} 
(see Theorem~\ref{thm:main-simpleproj}).

\begin{theorem}\label{thm:intro2}
A factorisable and pivotal finite tensor category $\Cc$ over an algebraically closed field of characteristic zero contains a simple projective object.
\end{theorem}

\begin{remark}
One can formulate the above theorem for positive characteristic if one adds an extra assumption, which we call ``Condition~P'' (Section~\ref{sec:projsimpleex}). Namely, we say that a finite braided tensor category satisfies Condition~P if there exists a projective object $P$ such that~$[P]$ is not nilpotent in the linearised Grothendieck ring (that is, in the Grothendieck ring tensored with the field). In characteristic zero, Condition~P is always satisfied (Lemma~\ref{lem:powers-nonzero-in-Gr}). 
Theorems~\ref{thm:intro2}, \ref{thm:intro3} and part 1 of Theorem~\ref{thm:intro5} 
remain true if instead of requiring the field to be of characteristic zero, one demands that the category satisfies Condition~P
(but for part 2 of Theorem~\ref{thm:intro5} we still need to require characteristic zero). 
The details are given in the body of the paper. 
\end{remark}

One noteworthy consequence is that the category $\Cc$ in the above theorem allows for 
a unique-up-to-scalars non-zero modified trace function
on the tensor ideal $\Proj(\Cc)$, 
that is, on the full subcategory of all projective objects in $\Cc$
	\cite{Geer:2010,Geer:2011a,Geer:2011b}, see Section~\ref{sec:traces}. 

A modified  trace on $\Proj(\Cc)$ is a family of $k$-linear maps $t_P : \End(P) \to k$, $P \in \Proj(\Cc)$, subject to a cyclicity and partial trace condition.
Recall that the 
categorical trace defined in terms of the pivotal structure on $\Cc$  vanishes identically on $\Proj(\Cc)$ 
unless $\Cc$ is semisimple
(Remark~\ref{rem:modified-trace-consequences}). In contrast to this, for the modified trace the
pairings $\Cc(P,Q) \times \Cc(Q,P) \to k$, 
	$(f,g) \mapsto t_Q(f \circ g)$ 
are non-degenerate for all $P,Q \in \Proj(\Cc)$ 
(see \cite{Costantino:2014sma} and Proposition~\ref{prop:mtrace} below).
In particular, the $t_P$ turn $\Proj(\Cc)$ into a Calabi-Yau category (Definition~\ref{def:CY-cat}).

One can use modified traces to define new link invariants in finite ribbon categories which are not accessible via the usual quantum 
traces~\cite{Geer:2009}. However, we will not pursue this point in the present paper.

\medskip

Let $\Ac$ be a finite abelian category over a field $k$ and assume that $\Proj(\Ac)$ is Calabi-Yau.
In Section~\ref{sec:sym-alg-ideal} we give a categorical definition of the Higman ideal as an ideal 
\be
	\Hig(\Ac) ~\subset~ \End(\Id_\Ac)
\ee
in the algebra of
natural endomorphisms of the identity functor on $\Ac$.
Recall that for algebras, $Z(A) \cong \End(\Id_{\rmod{A}})$, i.e.\ the centre of $A$ is isomorphic as a $k$-algebra to the
algebra of natural endomorphisms of the identity functor on the category of finite-dimensional (right, say) $A$-modules. The definition of $\Hig(\Ac)$ is such that via this isomorphism, $\Hig(A)$ gets mapped to $\Hig(\rmod{A})$.

	As in Theorem~\ref{thm:intro2} 
above, let $\Cc$ be a factorisable and pivotal finite tensor category over an algebraically closed field $k$ of characteristic zero. 
Let $G$ be a projective generator of $\Cc$ and $E = \End(G)$. Then $\Cc$ is equivalent as a $k$-linear category to $\rmod{E}$. As discussed above, we obtain a modified trace $t_G : E \to k$ which turns $E$ into a symmetric Frobenius algebra, and as in \eqref{eq:intro-Z-C} we get isomorphisms
\be\label{eq:intro-End-Z-C}
	\End(\Id_\Cc) \longrightarrow Z(E) \longrightarrow C(E)
	\quad , \quad
	\eta \longmapsto \eta_G \longmapsto t_G(\eta_G \circ -) \ .
\ee
Our second main result is Theorem~\ref{thm:trace-vs-tG}, which links characters of $E$-modules to
	the modified trace over certain elements in $\End(\Id_\Cc)$ obtained from the internal character of the corresponding object in $\Cc$. 
The details of this are
too lengthy for this introduction, and we refer to Section~\ref{sec:ch-tr}.

Theorem~\ref{thm:trace-vs-tG} is important 	for what follows  as it implies that $\Hig(\Cc)$ is spanned by the
	images of	
internal characters of projective objects (Proposition~\ref{prop:ReyHig=span}).

\medskip

For a factorisable finite ribbon category one obtains a projective $SL(2,\Zb)$-action on $\End(\Id_\Cc)$, and in fact actions of all surface mapping class groups on appropriate Hom-spaces \cite{Lyubashenko:1994tm}.

Our third main result generalises part 2 of Theorem~\ref{thm:cw-hopf-results} (see Corollary~\ref{cor:Hig-SL2Z-sub})

\begin{theorem}\label{thm:intro3}
Let $\Cc$ be a factorisable finite ribbon category over an algebraically closed field of characteristic zero. Then $\Hig(\Cc)$ is an $SL(2,\Zb)$-submodule of $\End(\Id_\Cc)$.
\end{theorem}

We note that under the conditions of this theorem, $\Hig(\Cc)=\End(\Id_\Cc)$ iff $\Cc$ is semisimple (Proposition~\ref{prop:cy-ssi-inclusion}).

Still under the conditions of the above theorem, one can define an injective ring homomorphism from the Grothendieck ring $\Gr(\Cc)$ to $\End(\Id_\Cc)$ \cite{Shimizu:2015} (see Section~\ref{sec:conventions}). 
This, in particular, turns $\End(\Id_\Cc)$ into a faithful $\Gr(\Cc)$-module by defining the representation map 
	$\rho^L : \Gr(\Cc) \to \End_k(\End(\Id_\Cc))$ 
to be left multiplication. 
We show the following statement, the first part of which generalises part 3 
of Theorem~\ref{thm:cw-hopf-results} (see Proposition~\ref{prop:GrC-diagonalisability}).

\begin{theorem}\label{thm:intro5}
Let $\Cc$ be a factorisable and pivotal finite tensor category over an algebraically closed field
	of characteristic zero.
\begin{enumerate}
\item The restriction of $\rho^L$ to the submodule $\Hig(\Cc)$ is diagonalisable.
\item The action $\rho^L$ can be diagonalised on $\End(\Id_\Cc)$ if and only if $\Cc$ is semisimple.
\end{enumerate}
\end{theorem}

See Proposition~\ref{prop:GrC-diagonalisability} for an explicit choice of basis diagonalising the $\Gr(\Cc)$-action. 
In particular, the $k$-linear Grothendieck ring $\Gr_k(\Cc)$ is semisimple iff $\Cc$ is semisimple (Corollary~\ref{cor:Grss-Css}).

\smallskip

We also investigate in Section~\ref{sec:Sinv-proj} how the restriction of the 
	projective
$SL(2,\Zb)$-action to $\Hig(\Cc)$ can be used to gain information about the decomposition of tensor products of projective objects, see Proposition~\ref{prop:M-via-S} for details.
 This can be thought of as a non-semisimple variant 
of the Verlinde formula \cite{Verlinde:1988sn}. 
See e.g.\ \cite{Fuchs:2003yu,Fuchs:2006nx,Gainutdinov:2016qhz,Creutzig:2016fms} for discussions of the Verlinde formula in the finitely non-semisimple setting, and~\cite{Cohen:2008} for a related result in the context of factorisable ribbon Hopf algebras.
As a corollary of the projective
$SL(2,\Zb)$-action on $\Hig(\Cc)$, we also give  a purely categorical counterpart
of the conjectural relation between modular properties of pseudo-trace functions and modified trace of Hopf link operators  \cite{Creutzig:2016fms} --- we prove a special case of this (namely when all labels are projective) in Proposition~\ref{prop:mod-tr-S}.

\begin{remark}\label{rem:motivation}
~\\[-1.5em]
\begin{enumerate}\setlength{\leftskip}{-1em}
\item 
A semisimple factorisable finite ribbon category is a modular tensor category \cite{tur,baki}. Since the qualifier ``modular'' refers to the projective action of the modular group $SL(2,\Zb)$, it would also be reasonable to refer to modular tensor categories in the sense of \cite{tur,baki} as {\em modular fusion categories}, and to call factorisable finite ribbon categories {\em modular finite tensor categories}. An important application of modular fusion categories is that they are precisely the data needed to construct certain three-dimensional topological field theories \cite{retu,tur,Bartlett:2015baa}.
For modular finite tensor categories in the above sense, a corresponding result is not known. However, a version of three-dimensional topological field theory build form such categories is studied in \cite{Kerler:2001}.
\item
	So far,
we have motivated our results from the theory of Hopf algebras. 
Another reason to look at factorisable finite tensor categories is provided by vertex operator algebras. A particularly benign 
class of vertex operator algebras $V$ are the so-called $C_2$-cofinite ones, for which we in addition require that they are simple, non-negatively graded and isomorphic to their contragredient module. Let $V$ be such a vertex operator algebra.

If $\Rep(V)$, the category of (quasi-finite dimensional, generalised) $V$-modules, is in addition semisimple, it is in fact a modular fusion category \cite{Huang2005}.	What is more, it is proved in \cite{Zhu1996} that the characters of $V$-modules form a projective representation of the modular group, which by \cite{Huang:2004bn} agrees with the representation obtained from the modular fusion category $\Rep(V)$.

If $\Rep(V)$ is not semisimple, 
it is known from \cite{Huang:2010,Huang:2007mj} that $\Rep(V)$ satisfies all properties of a finite braided tensor category except for rigidity (which is, however, conjectured in \cite{Huang:2009xq,Gainutdinov:2016qhz} and proven in a different setting for $W_p$ models in~\cite{Tsuchiya:2012ru}). 
It is furthermore conjectured that also in the non-semisimple case, $\Rep(V)$ is factorisable \cite{Gainutdinov:2016qhz,Creutzig:2016fms}.
For the modular group action one now has to pass to pseudo-trace functions \cite{Miyamoto:2002ar,Arike:2011ab}, and it is tempting to conjecture that this modular group action agrees with the one on $\End(\Id_{\Rep(V)})$ discussed above, see \cite[Conj.\,5.10]{Gainutdinov:2016qhz} for a precise formulation. 
For so-called symplectic fermion conformal field theories \cite{Kausch:1995py,Gaberdiel:1996np,Abe:2005}, the modular group actions obtained from pseudo-trace functions and from $\Rep(V)$ have been compared in \cite{Gainutdinov:2015lja,FGRprep2} and they indeed agree.

If these conjectures were true,
 Theorems~\ref{thm:intro2} and \ref{thm:intro3} would imply that such a~$V$ always has at least one simple and projective module, and that the characters (not the pseudo-trace functions) of projective modules transform into each other under the modular group action. Both properties hold in the very few non-semisimple examples where they can be verified 
\cite{Kausch:1995py,Fuchs:2003yu,[FGST],Nagatomo:2009xp,Abe:2005}.

The study of properties of factorisable tensor categories and in particular of $\Rep(V)$ is also important input for the construction of bulk
conformal field theory correlators in the non-semisimple setting as in \cite{Fuchs:2010mw,Fuchs:2013lda,Fuchs:2016wjr}.
\end{enumerate}
\end{remark}

We conclude the paper in Section~\ref{sec:sf-example} with the computation of the modified trace
in a class of examples of factorisable finite ribbon categories 
\cite{Davydov:2012xg,Runkel:2012cf}
arising in symplectic fermion conformal field theories.
We also illustrate in this example how the projective $SL(2,\Zb)$-action on the Higman ideal can be used to obtain information about the tensor product of projective objects.

\medskip

\noindent
{\bf Acknowledgements:}
We thank
	Anna Beliakova,
	Christian Blanchet,
	Alexei Davydov,
	J\"urgen Fuchs,
	Nathan Geer,
	Jan Hesse,
	Dimitri Nikshych,
	Ehud Meir,
	Victor Ostrik and
	Christoph Schweigert
for helpful discussions, and we are grateful to
	Miriam Cohen,
	Alexei Davydov,
	Ehud Meir and
	Sara Westreich
for comments on a draft of this paper.
The work of AMG was  supported by  DESY and CNRS. 

\mbox{}\\
\textbf{Convention}:
Through this paper, $k$ denotes an algebraically closed field, possibly of positive characteristic.\footnote{
	Some definitions and results of this paper do not actually require algebraic closedness of $k$. For example, algebraic closedness is not required in the first half of Section \ref{sec:conventions}, up to and excluding Theorem~\ref{thm:chiinj}, 
or in the definition of a modified right trace in Section \ref{sec:traces}, etc. But to avoid confusion we prefer to require algebraic closedness for the entire paper.}

\section{Factorisable finite tensor categories}\label{sec:conventions}

In this section we introduce notation and recall some definitions and results that will be used later, such as the different characterisations of a factorisable finite tensor category. The reader is invited to just skim through this section and return to it when necessary.

\medskip

For $\Ac$ an
	essentially small
abelian category we write (see e.g.\ \cite{EGNO-book} for more details)
\begin{itemize}
\item $\Irr(\Ac)$: a choice of representatives for the isomorphism classes of simple objects 
of~$\Ac$, i.e.\ every simple object is isomorphic to one and only one element of $\Irr(\Ac)$, 
\item $\Proj(\Ac)$: the full subcategory of $\Ac$ consisting of all projective objects,
\item $\Gr(\Ac)$: the Grothendieck group of $\Ac$. We write $[X] \in \Gr(\Ac)$ for the class of $X \in \Ac$ in the Grothendieck group,
\item $\Gr_k(\Ac)$: if $\Ac$ is $k$-linear
we set $\Gr_k(\Ac) := k \otimes_\Zb \Gr(\Ac)$. Note that the canonical ring homomorphism $\Gr(\Ac) \to \Gr_k(\Ac)$ is injective iff $k$ has characteristic zero.
\end{itemize}
A $k$-linear abelian category $\Ac$ is called {\em locally finite} if all Hom-spaces are finite dimensional over $k$ and every object of $\Ac$ has finite length. A locally finite abelian category is called {\em finite} if $\Irr(\Ac)$ is finite and every simple object has a projective cover. For a finite abelian category $\Ac$ we set
\begin{itemize}
\item $P_U$: a choice of projective cover $\pi_U : P_U \to U$ for each simple $U \in \Irr(\Ac)$.
\end{itemize}

One can show that a $k$-linear abelian category is finite iff it is equivalent as
	a
$k$-linear category to the category of finite-dimensional modules over a finite-dimensional $k$-algebra.

For a finite abelian category $\Ac$ we have $\Gr(\Ac) = \bigoplus_{U \in \Irr(\Ac)} \Zb\, [U]$. (On the other hand, if infinite direct sums are allowed in $\Ac$ we have $\Gr(\Ac)=0$ because of the short exact sequence $X \to \bigoplus_{\mathbb{N}} X \to \bigoplus_{\mathbb{N}} X$.)

\medskip

A monoidal category $\Mc$ is called {\em rigid} if for each $X \in \Mc$ there is  $X^* \in \Mc$ 
(the {\em left dual}) and 
$\ldX$ (the {\em right dual}), together with morphisms
\begin{align}
	\ev_X &: X^* \otimes X \to \one \ ,
	& \coev_X &: \one \to X \otimes X^* \ ,
\nonumber \\
	\widetilde\ev_X &: X \otimes \ldX \to \one \ ,
	& \widetilde\coev_X &: \one \to \ldX \otimes X \ ,
\end{align}
where $\one \in \Mc$ is the tensor unit. These morphisms have to satisfy the zig-zag identities, see e.g.\ \cite{EGNO-book} for more details.
We will later use string diagram notation. Our diagrams will be read bottom to top and the above duality morphisms will be written as
\begin{align}
  \raisebox{-0.3\height}{
  \begin{picture}(26,22)
   \put(0,6){\scalebox{.75}{\includegraphics{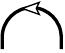}}}
   \put(0,6){
     \setlength{\unitlength}{.75pt}\put(-146,-155){
     \put(143,145)  {\scriptsize $ X^* $}
     \put(169,145)  {\scriptsize $ X $}
     }\setlength{\unitlength}{1pt}}
  \end{picture}}  
  &~=~ \ev_X 
  \quad , &
  \raisebox{-8pt}{
  \begin{picture}(26,22)
   \put(0,0){\scalebox{.75}{\includegraphics{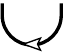}}}
   \put(0,0){
     \setlength{\unitlength}{.75pt}\put(-146,-155){
     \put(143,183)  {\scriptsize $ X $}
     \put(169,183)  {\scriptsize $ X^* $}
     }\setlength{\unitlength}{1pt}}
  \end{picture}}  
  &~=~ \coev_X 
  ~~,
\nonumber \\[1.3em]
  \raisebox{-8pt}{
  \begin{picture}(26,22)
   \put(0,6){\scalebox{.75}{\includegraphics{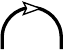}}}
   \put(0,6){
     \setlength{\unitlength}{.75pt}\put(-146,-155){
     \put(143,145)  {\scriptsize $ X $}
     \put(169,145)  {\scriptsize $ \ldX $}
     }\setlength{\unitlength}{1pt}}
  \end{picture}}  
  &~=~ \widetilde\ev_X 
  \quad , &
  \raisebox{-8pt}{
  \begin{picture}(26,22)
   \put(0,0){\scalebox{.75}{\includegraphics{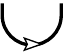}}}
   \put(0,0){
     \setlength{\unitlength}{.75pt}\put(-146,-155){
     \put(143,183)  {\scriptsize $ \ldX $}
     \put(169,183)  {\scriptsize $ X $}
     }\setlength{\unitlength}{1pt}}
  \end{picture}}  
  &~=~ \widetilde\coev_X 
  ~~.
\end{align}

A monoidal category $\Mc$ with left duals is {\em pivotal} if it is equipped with a natural monoidal isomorphism 
\be
	\delta : \Id_{\Mc} \to (-)^{**} \ .
\ee
It is then automatically rigid via $\ldX := X^*$ and (we omit the tensor product between objects for better readability)
\begin{align}\label{eq:ev-coev-tilde}
	\widetilde\ev_X &~=~ \big[\, XX^* \xrightarrow{\delta_X \otimes \id} X^{**} X^* \xrightarrow{\ev_{X^*}} \one \,\big] \ ,
\nonumber\\	
	\widetilde\coev_X &~=~ \big[\, \one \xrightarrow{\coev_{X^*}}X^{*} X^{**} 
	\xrightarrow{\id \otimes \delta_X^{-1}} X^{*} X \,\big] \ .
\end{align}

Let $\Mc$ be a monoidal category with right duals. Below we will make use of the functor ${}^*(-) \otimes (-) : \Mc^{\mathrm{op}} \times \Mc \to \Mc$, $(X,Y) \mapsto \ldX \otimes Y$, and of dinatural transformations between this functor and the constant functor with value $X$ for some $X \in \Mc$. We refer to \cite{MacLane-book} for dinatural transformations in general and 
	to \cite[Sec.\,4]{Fuchs:2010mw} or \cite[Sec.\,3]{FGRprep1}
for these specific functors. 

The following lemma is a slight generalisation of 
\cite[Lem.\,2.5.1]{Geer:2010} which in turn follows \cite{Deligne:2004}.
	It is proven in the same way as in \cite{Geer:2010} and we relegate the proof to Appendix~\ref{app:dinat-exact}.

\begin{lemma}\label{lem:dinat-Gr}
Let $\Mc$ be an
abelian monoidal category with right duals and biexact tensor product functor.
Suppose that there is a projective object $P \in \Mc$ and a surjection $p : P \to \one$.
Let $X \in \Mc$ and let
$\eta$ be a dinatural transformation from ${}^*(-) \otimes(-)$ to $X$. Consider the following commuting diagram in $\Mc$ with exact rows:
\be\label{eq:dinat-Gr_diag}
\xymatrix{
 0 \ar[r] & A \ar[r]^f \ar[d]^a & B \ar[r]^g \ar[d]^b & C \ar[r] \ar[d]^c & 0
\\
 0 \ar[r] & A \ar[r]^f & B \ar[r]^g & C \ar[r] & 0
 }
\ee
In this case, the following equality of morphisms $\one \to X$ holds
for any choice of endomorphisms $a$, $b$,  and $c$
	that make \eqref{eq:dinat-Gr_diag} 
commute:
\be\label{eq:dinat-Gr_rel}
\eta_B \circ (\id_{{\rule{0pt}{6pt}}^*\hspace{-2pt}B} \otimes b) \circ \widetilde\coev_B
~=~
\eta_A \circ (\id_{\,{\rule{0pt}{6pt}}^*\hspace{-2.5pt}A} \otimes a) \circ \widetilde\coev_A
+
\eta_C \circ (\id_{{\rule{0pt}{6pt}}^*\hspace{-1pt}C} \otimes c) \circ \widetilde\coev_C \ .
\ee
\end{lemma}

\begin{remark}\label{rem:all-endo-are-id}
~\\[-1.5em]
\begin{enumerate}\setlength{\leftskip}{-1em}
\item 
A corresponding statement as that in Lemma~\ref{lem:dinat-Gr} holds for a dinatural transformation $\xi$ from $(-) \otimes (-)^*$ to $X$. In this case, for $a,b,c$ as in \eqref{eq:dinat-Gr_diag} we have
\be\label{eq:dinat-Gr_rel2}
\xi_B \circ (b \otimes \id_{B^*}) \circ \coev_B
~=~
\xi_A \circ (a \otimes \id_{A^*}) \circ \coev_A
\,+\,
\xi_C \circ (c \otimes \id_{C^*}) \circ \coev_C
\ .
\ee
\item
As a consequence of Lemma~\ref{lem:dinat-Gr}, the morphism $\eta_A \circ \widetilde\coev_A \in \Mc(\one,X)$ only depends on the class $[A] \in \Gr(\Mc)$. Indeed, if one chooses $a,b,c$ in \eqref{eq:dinat-Gr_diag} to be identities, then \eqref{eq:dinat-Gr_rel} states that given a short exact sequence 
$0\to A \to B \to C\to 0$ 
we have the identity
\be
\eta_B \circ \widetilde\coev_B
~=~
\eta_A \circ \widetilde\coev_A
+
\eta_C \circ \widetilde\coev_C
\ee
of morphisms in $\Mc(\one,X)$. For example, if $\Mc$ is in addition pivotal, one can take $X=\one$ and the dinatural transformation $\eta_A = \ev_A$. One obtains the statement that the quantum dimension only depends on the class in $\Gr(\Mc)$.
\end{enumerate}
\end{remark}

A {\em finite tensor category} $\Cc$ is a category which (see \cite{Etingof:2003})
\begin{itemize}
\item is a $k$-linear finite abelian category,
\item is a rigid monoidal category with $k$-bilinear tensor product functor,
\item has a simple tensor unit.
\end{itemize}
The tensor product functor of an abelian rigid monoidal category is automatically biadditive and biexact \cite[Prop.\,4.2.1]{EGNO-book}.
A finite tensor category is called {\em unimodular} if $(P_\one)^* \cong P_\one$, see \cite{Etingof:2004,EGNO-book} for more details.

The following technical corollary to Lemma~\ref{lem:dinat-Gr} generalises \cite[Cor.\,2.5.2]{Geer:2010} and will be needed later. The proof is given in Appendix~\ref{app:dinat-exact}.

\begin{corollary}\label{cor:dinat-jac-0}
Let $\Cc$ be a finite tensor category over
$k$. 
Let $X \in \Cc$ and let
$\eta$ be a dinatural transformation from ${}^*(-) \otimes(-)$ to $X$
and $\xi$ from $(-) \otimes(-)^*$ to $X$.
Let $A \in \Cc$ and $f \in \End(A)$ be such that for all simple $U \in \Cc$ and all $u : A \to U$ we have $u \circ f = 0$.
 Then:
\begin{enumerate}
\item 
$\eta_A \circ (\id \otimes f) \circ \widetilde\coev_A = 0$ and
$\xi_A \circ (f \otimes \id) \circ \coev_A = 0$.
\item If $\Cc$ is in addition braided 
	and pivotal,
then for any $B \in \Cc$,
\be \label{eq:f-loop-zero}
  \raisebox{-0.5\height}{\setlength{\unitlength}{.75pt}
  \begin{picture}(60,170)
   \put(0,0){\scalebox{.75}{\includegraphics{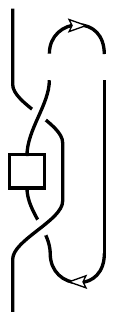}}}
   \put(0,0){
     \put(0,0){
     \put (5,0) {\scriptsize$ B $}
     \put (5,161) {\scriptsize$ B $}
     \put (22,125) {\scriptsize$ A $}
     \put (47,125) {\scriptsize$ A^* $}
     \put (12,76) {\scriptsize$ f $}
     }\setlength{\unitlength}{1pt}}
  \end{picture}}
 ~= 0 \ .
\ee	
\end{enumerate}
\end{corollary}

For the rest of this section, we fix
\begin{itemize}
\item $\Cc$: a braided and pivotal finite tensor category 
	over $k$.
\end{itemize}
Our notation for the coherence and braiding isomorphisms in $\Cc$ is, for $U,V,W \in \Cc$,
\begin{align}
\alpha_{U,V,W} &\,:\; U \otimes (V \otimes W) \longrightarrow (U \otimes V) \otimes W
& \text{(associator)} \ ,
\nonumber \\
\lambda_U &\,:\; \one \otimes U \longrightarrow U
\quad , \quad
\rho_U \,:\; U \otimes \one \longrightarrow U
& \text{(unit isomorphisms)} \ ,
\nonumber \\
c_{U,V} &\,:\; U \otimes V \longrightarrow V \otimes U
& \text{(braiding)} \ .
\end{align}
We write $\End(\Id_\Cc)$ for the natural endomorphisms of the identity functor 
on $\Cc$.  

\medskip

Given $V \in \Cc$, consider the natural transformation $\smap(V) 
\in \End(\Id_\Cc)$: for all $X \in \Cc$,
\begin{align}
	\smap(V)_X
&~=~
\big[
 X 
\xrightarrow{\sim} \one X
\xrightarrow{\widetilde\coev_{V} \ot \id} (V^{*} V) X
\xrightarrow{\sim} V^{*} (V X)
\xrightarrow{\id \ot (c^{-1}_{V,X} \circ c^{-1}_{X,V})} V^{*} (V X)
\nonumber \\ 
& \hspace{4em} \xrightarrow{\sim} (V^{*} V) X
\xrightarrow{\ev_{V} \otimes \id} \one X
\xrightarrow{\sim}X
\big] 
\nonumber \\ 
&~=~
   \raisebox{-0.5\height}{\setlength{\unitlength}{.75pt}
  \begin{picture}(70,118)
   \put(0,0){\scalebox{.75}{\includegraphics{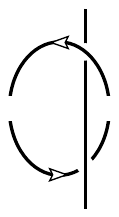}}}
   \put(0,0){
     \put(0,0){
     \put (38,0) {\scriptsize$ X $}
     \put (38,110) {\scriptsize$ X $}
     \put (50,55.5) {\scriptsize$ V $}
     \put (2,55.5) {\scriptsize$ V^* $}
     }\setlength{\unitlength}{1pt}}
  \end{picture}}
~\overset{(*)}=~  
   \raisebox{-0.5\height}{\setlength{\unitlength}{.75pt}
  \begin{picture}(70,118)
   \put(0,0){\scalebox{.75}{\includegraphics{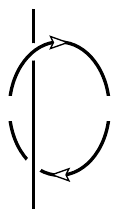}}}
   \put(0,0){
     \put(0,0){
     \put (13,0) {\scriptsize$ X $}
     \put (13,110) {\scriptsize$ X $}
     \put (0,55.5) {\scriptsize$ V^* $}
     \put (50,55.5) {\scriptsize$ V^{**} $}
     }\setlength{\unitlength}{1pt}}
  \end{picture}}
 \quad ,
\label{eq:eq:sigma-Gr-EndId-def_pic}
\end{align}
where in (*) we used properties of the braiding
to deform the string diagram, as well as pivotality to exchange left and right duality morphisms. 
This natural transformation appears
	in \cite[Sec.\,I.1.5]{tur} and
 also in \cite{Costantino:2014sma,Creutzig:2016fms} where it is called 
 ``general Hopf link'' and
 ``open Hopf link operator'', respectively
	(see~\cite[Rem.\,3.10\,(2)]{Gainutdinov:2016qhz}
	 for the conventions used here).

\begin{lemma}\label{lem:sigma_Gr_ringhom}
$\sigma(V)$ only depends on the class of $V$ in $\Gr(\Cc)$, and the resulting map $\Gr(\Cc) \to \End(\Id_\Cc)$
 is a ring homomorphism.
\end{lemma}

\begin{proof}
Let $X \in \Cc$ be arbitrary and
consider the dinatural transformation from $(-)^* \otimes(-)$ to $X \otimes X^*$ given by
	(see also \cite[Sec.\,4.4]{FGRprep1})
\be
	\eta_V ~:=~   
	\raisebox{-0.5\height}{\setlength{\unitlength}{.75pt}
  \begin{picture}(76,135)
   \put(0,0){\scalebox{.75}{\includegraphics{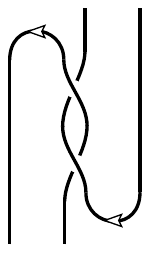}}}
   \put(0,0){
     \put(0,0){
     \put (4,0) {\scriptsize$ V^* $}
     \put (30,0) {\scriptsize$ V $}
     \put (40,128) {\scriptsize$ X $}
     \put (67,128) {\scriptsize$ X^* $}
     }\setlength{\unitlength}{1pt}}
  \end{picture}}
 \quad .
\ee
Note that $(\sigma(V)_X \otimes \id_{X^*}) \circ \coev_X = \eta_V \circ \widetilde\coev_V$. By Lemma~\ref{lem:dinat-Gr} (or rather Remark~\ref{rem:all-endo-are-id}\,(2)), the right hand side factors through $\Gr(\Cc)$, 
hence so does $\sigma$.
That the resulting map is a ring map follows from the identity $\sigma(V \otimes W)_X = \sigma(V)_X \circ \sigma(W)_X$, which one can verify by a straightforward calculation~\cite[Sec.\,4.5]{tur}
(see also \cite[Sec.\,3.1]{Creutzig:2016fms}).
\end{proof}

By abuse of notation, we also use the symbol $\smap$ for the resulting ring homomorphism
\be\label{eq:sigma-Gr-EndId-def}
\smap : \Gr(\Cc) \to \End(\Id_\Cc) \quad ,\quad \smap([V])_X \text{ is given by \eqref{eq:eq:sigma-Gr-EndId-def_pic}} \ .
\ee

Let $\coend$ be the coend 
(see e.g.\ \cite{Kerler:2001} or the review
 in~\cite[Sec.\,4]{Fuchs:2010mw} or in~\cite[Sec.\,3]{FGRprep1})
\be
	\coend ~=\, \int^{X \in \Cc} \hspace{-.5em} X^* \ot X 
\ee
and denote by 
$\iota_X : X^* \ot X \to \coend$, $X \in \Cc$,
the corresponding dinatural transformation.
The coend $\coend$ exists since $\Cc$ is a finite tensor category (see e.g.\ \cite[Cor.\,5.1.8]{Kerler:2001}). It carries the structure of a Hopf algebra
in $\Cc$ and is equipped with a Hopf pairing $\omega : \coend \ot \coend \to \one$ \cite{Lyubashenko:1995}. 
We denote the product, coproduct, unit, counit and antipode of the Hopf algebra $\coend$ by $\mu_\coend$, $\Delta_\coend$, $\eta_\coend$, $\eps_\coend$, $S_\coend$, respectively	(we use the conventions in \cite[Sec.\,3]{FGRprep1}).

If $\Cc$ is unimodular,
the coend $\coend$ admits a 
	two-sided
integral $\Lambda_\coend : \one \to \coend$
\cite[Thm.\,6.8]{Shimizu:2014}.\footnote{
	For factorisable $\Cc$ (see Definition~\ref{def:factorisable} below), the existence of two-sided integrals has been shown in \cite[Thm.\,6.11]{Lyubashenko:1995}, see also \cite[Cor.\,5.2.11]{Kerler:2001}.}
The integral is unique up to a scalar factor. 
Dually, $\coend$ admits a
	two-sided
 cointegral
  $\coint_\coend : \coend \to \one$ which can be normalised such that $\coint_\coend \circ \Lambda_\coend = \id_\one$ (see e.g.\ \cite[Sec.\,4.2.3]{Kerler:2001}).

We will need the linear maps 
\be
\xymatrix{
\End(\Id_\Cc)
 \ar[r]^{\psi}
 &
\Cc(\coend,\one)
 \ar@/^1pc/[r]^{\Rad}
 &
\Cc(\one,\coend)
\ar@/^1pc/[l]_\Omega
}
\quad ,
\ee
which are defined as follows.
\begin{itemize}
\item
For $\alpha \in \End(\Id_\Cc)$, the value $\psi(\alpha)$ is defined uniquely by the universal property of the coend
	$(\coend,\iota)$
via,  for all $X \in \Cc$,
\be
	\psi(\alpha) \circ \iota_X = \ev_X \circ (\id \ot \alpha_X) \ .
\ee
\item
$\Omega$ is defined via the Hopf pairing of $\coend$ as, for $f : \one \to \coend$,
\be\label{eq:Omega-def}
	\Omega(f) = \big[\, \coend \xrightarrow{\sim} \one \coend \xrightarrow{f \otimes \id} \coend\coend \xrightarrow{\omega} \one \,\big]	 \ . 
\ee
\item
The definition of $\rho$ requires $\Cc$ to be unimodular and the choice of a non-zero integral $\Lambda_\coend : 1 \to \coend$. In this case we set, for $g : \coend \to \one$,
\be\label{eq:rho-def}
	\rho(g) = \big[\, \one \xrightarrow{\Lambda_\coend} \coend \xrightarrow{\Delta_\coend} \coend\coend \xrightarrow{g \otimes \id} \one \coend \xrightarrow{\sim} \coend \,\big]	 \ . 
\ee
\end{itemize}

The bialgebra structure on
	$\coend$
allows one to endow $\Cc(\one,\coend)$ and $\Cc(\coend,\one)$ with the structure of a $k$-algebra. $\End(\Id_\Cc)$ is equally a $k$-algebra. We have:

\begin{lemma}\label{lem:Omega-rho-properties}
\begin{enumerate}
\item $\psi$ is an isomorphism of $k$-algebras.
\item $\Omega$ is a $k$-algebra homomorphism (but not necessarily an isomorphism).
\item For $\Cc$ unimodular,
$\Rad$ is an isomorphism (but not necessarily a $k$-algebra homomorphism).
\end{enumerate}
\end{lemma}

For the proof of this lemma, see e.g.\ \cite{Lyubashenko:1995}, 
\cite[Prop.\,5.2.5]{Kerler:2001},
\cite[Cor.\,4.2.13]{Kerler:2001}
and
\cite[Sec.\,3.1]{Shimizu:2016}
(and also \cite[Sec.\,2]{Gainutdinov:2016qhz}).
One verifies that the inverses of $\psi$ and $\rho$ are, for $X \in \Cc$, $f : \coend \to \one$, $g : \one \to \coend$,
\begin{align}
	\psi^{-1}(f)_X &~=~ \big[ X \xrightarrow{\sim} \one X
	\xrightarrow{\coev_X \ot \id} (X  X^*)  X
	\xrightarrow{\sim} X  (X^*  X)
	\xrightarrow{\id \ot \iota_X} X \coend
	\xrightarrow{\id \ot f} X \one
	\xrightarrow{\sim} X\big] \ ,
	\nonumber\\
\Rad^{-1}(g) &~=~ \big[\, \coend \xrightarrow{\sim} \one \coend \xrightarrow{g \ot \id} \coend\coend
\xrightarrow{S_\coend \ot \id} \coend\coend \xrightarrow{\mu_\coend} \coend \xrightarrow{\coint_\coend} \one \,\big] \ .
\end{align}

The {\em internal character}	of $V \in \Cc$ is the element $\chi_V \in \Cc(\one,\coend)$ given by 
\cite{Fuchs:2010mw,Shimizu:2015} 
(we use the convention in \cite{Gainutdinov:2016qhz})
\be\label{eq:chiV-def}
	\chi_V ~=~ 
	\big[\, \one \xrightarrow{\widetilde\coev_V} V^* \ot V   \xrightarrow{\iota_{V}} \coend \,\big] \ .
\ee
By Lemma~\ref{lem:dinat-Gr}, this map factors through the Grothendieck ring, i.e.\ we obtain a map 
\be
\chi :\; \Gr(\Cc) \to \Cc(\one,\coend), \qquad [V] \mapsto \chi_V\ .
\ee
 By abuse of notation, we will denote the induced map $\Gr_k(\Cc) \to \Cc(\one,\coend)$, 
$1 \otimes_\Zb [V] \mapsto \chi_V$ by $\chi$ as well.

\begin{theorem}[{\cite[Sec.\,4.5]{Fuchs:2010mw} and \cite[Thm.\,3.11, Prop.\,3.14, Cor.\,4.2, Thm.\,5.12]{Shimizu:2015}}]~
\label{thm:chiinj}
\begin{enumerate}
\item
The linear map $\chi : \Gr_k(\Cc) \to \Cc(\one,\coend)$ is an injective $k$-algebra homomorphism. 
\item
Suppose in addition that $\Cc$ is unimodular. Then $\chi$ is surjective iff $\Cc$ is semisimple.
\end{enumerate}
\end{theorem}

This theorem
	remains true for $\Cc$ not braided
(but still a pivotal finite tensor category over an algebraically closed field).

Define, 	for
	$\Cc$ unimodular and 
$M \in \Cc$,
\be\label{eq:tildechi-def}
\phi_M ~:=~ \psi^{-1}\big(\Rad^{-1}(\chi_M)\big) ~\in~ \End(\Id_\Cc)\ .
\ee
After substituting the definitions, 	one can check that, for all $X \in \Cc$,
\be\label{eq:phi_M-explicit}
	(\phi_M)_X ~= 
	  \raisebox{-0.5\height}{\setlength{\unitlength}{.75pt}
  \begin{picture}(100,231)
   \put(0,0){\scalebox{.75}{\includegraphics{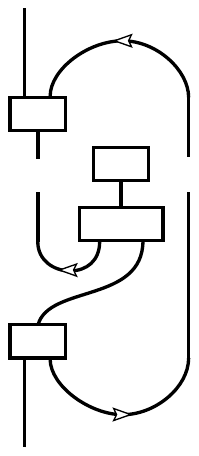}}}
   \put(0,0){
     \put(0,0){
     \put (8,0) {\scriptsize$ X $}
     \put (8,226) {\scriptsize$ X $}
     \put (14,58) {\scriptsize$ \id $}
     \put (14,167) {\scriptsize$ \id $}
     \put (42,116) {\scriptsize$ \iota_{X\otimes M^*} $}
     \put (52,145) {\scriptsize$ \Lambda^{\mathrm{co}}_\coend $}
     \put (0,139) {\scriptsize$ X{\otimes}M^* $}
     \put (86,139) {\scriptsize$ M $}
     }\setlength{\unitlength}{1pt}}
  \end{picture}}
 \quad .
\ee
As $\chi_M$ only depends on $[M] \in \Gr(\Cc)$, so does $\phi_M$.
Since $\psi$ and $\Rad$ are isomorphisms, by Theorem~\ref{thm:chiinj} the set $\{\, \phi_U \,|\, U \in \Irr(\Cc)\,\}$ is linearly independent in $\End(\Id_\Cc)$.
We note that since $\rho$ is not necessarily an algebra map, neither is the linear map $\Gr_k(\Cc) \to \End(\Id_\Cc)$, $[M] \mapsto \phi_M$.

We define the linear map $\mathscr{S}_\Cc : \End(\Id_\Cc) \to \End(\Id_\Cc)$, the {\em modular $S$-transformation}, as
\be\label{eq:SC-def}
\mathscr{S}_\Cc ~=~ \big[\, \End(\Id_\Cc) \xrightarrow{\,\psi\,} \Cc(\coend,\one) \xrightarrow{~\Rad~} \Cc(\one,\coend) \xrightarrow{~\Omega~} \Cc(\coend,\one) \xrightarrow{~\psi^{-1}~} \End(\Id_\Cc) \, \big] \ .
\ee
We have seen that $\psi$ and $\Omega$ are algebra maps, while $\Rad$ is in general not.
Thus in general $\mathscr{S}_\Cc$ is not an algebra map, either.
Recall the definition of $\smap$ in \eqref{eq:sigma-Gr-EndId-def}.
A straightforward calculation shows that $\smap([M]) = \psi^{-1}\circ \Omega (\chi_M)$. Combining this with \eqref{eq:tildechi-def} and \eqref{eq:SC-def} gives, for all $M \in \Cc$,
\be\label{eq:smap-via-SC-phiM}
	\smap([M]) ~=~ \mathscr{S}_\Cc(\phi_M) \ ,
\ee
cf.\ \cite[Rem.\,3.10\,(2)]{Gainutdinov:2016qhz}.
 In particular, by Lemma~\ref{lem:sigma_Gr_ringhom}, the combination $[M] \mapsto \mathscr{S}_\Cc(\phi_M)$ is an algebra homomorphism from $\Gr_k(\Cc)$ to $\End(\Id_\Cc)$.

\begin{remark}
Since the coend $\coend$ is unique up to unique isomorphism, and the cointegral $\coint_\coend$ is unique up to sign, 
it is straightforward to verify that
the $\phi_M$, $M \in \Cc$ and $\mathscr{S}_\Cc$ depend on the choice of 
$(\coend,\coint_\coend)$
only up to an overall sign, see \cite[Prop.\,5.3]{FGRprep1}. In particular, $\mathscr{S}_\Cc(\phi_M)$ is independent of the choice of 
$(\coend,\coint_\coend)$, as we already saw explicitly in \eqref{eq:smap-via-SC-phiM}.
\end{remark}

Next we introduce 
	four equivalent 
non-degeneracy requirement for the braiding. (This does not require a pivotal structure, hence we use the letter $\mathcal{D}$ and reserve $\Cc$ for the pivotal case as declared above.)

\begin{theorem}[{\cite[Thm.\,1.1]{Shimizu:2016}}]\label{thm:factequiv}
Let $\mathcal{D}$ be a braided finite tensor category over an algebraically closed field.
The following conditions are equivalent:
\begin{enumerate}
\item Every transparent object in $\mathcal{D}$ is isomorphic to a direct sum of tensor units. ($T \in\mathcal{D}$ is {\em transparent} if for all $X \in \mathcal{D}$, $c_{X,T} \circ c_{T,X} = \id_{T\ot X}$.)
\item The canonical braided monoidal functor 
$\mathcal{D} \boxtimes \overline{\mathcal{D}} \to \mathcal{Z}(\mathcal{D})$ 
is an equivalence. (Here, $\boxtimes$ is the Deligne product, $\overline{\mathcal{D}}$ is the same tensor category as $\mathcal{D}$, but has inverse braiding, 
and $\mathcal{Z}(\mathcal{D})$ is the Drinfeld centre of $\mathcal{D}$.)
\item The pairing $\omega : \coend \ot \coend \to \one$ is non-degenerate (in the sense that there exists a copairing $\one \to \coend \ot \coend$).
\item $\Omega$ is an isomorphism.
\end{enumerate}
\end{theorem}

\begin{definition}\label{def:factorisable}
$\mathcal{D}$ as in Theorem~\ref{thm:factequiv} is called {\em factorisable} if it satisfies the equivalent conditions 
1--4 there.
\end{definition}

	In \cite[Prop.\,4.5]{Etingof:2004} the following result is proved (using formulation 2 of factorisability in Theorem~\ref{thm:factequiv}):

\begin{theorem} 
Let $\Cc$ be factorisable. Then $\Cc$ is unimodular.
\end{theorem}

In particular, for factorisable $\Cc$ the Hopf algebra given by the coend $\coend$ admits integrals and cointegrals.
Moreover, in this case we can normalise the integral such that
\be
	\omega \circ (\Lambda_\coend \ot \Lambda_\coend) = \id_\one \ ,
\ee
see e.g.\ \cite[Sec.\,5.2.3]{Kerler:2001}. Since the space of integrals is one-dimensional, this determines $\Lambda_\coend$ up to a sign.

\begin{remark}\label{rem:SC-iso-iff-fact}
Suppose $\Cc$ is factorisable. Then by the above theorem, $\Cc$ is unimodular and we have the isomorphism $\rho$ from \eqref{eq:rho-def} at our disposal (Lemma~\ref{lem:Omega-rho-properties}). Hence the linear endomorphism $\mathscr{S}_\Cc$ of $\End(\Id_\Cc)$ in \eqref{eq:SC-def} is defined, and 
by condition 4 in Theorem~\ref{thm:factequiv}, $\mathscr{S}_\Cc$ is invertible.
\end{remark}

Using Lemma~\ref{lem:sigma_Gr_ringhom} and \eqref{eq:tildechi-def}, \eqref{eq:smap-via-SC-phiM}, we get the following corollary to Theorem~\ref{thm:chiinj}.

\begin{corollary}\label{cor:inj-alg-hom}
Let $\Cc$ be factorisable. Then the algebra map 
$\sigma \colon \Gr_k(\Cc) \to \End(\Id_\Cc)$
 is injective.
If in addition $\mathrm{char}(k)=0$, the ring homomorphism
 $\smap \colon \Gr(\Cc) \to \End(\Id_\Cc)$ is injective,
too.
\end{corollary}

The following important theorem is proved in \cite{Lyubashenko:1995},
see also 
	\cite[Sec.\,5.1]{FGRprep1} 
for a summary in the notation used here.

\begin{theorem}\label{thm:proj-sl2Z-EndId}
Let $\Cc$ be factorisable and ribbon with ribbon twist
$\theta\in\End(\Id_\Cc)$. Then there is 
a projective representation of $SL(2,\Zb)$ on $\End(\Id_\Cc)$ 
	for which the $S$ and $T$ generators act as
\be
\big(
\begin{smallmatrix}
0 & -1 \\
1 & 0
\end{smallmatrix}
\big)
~\longmapsto~ \mathscr{S}_\Cc
\quad , \quad
\big(
\begin{smallmatrix}
1 & 1 \\
0 & 1
\end{smallmatrix}
\big)
~\longmapsto~ \theta \circ (-)
\quad .
\ee
\end{theorem}

\begin{remark}
For $\Cc=\Rep H$ for $H$ a 
	finite-dimensional
ribbon Hopf algebra over $k$, we have $\one=k$, and the coend is given by $\coend = H^*$ with the coadjoint action \cite{Lyubashenko:1994tm,Kerler:1996}. 
Let $S$ be the antipode of $H$, $R$ the universal R-matrix, $v$ the ribbon element, and
   $u=\sum_{(R)}S(R_2)R_1$ the Drinfeld element.
We will
	(for the sake of this remark)
identify $\End(\Id_\Cc)$ with the centre $Z(H)$ of $H$.
Then we have:
\begin{itemize}
\item
The internal characters $\chi_V$ are the  $q$-characters:
the images $\chi_V(1)$
are linear forms on~$H$ 
invariant under the coadjoint action. They are the trace functions $\chi_V(1)=\mathrm{Tr}_V(u^{-1}v\,\cdot\,)$
introduced and studied in~\cite{Drinfeld},
	see also \cite{[Kerler]}.
\item
The map 
$S\circ
\psi^{-1}\circ\,\Omega$ is then \textit{the Drinfeld mapping} from the space of $q$-characters to $Z(H)$ given by 
$\chi(\cdot)\mapsto \sum_{(M)} \chi(M_1)M_2$ 
for the monodromy matrix $M=R_{21}R$,
see~\cite[Prop.\,1.2]{Drinfeld}. 
The central elements $\smap([V])_H$
(where $H$ is the left regular $H$-module)
are the images of $\mathrm{Tr}_V(u^{-1}v\,\cdot\,)$ under the Drinfeld mapping composed with~$S^{-1}$.

\item The map 
$S\circ
\psi^{-1}\circ\rho^{-1}$ is \textit{the Radford mapping} from the space of $q$-characters to $Z(H)$ 
given by  $\chi(\cdot)\mapsto (\chi\otimes\id)(\Delta(\boldsymbol{c}))$, for the (co)integral $\boldsymbol{c}\in H$ (that can be computed from $\coint_{\coend}$ using the duality maps),
see~\cite{R1} for properties of this map. 
	The
central elements $(\phi_V)_H$ from~\eqref{eq:phi_M-explicit}  are the images of $\mathrm{Tr}_V(u^{-1}v\,\cdot\,)$ under the Radford mapping 
	composed with~$S^{-1}$,
	see also \cite[Rem.\,7.10]{FGRprep1}.

\item 
The modular $S$-transformation on a quasi-triangular Hopf algebra was introduced in~\cite{Lyubashenko:1994ab}
following the categorical construction of~\cite{Lyubashenko:1995}. 
	The images of $\mathrm{Tr}_V(u^{-1}v\,\cdot\,)$ under the Drinfeld and Radford mappings are related by the modular 
$S$-trans\-for\-ma\-ti\-on~\cite[Sec.\,2]{[Kerler]}, see~\cite[Sec.\,5]{[FGST]} for the statement in this form. This result is generalised by~\eqref{eq:smap-via-SC-phiM}.
\end{itemize}
\end{remark}

\section{Existence of a simple projective object}\label{sec:projsimpleex}

Throughout the rest of this paper, the following technical condition will play an important role.
\begin{quote}
{\bf Condition P}:
A finite tensor category $\mathcal{M}$ over $k$ satisfies 
{\em Condition~P} 
if there exists a projective object $P \in \mathcal{M}$ such that $[P]$ is not nilpotent in the linearised Grothendieck ring $\Gr_k(\mathcal{M})$.
\end{quote}
Two important classes of categories which satisfy Condition~P are described in the next two lemmas.

\begin{lemma}\label{lem:ssi-sat-condP}
Every semisimple finite tensor category $\mathcal{M}$ satisfies Condition~P.
\end{lemma}

\begin{proof}
Since $\mathcal{M}$ is semisimple, every simple object is projective. Furthermore, for every simple object $U \in \mathcal{M}$, $[U]$ is non-zero in $\Gr_k(\mathcal{M})$. Hence we may take $P = \one$, the tensor unit. Then for all $m>0$, $[\one]^m = [\one] \neq 0$ in $\Gr_k(\mathcal{M})$.
\end{proof}

\begin{lemma}\label{lem:powers-nonzero-in-Gr}
Let $\Cc$ be a locally finite braided tensor category over some field. For all $X \in \Cc$, $X \neq 0$ and $m>0$ we have $[X]^m \neq 0$ in $\Gr(\Cc)$.
If in addition $\Cc$ is finite and if the field is of characteristic zero, then $\Cc$ satisfies Condition~P.
\end{lemma}

\begin{proof}
Since by assumption all objects have finite length, we have
$\Gr(\Cc) = \bigoplus_{U \in \Irr(\Cc)} \Zb [U]$. Thus for $X \in \Cc$ it follows that  $X=0$ iff $[X]=0$. 

The map $\ev_X : X^* \otimes X \to \one$ is non-zero, hence $X^* \otimes X \neq 0$, and so $[X^*][X] \neq 0$. Iterating this argument shows $(X^* \otimes X)^* \otimes (X^* \otimes X) \neq 0$,
 i.e.\ $([X^*][X])^2 \neq 0$, etc. 
 We used here Drinfeld's canonical isomorphism $u_X: X \to X^{**}$ given by 
\be
u_X ~=
  \raisebox{-0.5\height}{\setlength{\unitlength}{.75pt}
  \begin{picture}(84,130)
   \put(0,0){\scalebox{.75}{\includegraphics{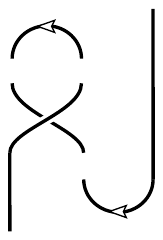}}}
     \put(0,0){
     \put (2,0) {\scriptsize$ X $}
     \put (72,122) {\scriptsize$ X^{**} $}
     \put (2,85) {\scriptsize$ X^* $}
     \put (39,39) {\scriptsize$ X^* $}
     \put (37,85) {\scriptsize$ X $}
     }\setlength{\unitlength}{1pt}
  \end{picture}}
  \quad .
\ee
Using commutativity of $\Gr(\Cc)$, after $m$ steps we find $[X^*]^m [X]^m \neq 0$.

As a finite
$\Cc$ has enough projectives, it contains a non-zero projective object $P$. Since in characteristic zero, the canonical ring homomorphism $\Gr(\Cc) \to \Gr_k(\Cc)$ is injective, $[P]^m \neq 0$ for all $m>0$ also holds in $\Gr_k(\Cc)$.
\end{proof}

\begin{remark}\label{rem:condP-fin-char}
In positive characteristic, Condition~P may or may not be satisfied. For example, suppose that $k$ is of characteristic $p$ and let $G$ be a finite group. If $p$ does not divide $|G|$, the category $k[G]$-mod of finite-dimensional $k[G]$-modules is semisimple and hence satisfies Condition~P by Lemma \ref{lem:ssi-sat-condP}.

On the other hand, if $G$ is a $p$-group the trivial $k[G]$-module is up to isomorphism the unique simple $k[G]$-module (see e.g.\ the corollary to Proposition 26 in Section 8.3 of \cite{Serre-book}). 
The projective cover of the trivial $k[G]$-module is $P := k[G]$ (see \cite[Sec.\,15.6]{Serre-book}).
The image of $P$ in $\Gr_k$ is $[P] = |G| \cdot [k] = 0$, as $p$ divides $|G|$. By the previous observations, every projective $k[G]$-module is isomorphic to a direct sum of $P$'s, and so $k[G]$-mod does not satisfy Condition~P.
\end{remark}

We can now state our first main result:

\begin{theorem}\label{thm:main-simpleproj}
Let $\Cc$ be a factorisable and pivotal finite tensor category over $k$.
If $\Cc$ satisfies Condition~P, it contains a simple projective object.
\end{theorem}

The proof relies on a series of lemmas and will be given at the end of this section. The idea is very simple: we show that in the absence of simple projectives, the injective algebra homomorphism $\smap$ from Corollary~\ref{cor:inj-alg-hom} would map the class $[P]$ of a projective object to a nilpotent natural endomorphism, 
	which is a contradiction to Condition~P.

\begin{remark}
~\\[-1.5em]
\begin{enumerate}\setlength{\leftskip}{-1em}
\item
Theorem~\ref{thm:main-simpleproj} is a generalisation of a result by Cohen and Westreich \cite[Cor.\,3.6]{Cohen:2008}. There, the authors show that a factorisable ribbon Hopf algebra 
	$H$
over an algebraically closed field of characteristic zero has an irreducible projective module. The proof strategy in \cite{Cohen:2008} is different from ours: in $\cite{Cohen:2008}$ it is observed that Hopf algebras with the above properties are also symmetric algebras (see Section~\ref{sec:traces} for the definition) and that fact is used to show existence of a simple projective
module.
 Our approach is in some sense opposite: we first show that there exists a simple projective module 
and then in Section~\ref{sec:traces} use the theory of
 modified traces of~\cite{Geer:2010}
 to deduce the existence of certain symmetric algebras.
In this sense, applying Theorem~\ref{thm:main-simpleproj} to $\Cc = \Rep(H)$ gives an alternative proof of the result in \cite{Cohen:2008}
	(in addition to having the benefit of not being restricted to characteristic zero).
\item
In characteristic $p$, a category $\Cc$ as in Theorem~\ref{thm:main-simpleproj} may or may not satisfy Condition~P.
To see this, we use the observations in Remark~\ref{rem:condP-fin-char}.
Let $A$ be a finite abelian group.
The Drinfeld double $D(k[A])$ is a factorisable Hopf algebra, and hence $D(k[A])$-mod is a factorisable and pivotal finite tensor category over $k$.
But as algebras, 
$D(k[A]) \cong k(A) \otimes_k k[A]$, and $k(A) = k^{|A|}$ is semisimple.

Suppose that $\mathrm{char}(k)=p$ does not divide $|A|$. Then $D(k[A])$-mod is semisimple and hence satisfies Condition~P (Lemma~\ref{lem:ssi-sat-condP}).

On the other hand, if $A$ is a $p$-group, the indecomposable projective modules in $D(k[A])$-mod are isomorphic to a simple $k(A)$-module tensored with $k[A]$. As in Remark~\ref{rem:condP-fin-char}, each of these have zero image in $\Gr_k(\Cc)$.

\item
The example in part 2 (combined with Remark~\ref{rem:condP-fin-char}) also shows that for $\mathrm{char}(k)=p$ and $A$ a finite abelian $p$-group, $D(k[A])$-mod does not contain a simple projective object. Thus we cannot drop Condition~P from Theorem~\ref{thm:main-simpleproj}.
\end{enumerate}
\end{remark}

We now turn to the proof of Theorem~\ref{thm:main-simpleproj}. We first gather the tools to see that $\sigma([P])$ is nilpotent for projective $P$ unless there is a simple projective object.

\begin{lemma}\label{lem:nat-endo-id_nilpot}
Let $\Ac$ be an abelian category and let $\eta$ be a natural 
	endomorphism
of the identity functor. If $\eta$ is zero on all simple objects, then it is nilpotent on all objects of finite length.
\end{lemma}

\begin{proof}
We use induction on the length of an object. Given an object $B$ of finite length $>1$, choose a short exact sequence $0 \to A \xrightarrow{f} B \xrightarrow{g} C \to 0$ such that $A,C$ have smaller length. Then
\be
	\xymatrix{
	A \ar[r]^{f} \ar[d]^{\eta_A} & B  \ar[r]^{g} \ar[d]^{\eta_B} & C  \ar[d]^{\eta_C}
	\\
	A  \ar[r]^{f} & B  \ar[r]^{g} & C
	}
\ee
commutes. By induction hypothesis, $\eta_A$ and $\eta_C$ are nilpotent. Hence there is $m>0$ such that
\be
	\xymatrix{
	A \ar[r]^{f} \ar[d]^{0} & B  \ar[r]^{g} \ar[d]^{(\eta_B)^m} & C  \ar[d]^{0}
	\\
	A  \ar[r]^{f} & B  \ar[r]^{g} & C
	}
\ee
commutes. Then $g \circ (\eta_B)^m = 0$ and $(\eta_B)^m \circ f = 0$, 
and since the image of $f$ is the kernel of $g$
we have the inclusions $\mathrm{im} (\eta_B)^{m}\subset\mathrm{im} f \subset \mathrm{ker} (\eta_B)^{m}$. This shows that
$(\eta_B)^{2m}=0$.
\end{proof}

\begin{corollary}\label{cor:nat-endo-id_nilpot}
Let $\Ac$ be a finite abelian category over some field and let $\eta \in \End(\Id_\Ac)$. If $\eta$ is zero on all simple objects, then $\eta$ is nilpotent, i.e.\ there is $m>0$ such that $\eta^m=0$.
\end{corollary}

\begin{proof}
By finiteness, $\Ac$ contains a projective generator $G$ (of finite length). By Lemma~\ref{lem:nat-endo-id_nilpot}, $(\eta_G)^m=0$ for some $m$. Since $\eta_G$ determines $\eta$, it follows that $\eta^m=0$.
\end{proof}

\begin{lemma}\label{lem:sigma-proj-zero-simple}
Let $\Cc$ be a pivotal finite braided tensor category over some field and let $P,U \in \Cc$ with $P$ projective and $U$ non-projective and simple. Then $\smap([P])_U=0$.
\end{lemma}

\begin{proof}
By \eqref{eq:eq:sigma-Gr-EndId-def_pic},
$\smap([P])_U$ is a composition of maps $U \rightarrow P^* \otimes (U \otimes P) \rightarrow U$. 
Since $P^* \otimes (U \otimes P)$ is projective,  
the map $P^* \otimes (U \otimes P) \rightarrow U$ factors as
$P^* \otimes (U \otimes P) \rightarrow P_U \rightarrow U$, for $P_U$ the projective cover of $U$. Thus $\smap([P])_U$ can be written as a composition $U \rightarrow P_U \rightarrow U$. By assumption $P_U \ncong U$, and so this composition is zero.
\end{proof}

Combining the above results, the proof of Theorem~\ref{thm:main-simpleproj} is now a straightforward consequence of Corollary~\ref{cor:inj-alg-hom}:

\begin{proof}[Proof of Theorem~\ref{thm:main-simpleproj}]\footnote{
	We thank Victor Ostrik for suggesting a simplification of our first version of the proof.}
Suppose that $\Cc$ does not contain a simple projective object. 
By Condition~P, we can find a projective object $P \in \Cc$ such that $[P]^m \neq 0$ in $\Gr_k(\Cc)$ for all $m>0$.

By Corollary~\ref{cor:inj-alg-hom}, $\smap : \Gr_k(\Cc) \to \End(\Id_\Cc)$ is an injective algebra homomorphism. 
In particular, since $[P]^m \neq 0$ we find that for each $m>0$,
$\smap([P])^m = \smap([P]^m) \neq 0$.
But by Lemma~\ref{lem:sigma-proj-zero-simple}, the natural transformation $\smap([P])$ is zero on all non-projective simples. By assumption, all simples in $\Cc$ are non-projective, and so Corollary~\ref{cor:nat-endo-id_nilpot} implies that $\smap([P])$ is nilpotent, which is a contradiction. 

Hence $\Cc$ must contain a simple projective object.
\end{proof}

\section{Non-degenerate trace on the projective ideal}\label{sec:traces}

In this section we will review some aspects of the theory of modified traces from \cite{Geer:2010,Geer:2011a,Geer:2011b}. In particular, we will see that the existence of a simple projective object
 established in Theorem~\ref{thm:main-simpleproj} guarantees the existence of a unique-up-to-scalars non-zero modified trace on the projective ideal, and that this trace induces non-degenerate pairings on Hom-spaces. 

\medskip

Let $\Cc$ be a pivotal finite tensor category over
$k$.
The full subcategory $\Proj(\Cc)$ of all projective objects in $\Cc$ is a tensor ideal. Following \cite{Geer:2011a}, a {\em modified right trace} on $\Proj(\Cc)$ is a family of $k$-linear functions
$\{ t_P : \End(P) \to k \,|\, P \in \Proj(\Cc) \}$ such that 
\begin{enumerate}
\item  {\em (partial trace)} For all $P \in Proj(\Cc)$, $X \in \Cc$ and $f \in \End(P \ot X)$ we have
\be\label{eq:mod-tr_partial-trace}
	t_{P \ot X}(f)
	~=~ t_P(\, tr_X^r(f) \,)
\ee
where $tr_X^r$ denotes the partial right trace 
\begin{align}
	tr_X^r(f) &~:=~ 
	\big[\,
	P 
	\xrightarrow{\sim} P \one
	\xrightarrow{\id \ot \coev_X} P (X X^*)
	\xrightarrow{\sim} (PX)X^*
	\xrightarrow{f \ot \id} (P X) X^*
	\nonumber\\
	&\hspace{4em} \xrightarrow{\sim} P(XX^*)
	\xrightarrow{\id \ot \widetilde\ev_X} P\one
	\xrightarrow{\sim} P
	\,\big] 
\nonumber\\
  & ~=~ \raisebox{-0.5\height}{\setlength{\unitlength}{.75pt}
  \begin{picture}(73,135)
   \put(0,0){\scalebox{.75}{\includegraphics{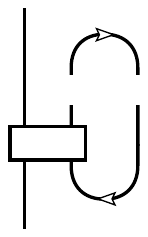}}}
     \put(0,0){
     \put (30,76) {\scriptsize$ X $}
     \put (62,76) {\scriptsize$ X^* $}
     \put (9,2) {\scriptsize$ P $}
     \put (9,123) {\scriptsize$ P $}
     \put (20,50) {\scriptsize$ f $}
     }
     \setlength{\unitlength}{1pt}
  \end{picture}}
  \qquad .
  \label{eq:partial-trace-def}
\end{align}
\item {\em (cyclicity)} For all $P,P' \in \Proj(\Cc)$, $f : P \to P'$, $g:  P' \to P$ we have
\be\label{eq:mod-tr_cyclic}
	t_{P'}(f \circ g) = t_P(g \circ f) \ .
\ee
\end{enumerate}

Given a modified right trace, one can define a family of pairings by setting, for all $M \in \Cc$, $P \in Proj(\Cc)$,
\be\label{eq:pairingMP}
	\Cc(M,P) \times \Cc(P,M) \to k
	\quad , \quad
	(f,g) \mapsto t_P(f \circ g) \ .
\ee
Note that the cyclicity property \eqref{eq:mod-tr_cyclic} only applies when also $M$ is projective.

\begin{remark}
There are also natural notions of left and left-right (or two-sided) modified traces for a pivotal category,
see~\cite{Geer:2011a} for details.
For $\Cc$ in addition ribbon, each modified right trace
on $\Proj(\Cc)$ is also a two-sided modified trace~\cite[Thm.\,3.3.1]{Geer:2010}.
In this case, one can construct  an isotopy invariant of
ribbon graphs in $\Rb^3$, where at least one strand is coloured by a projective object $P$ of $\Cc$, by cutting  the strand and computing $t_P$ of the corresponding endomorphism of $P$,
 see~\cite[Thm.\,3]{Geer:2009}. (There, the modified trace was used implicitly as the modified dimension for simple projectives.)
\end{remark}

In this paper we will only use modified right traces, and we will refer to these just as ``modified traces''.

\begin{proposition}
\label{prop:mtrace}
Let $\Cc$ be a unimodular pivotal finite tensor category over 
	$k$.
Suppose that $\Cc$ contains a simple projective object.
\begin{enumerate}
\item There exists a nonzero modified trace on $\Proj(\Cc)$, and this modified trace is unique up to scalar multiples. 
\item For any choice of non-zero modified trace $t$ on $\Proj(\Cc)$ and for all $M \in \Cc$, $P \in Proj(\Cc)$, the pairing \eqref{eq:pairingMP} is non-degenerate.
\end{enumerate}
\end{proposition}

The first part of the proposition is proved in \cite[Cor.\,3.2.1]{Geer:2011b}\footnote{
In \cite[Sect.\,3.2]{Geer:2011b} it is assumed that every indecomposable object of $\Cc$ is absolutely indecomposable (see \cite{Geer:2011b} for definitions). This condition does typically not hold (e.g.\ in the example in Section~\ref{sec:sf-example}). However, the arguments in \cite[Sect.\,3.2]{Geer:2011b} actually only require this condition on $\Proj(\Cc)$, where it holds for any finite tensor category over an algebraically closed field.}.
We will review the argument below as we hope it to be helpful to have the more general construction of \cite{Geer:2011a,Geer:2011b} specialised to the present setting, and since
we will need the explicit construction of the
modified trace below anyway. The second part of the proposition is proved in \cite[Prop.\,6.6]{Costantino:2014sma} for the Hopf algebra $\overline U^H_q sl(2)$, but the same proof works in general as we also review below. 

Before turning to the proof of Proposition~\ref{prop:mtrace}, we need some preparation. For the rest of this section, we fix:
\begin{itemize}
\item $\Cc$: a unimodular pivotal finite tensor category over $k$,
\item $Q$: a simple projective object in $\Cc$
	(in particular we assume that such a $Q$ exists).
\end{itemize}

\begin{lemma}[{\cite[Thm.\,3.1.3\,\&\,Sec.\,3.2]{Geer:2011b}}]\label{lem:simple-proj-ambi}
For all $g \in \End(Q \otimes Q^*)$ we have
\be\label{eq:simple-proj-ambi}
  \raisebox{-0.5\height}{\setlength{\unitlength}{.75pt}
  \begin{picture}(76,157)
   \put(0,0){\scalebox{.75}{\includegraphics{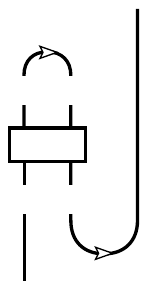}}}
     \put(0,0){
     	\put(-24,0){
     		\put (34,52) {\scriptsize$ Q $}
     		\put (57,52) {\scriptsize$ Q^* $}
     		\put (34,103.5) {\scriptsize$ Q $}
     		\put (57,103.5) {\scriptsize$ Q^* $}
     		\put (48,78) {\scriptsize$ g $}
     		}
     \put (65,150) {\scriptsize$ Q $}
     \put (11,5) {\scriptsize$ Q $}
     }\setlength{\unitlength}{1pt}
  \end{picture}}
	~~=
  \raisebox{-0.5\height}{\setlength{\unitlength}{.75pt}
  \begin{picture}(76,157)
   \put(0,0){\scalebox{.75}{\includegraphics{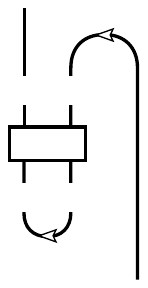}}}
     \put(0,0){
     	\put(-24,0){
     		\put (34,52) {\scriptsize$ Q $}
     		\put (57,52) {\scriptsize$ Q^* $}
     		\put (34,103.5) {\scriptsize$ Q $}
     		\put (57,103.5) {\scriptsize$ Q^* $}
     		\put (48,78) {\scriptsize$ g $}
     		}
     \put (65,5) {\scriptsize$ Q $}
     \put (11,150) {\scriptsize$ Q $}
     }\setlength{\unitlength}{1pt}
  \end{picture}}
  \qquad .
\ee
\end{lemma}

\begin{proof}
Write $Q \otimes Q^* = \bigoplus_{U \in \Irr(\Cc)} P_U^{\oplus m_U}$ and let $i_U : P_U^{\oplus m_U} \to Q \otimes Q^*$ and $p_U : Q \otimes Q^* \to P_U^{\oplus m_U}$ be the corresponding embedding and projection maps. The resulting idempotent will be denoted by $e_U = i_U \circ p_U \in \End(Q \otimes Q^*)$.
Since $\Cc(Q \otimes Q^*,\one)\cong \Cc(Q,Q)$
is one-dimensional ($Q$ being simple and 
$k$ algebraically closed), 
we have $m_\one = 1$.

Since $P_\one$ is indecomposable, there is an $\alpha \in k$ and a nilpotent element $n \in \End(P_\one)$ such that (using again algebraically closedness of $k$)\footnote{
	Since $P_\one$ is indecomposable, by Fitting's Lemma an element $x \in \End(P_\one)$ is either nilpotent or invertible. As $k$ is algebraically closed, 
$x - \alpha \cdot \id$  is not invertible for some  $\alpha \in k$  (as the $k$-linear map $x \circ (-)$ on $\End(P_\one)$ has an eigenvector), hence it is nilpotent, and therefore $x = \alpha \cdot \id + n$, as we state.}
\be\label{eq:simple-proj-ambi-aux1}
	p_\one \circ g \circ i_\one
	~=~ \alpha \cdot \id_{P_\one} + n \ .
\ee

For all $u : \one \to P_\one$ and $v : P_\one \to \one$ we have 
\be\label{eq:simple-proj-ambi-aux2}
n \circ u = 0 \qquad  \text{and} \qquad v \circ n=0 \ .
\ee
To see $n\circ u=0$, first note that by unimodularity (i.e.\ since $(P_\one)^* \cong P_\one$) $\Cc(\one,P_\one)$ is one-dimensional. Suppose $u \neq 0$ and 
hence $u$ forms a basis of $\Cc(\one,P_\one)$. Then there is $\lambda \in k$ such that 
$n \circ u = \lambda \cdot u$.
 But $n$ is nilpotent, say $n^m=0$,
and so applying $n$ on both sides of the latter equality
	$m$
times we get $\lambda^{m}=0$ and therefore 
 $\lambda=0$.  For $v$ the argument is analogous. 

Next insert the identity $g = \sum_{U,V} e_U \circ g \circ e_V$ into \eqref{eq:simple-proj-ambi}. Since for $U \ncong \one$, $\Cc(P_U,\one)=0$ and $\Cc(\one,P_U)=0$ (using unimodularity),
 only the term with $U=V=\one$
will contribute on either side of \eqref{eq:simple-proj-ambi}.

The claim of the lemma now follows from the observation that
taking $u=p_\one\circ \coev_Q$ and $v=\widetilde\ev_Q \circ i_\one$ we have
 by \eqref{eq:simple-proj-ambi-aux1} and \eqref{eq:simple-proj-ambi-aux2}
\be
	\widetilde\ev_Q \circ e_\one \circ g \circ e_\one
	= \alpha \cdot \widetilde\ev_Q
	\quad , \quad
	e_\one \circ g \circ e_\one \circ \coev_Q
	= \alpha \cdot \coev_Q \ ,
\ee
where we used that $\widetilde\ev_Q \circ e_\one = \widetilde\ev_Q$ and $e_\one \circ \coev_Q = \coev_Q$ (because $(P_\one)^*\cong P_\one$). 
	Both  the left and right hand sides of~\eqref{eq:simple-proj-ambi} are therefore equal to $\alpha\,\id_Q$.
\end{proof}

The above lemma will be used in the following form: for all $f \in \End(Q^* \otimes Q)$,
\be\label{eq:simple-proj-leftright}
  \raisebox{-0.5\height}{\setlength{\unitlength}{.75pt}
  \begin{picture}(77,159)
   \put(0,0){\scalebox{.75}{\includegraphics{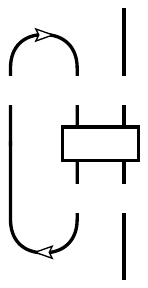}}}
     \put(0,0){
     	\put(0,0){
     		\put (34,51) {\scriptsize$ Q^* $}
     		\put (57,51) {\scriptsize$ Q $}
     		\put (34,104) {\scriptsize$ Q^* $}
     		\put (57,104) {\scriptsize$ Q $}
     		\put (48,78) {\scriptsize$ f $}
     		}
   	 \put (3,104) {\scriptsize$ Q $}
     \put (58,5) {\scriptsize$ Q $}
     \put (58,150) {\scriptsize$ Q $}
     }\setlength{\unitlength}{1pt}
  \end{picture}}
	~~=~
  \raisebox{-0.5\height}{\setlength{\unitlength}{.75pt}
  \begin{picture}(120,178)
   \put(0,0){\scalebox{.75}{\includegraphics{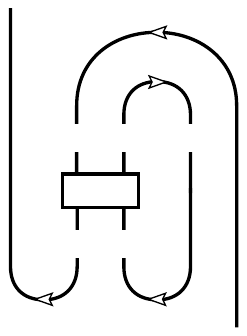}}}
     \put(0,0){
     	\put(0,0){
     		\put (34,52) {\scriptsize$ Q^* $}
     		\put (57,52) {\scriptsize$ Q $}
     		\put (34,104) {\scriptsize$ Q^* $}
     		\put (57,104) {\scriptsize$ Q $}
     		\put (48,78) {\scriptsize$ f $}
     		}
     \put (111,5) {\scriptsize$ Q $}
     \put (3,172) {\scriptsize$ Q $}
     \put (88,104) {\scriptsize$ Q^* $}
     }\setlength{\unitlength}{1pt}
  \end{picture}}
  \qquad .
\ee
This is a consequence of \eqref{eq:simple-proj-ambi} upon substituting
\be
	g ~= 
	  \raisebox{-0.5\height}{\setlength{\unitlength}{.75pt}
  \begin{picture}(99,154)
   \put(0,0){\scalebox{.75}{\includegraphics{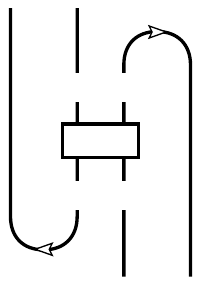}}}
     \put(0,0){
     	\put(0,0){
     		\put (34,52) {\scriptsize$ Q^* $}
     		\put (57,52) {\scriptsize$ Q $}
     		\put (34,103.5) {\scriptsize$ Q^* $}
     		\put (57,103.5) {\scriptsize$ Q $}
     		\put (48,78) {\scriptsize$ f $}
     		}
     \put (57,5) {\scriptsize$ Q $}

     \put (89,5) {\scriptsize$ Q^* $}
     \put (3,149) {\scriptsize$ Q $}
     \put (35,149) {\scriptsize$ Q^* $}
     }\setlength{\unitlength}{1pt}
  \end{picture}}
	\qquad .
\ee

Let us fix a non-zero linear function
for the given simple projective $Q$,
\be\label{eq:tQ-choice}
	t_Q ~:~ \End(Q) \to k \ .
\ee
Since $\End(Q) = k \, \id_Q$, such a function is unique up to a constant, and it is uniquely determined by its value $t_Q( \id_Q ) \in k^\times$.

Let $P \in \Proj(\Cc)$ and choose an object $X \in \Cc$ such that there is a surjection $p : Q \otimes X \to P$. Such an $X$ always exists, for example $X = Q^* \otimes P$. As $P$ is projective, there is a morphism $i : P \to Q \otimes X$ such that $p \circ i = \id_P$. That is, $i,p$ realise $P$ as a direct summand of $Q \otimes X$. Consider the function
\be\label{eq:tP-defn}
	t_P : \End(P) \to k
	\quad , \quad
	h \mapsto t_Q\big(\, tr^r_X(i \circ h \circ p) \,\big) \ .
\ee
The next lemma shows 
	in particular
that this notation is indeed consistent with \eqref{eq:tQ-choice} when setting $P=Q$ (in the lemma, take $X=\one$ and $i,p$ (inverse) unit isomorphisms).

\begin{lemma}\label{lem:tr-indep-choice}
The function $t_P$ in \eqref{eq:tP-defn} is independent of the choice of $X$ and of $p : Q \otimes X \to P$, $i : P \to Q \otimes X$.
\end{lemma}

\begin{proof}
This proof is taken from {\cite[Sec.\,4.5]{Geer:2011a}}. The argument is more general than needed for the statement of the lemma, but we can reuse it later to show cyclicity.

Let $P' \in \Proj(\Cc)$ and pick $X'$, $p' : Q \otimes X' \to P'$, 
$i' : P' \to Q \otimes X'$. For $u : P \to P'$ and $v : P' \to P$
	arbitrary,
consider the morphism $f \in \End(Q^* \otimes Q)$ given by
\be
	f 
	~=~ 
  \raisebox{-0.5\height}{\setlength{\unitlength}{.75pt}
  \begin{picture}(102,383)
   \put(0,0){\scalebox{.75}{\includegraphics{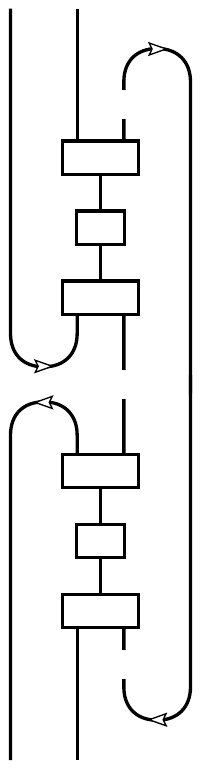}}}
     \put(0,0){
     \put (3,2) {\scriptsize$ Q^* $}
     \put (35,2) {\scriptsize$ Q $}
     \put (3,378) {\scriptsize$ Q^* $}
     \put (35,378) {\scriptsize$ Q $}
     \put (57,55) {\scriptsize$ X' $}
     \put (57,190) {\scriptsize$ X $}
     \put (57,324) {\scriptsize$ X' $}
     \put (52.5,97) {\scriptsize$ P' $}
     \put (52.5,130) {\scriptsize$ P $}
     \put (52.5,248) {\scriptsize$ P $}
     \put (52.5,281) {\scriptsize$ P' $}
     \put (48,81) {\scriptsize$ p' $}
     \put (48,114) {\scriptsize$ v $}
     \put (48,148) {\scriptsize$ i $}
     \put (48,232) {\scriptsize$ p $}
     \put (48,265) {\scriptsize$ u $}
     \put (48,298) {\scriptsize$ i' $}
     }\setlength{\unitlength}{1pt}
  \end{picture}}
	\qquad .
\ee
It is easy to check that inserting $f$ into \eqref{eq:simple-proj-leftright} results in the identity
\be\label{eq:tr-tr}
	tr^r_{X'}(i'\circ u \circ v \circ p')
	=
	tr^r_{X}(i\circ v \circ u \circ p) \ ,
\ee
where for RHS we used 
cyclicity of the categorical trace, the zig-zag axiom and that $p'\circ i'= \id_{P'}$.
For $P=P'$ and 	$u=h$, $v=\id_P$, 
this shows the statement of the lemma.
\end{proof}

\begin{proof}[Proof of Proposition~\ref{prop:mtrace}]~ 

\smallskip

\noindent
{\em Part 1.} (following \cite[Sec.\,4.5]{Geer:2011a})

\smallskip

\noindent
{\em Existence:} By Lemma~\ref{lem:tr-indep-choice} we have a family of functions $\big(t_P : \End(P) \to k\big)_{P \in \Proj(\Cc)}$. 
This family is not identically zero since by the definition in \eqref{eq:tQ-choice}, $t_Q$ is not zero.
	By \eqref{eq:tr-tr} in the proof of Lemma~\ref{lem:tr-indep-choice},
the $t_P$ satisfy the cyclicity requirement~\eqref{eq:mod-tr_cyclic}.

For the partial trace property \eqref{eq:mod-tr_partial-trace} choose $Y \in \Cc$ such that there are $p : Q \otimes Y \to P$, $i : P \to Q \otimes Y$, $p \circ i =\id_P$, in order to compute $t_P$ from the definition in \eqref{eq:tP-defn}
	with $X$ replaced by $Y$. 
Then fix 
\be
	p' = \big[ \, Q  (Y  X) \xrightarrow{\sim} (Q  Y)  X \xrightarrow{ p \otimes \id_X} P X \, \big] \quad , \quad
	i' = \big[ \,  P X \xrightarrow{ i \otimes \id_X} (Q  Y)  X \xrightarrow{\sim} Q  (Y  X)  \, \big]
\ee
to compute $t_{P \otimes X}$ from \eqref{eq:tP-defn}. A short calculation shows  \eqref{eq:mod-tr_partial-trace}.

\smallskip

\noindent
{\em Uniqueness:} First note that for the modified trace $t_P$ in \eqref{eq:tP-defn}
	(or in fact for any modified trace)
 we have, for $f \in \End(Q \otimes X)$,
\be\label{eq:mtrace-aux1}
	tr^r_X(f) = \frac{t_{Q \otimes X}(f)}{t_Q(\id_Q)} \cdot \id_Q \ .
\ee
To see this, apply $t_Q$ to both sides.
Let now $\big(t_P' : \End(P) \to k\big)_{P \in \Proj(\Cc)}$ be a modified trace. 
	Let $P \in \Proj(\Cc)$ and $h \in \End(P)$ be given. Pick $X$, $p : Q \otimes X \to P$, $i : P \to Q \otimes X$ such that $p \circ i=\id_P$ and compute
\begin{align}
t'_P(h)
&~=~
t'_P(h \circ p \circ i)
\overset{\text{cycl.}}=
t'_{Q \otimes X}(i \circ h \circ p)
\overset{\text{part.tr.}}=
t'_{Q}( tr^r_X(i \circ h \circ p))
\nonumber\\ &
\overset{\text{\eqref{eq:mtrace-aux1}}}=
\frac{t_{Q \otimes X}(i \circ h \circ p)}{t_Q(\id_Q)} \cdot t'_{Q}(\id_Q)
\overset{\text{cycl.}}=
\frac{t'_{Q}(\id_Q)}{t_Q(\id_Q)} \cdot t_P(h) \ .
\end{align}

\noindent
{\em Part 2.} 
(following the proof of \cite[Prop.\,6.6]{Costantino:2014sma})

\smallskip

\noindent
	We will show non-degeneracy in the first argument of the pairing \eqref{eq:pairingMP}. Non-degeneracy in the second argument can be seen analogously. 
	
Let $M \in \Cc$, $P \in \Proj(\Cc)$ and $f \in \Cc(M,P)$, $f \neq 0$, 
be given. We need to show that there is a $g \in \Cc(P,M)$ such that $t_P(f \circ g) \neq 0$.
Pick $X$, $p : Q \otimes X \to P$, $i : P \to Q \otimes X$ such that $p \circ i=\id_P$. 
Then also $i \circ f \neq 0$. Define 
\be
	u ~:=~ \big[ \, 
	MX^* 
	\xrightarrow{(i \circ f) \otimes \id} (QX)X^* \xrightarrow{\sim} Q(XX^*)
	\xrightarrow{\id \otimes \widetilde\ev_X}  Q\one \xrightarrow{\sim}  Q 
	\,\big]  \ .
\ee
Using the zig-zag identity,
we can recover $i \circ f$ from $u$ as
\be
	i \circ f ~=~ \big[ \, M \xrightarrow{\sim} M \one
	\xrightarrow{\id \otimes \widetilde\coev_X} M(X^*X) \xrightarrow{\sim} 
	(M X^*)X
	\xrightarrow{u\otimes \id_X}  QX  
	\,\big]  \ .
\ee
In particular, since $i \circ f \neq 0$, we must have $u \neq 0$.
Since $Q$ is simple, $u$ is surjective. Since $Q$ is projective, there is $v : Q \to M \otimes X^*$ such that $u \circ v = \id_Q$. Define
\be
	g ~:=~ \big[ \, P 
	\xrightarrow{i} QX
	\xrightarrow{v \otimes \id} (MX^*)X
	\xrightarrow{\sim} M(X^*X)
	\xrightarrow{\id \otimes \ev_X}  M\one \xrightarrow{\sim} M 
	\,\big] \ .
\ee
Then
\begin{align}
t_P(f \circ g) 
&\overset{\text{part.tr.}}= 
t_{P \otimes X^*}\!\!\left(~
  \raisebox{-0.5\height}{\setlength{\unitlength}{.75pt}
  \begin{picture}(85,174)
   \put(0,0){\scalebox{.75}{\includegraphics{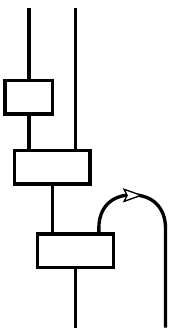}}}
     \put(0,0){
     \put (32,-1) {\scriptsize$ P $}
     \put (74,-1) {\scriptsize$ X^* $}
     \put (9,168) {\scriptsize$ P $}
     \put (32,168) {\scriptsize$ X^* $}
     \put (14,63) {\scriptsize$ Q $}
     \put (0,100) {\scriptsize$ M $}
     \put (34,44) {\scriptsize$ i $}
     \put (22,84) {\scriptsize$ v $}
     \put (10,118) {\scriptsize$ f $}
     }\setlength{\unitlength}{1pt}
  \end{picture}}
 ~~\right)
~~\overset{\text{cycl}}= ~
t_{Q}\!\!\left(~
  \raisebox{-0.5\height}{\setlength{\unitlength}{.75pt}
  \begin{picture}(71,174)
   \put(0,0){\scalebox{.75}{\includegraphics{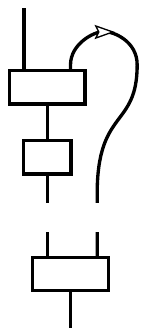}}}
     \put(0,0){
     \put (32,-1) {\scriptsize$ Q $}
     \put (9,169) {\scriptsize$ Q $}
     \put (19,60) {\scriptsize$ M $}
     \put (44,60) {\scriptsize$ X^* $}
     \put (27,106) {\scriptsize$ P $}
     \put (34,33) {\scriptsize$ v $}
     \put (22,89) {\scriptsize$ f $}
     \put (22,122) {\scriptsize$ i $}
     }\setlength{\unitlength}{1pt}
  \end{picture}}
 ~~\right)
\nonumber \\ 
&\overset{\phantom{\text{part.tr.}}}= 
t_Q(u \circ v) ~=~ 
t_Q(\id_Q) ~\neq~ 0 \ ,
\end{align}
where for the first equality we used the partial trace property~\eqref{eq:mod-tr_partial-trace} on RHS and then the pivotal structure in~\eqref{eq:ev-coev-tilde}.
\end{proof}

\begin{lemma}\label{lem:exchange-partial-trace}
Suppose $\Cc$ is in addition braided. Let $t$ be a modified trace on $\Cc$ and let 
$R,S \in \Proj(\Cc)$ and $f \in \End(R)$, $g \in \End(S)$. Then
\be\label{eq:tPtQ}
	t_R\!\left(~
  \raisebox{-0.5\height}{\setlength{\unitlength}{.75pt}
  \begin{picture}(60,170)
   \put(0,0){\scalebox{.75}{\includegraphics{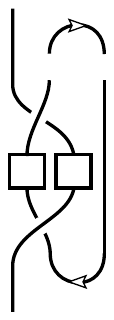}}}
   \put(0,0){
     \put(0,0){
     \put (5,0) {\scriptsize$ R $}
     \put (5,161) {\scriptsize$ R $}
     \put (22,125) {\scriptsize$ S $}
     \put (47,125) {\scriptsize$ S^* $}
     \put (12,76) {\scriptsize$ g $}
     \put (36,76) {\scriptsize$ f $}
     }\setlength{\unitlength}{1pt}}
  \end{picture}}
	~\right) 
	~~=~~
	t_S\!\left(~
  \raisebox{-0.5\height}{\setlength{\unitlength}{.75pt}
  \begin{picture}(60,170)
   \put(0,0){\scalebox{.75}{\includegraphics{pics/pic04.pdf}}}
   \put(0,0){
     \put(0,0){
     \put (5,0) {\scriptsize$ S $}
     \put (5,161) {\scriptsize$ S $}
     \put (22,125) {\scriptsize$ R $}
     \put (47,125) {\scriptsize$ R^* $}
     \put (12,76) {\scriptsize$ f $}
     \put (36,76) {\scriptsize$ g $}
     }\setlength{\unitlength}{1pt}}
  \end{picture}}
	~\right) 
	\qquad .
\ee
\end{lemma}

\begin{proof}
One computes
\be
\text{RHS of \eqref{eq:tPtQ}}
\overset{\text{(*)}}{=}
t_{R(SR^*)} (h)
	\overset{\text{part tr.}}{=}
	t_R(tr_{SR^*}^r(h))
	\overset{\text{(**)}}{=}
\text{LHS of \eqref{eq:tPtQ}} \ ,
\ee
where
\be
h ~=~
  \raisebox{-0.5\height}{\setlength{\unitlength}{.75pt}
  \begin{picture}(68,209)
   \put(0,0){\scalebox{.75}{\includegraphics{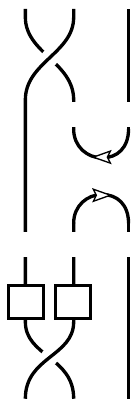}}}
     \put(0,0){
     \put (8,-1) {\scriptsize$ R $}
     \put (32,-1) {\scriptsize$ S $}
     \put (57,-1) {\scriptsize$ R^* $}
     \put (8,201) {\scriptsize$ R $}
     \put (32,201) {\scriptsize$ S $}
     \put (57,201) {\scriptsize$ R^* $}
     \put (8.5,54) {\scriptsize$ g $}
     \put (32.5,54) {\scriptsize$ f $}
     \put (8,80.5) {\scriptsize$ S $}
     \put (32,80.5) {\scriptsize$ R $}
     \put (57,80.5) {\scriptsize$ R^* $}
     \put (32,143) {\scriptsize$ R $}
     \put (57,143) {\scriptsize$ R^* $}
     }
     \setlength{\unitlength}{1pt}
  \end{picture}} \qquad .
\ee
Step (*) follows from cyclicity of the modified trace, as well as naturality of the braiding. In step (**)
 we  used that the partial trace over $S\otimes R^*$ amounts to closing the double string, see \eqref{eq:partial-trace-def}. 
 In the last equality,
 to straighten the line coloured by $R$ we also used~\eqref{eq:ev-coev-tilde} and the property of the pivotal structure that $\delta^*_R = (\delta_{R^*})^{-1}$.
\end{proof}

\begin{remark}\label{rem:modified-trace-consequences}
~\\[-1.5em]
\begin{enumerate}\setlength{\leftskip}{-1em}
\item
The result of Proposition~\ref{prop:mtrace} is remarkable, because the categorical trace defined on all of $\Cc$ via the pivotal structure vanishes on $\Proj(\Cc)$ unless $\Cc$ is semisimple. Indeed, 
$\Cc$ is semisimple iff  $P_\one = \one$,
 which in turn is equivalent to the existence of $f,g$ such that $\one \xrightarrow{f} P_\one \xrightarrow{g} \one$ is non-zero. But for any projective object $P$, the categorical trace
\be
\one \xrightarrow{\coev_P} P \ot P^* \xrightarrow{f \ot \id} P \ot P^* \xrightarrow{\widetilde\ev_P} \one \ ,
\ee 
factors through $P_\one$ 
	in this way. 
\item
For each $P \in \Proj(\Cc)$, $t_P$ turns the $k$-algebra $\End(P)$ into a symmetric Frobenius algebra, or a symmetric algebra for short. We will review some aspects of symmetric algebras and their categories of modules in Section~\ref{sec:sym-alg-ideal}.
\end{enumerate}
\end{remark}

Since a factorisable finite tensor category is automatically unimodular \cite[Prop.\,4.5]{Etingof:2004}, 
	we state the following consequence of Theorem~\ref{thm:main-simpleproj} and Proposition~\ref{prop:mtrace} for later use:

\begin{corollary}\label{cor:factorisable-implies-trace}
A factorisable and pivotal finite tensor category $\Cc$ 
over $k$ which satisfies Condition~P
admits an up-to-scalar unique modified trace on $\Proj(\Cc)$. For each non-zero modified trace, the pairings \eqref{eq:pairingMP} are non-degenerate.
\end{corollary}

This result generalises the existence of a modified trace for factorisable Hopf algebras (over an algebraically closed field of characteristic zero) proven in~\cite[Thm.\,4.3.1]{Geer:2011b}.

\begin{remark}
The converse of Corollary~\ref{cor:factorisable-implies-trace} does not hold, i.e.\ it is not true that a pivotal finite tensor category $\Cc$ over $k$ which satisfies Condition~P and admits a non-zero modified trace on $\Proj(\Cc)$ can be equipped with a braiding that makes it factorisable. The simplest example is the category of super-vector spaces, which only admits symmetric
braidings. A non-semisimple example is provided by the restricted quantum group $\overline{U}_q sl(2)$.  
This is a unimodular Hopf algebra which has a simple projective module but whose category of  modules does not admit a braiding~\cite{[FGST],KS,Gainutdinov:2015lja}.
\end{remark}

\section{Symmetric algebras and ideals in $\boldsymbol{\End(\Id_\Cc)}$}\label{sec:sym-alg-ideal}

In this section we will collect some facts about symmetric algebras following \cite{Broue:2009,Cohen:2008} and translate them into categorical statements. In the end we give implications for categories with modified trace as in the previous section.

\medskip

For a $k$-algebra $A$, denote by
\be
	C(A) ~=~ \big\{ \,\varphi : A \to k \,\big|\, \varphi(ab) = \varphi(ba) \text{ for all }a,b \in A \,\big\} 
\ee
the space of {\em central forms} on $A$. 
A {\em symmetric algebra} $A$ over $k$ (aka a {\em symmetric Frobenius algebra}) is a finite-dimensional $k$-algebra together with a central form $\eps$ such that the induced pairing $(a,b) \mapsto \eps(ab)$ is non-degenerate. In this case the map
\be\label{eq:zeta:ZA-CA}
	\zeta : Z(A) \to C(A) \quad , \quad z \mapsto \eps(z \cdot (-))
\ee
is an isomorphism, and the invertible elements of $Z(A)$ are precisely mapped to those elements of $C(A)$ that induce a non-degenerate pairing (see e.g.\ \cite[Lem.\,2.5]{Broue:2009}).

\newcommand{\cop}{\gamma}
Let $(A,\eps)$ be a symmetric algebra and denote by 
$\cop \in A \otimes A$, $\cop = \sum_{(\cop)} \cop' \otimes \cop''$, 
the copairing, that is, for all $a \in A$ we have
 $a = \sum_{(\cop)} \eps(a \cop')\cop''
 = \sum_{(\cop)}  \cop'\eps(\cop''a)$.
	We recall that the copairing is unique and that, since $\eps$ is central, the copairing is symmetric.
Define the map
\be\label{eq:tau-map-A-ZA-def}
	\tau : A \to Z(A)
	\quad , \quad
	a \mapsto \sum_{(\cop)} \cop' a \cop'' \ .
\ee
It is easy to check that the image of $\tau$ does indeed lie in the centre of $A$ (see e.g.\ \cite[Sec.\,3.A]{Broue:2009}). 
The argument uses that the copairing $\cop$ of a non-degenerate 
invariant 
pairing satisfies $\sum_{(\cop)}(a\cop')\otimes \cop'' = \sum_{(\cop)}\cop'\otimes (\cop'' a)$ for all $a\in A$.

We will need the following chains of inclusions in $C(A)$ and $Z(A)$, compatible with the map $\zeta$ from \eqref{eq:zeta:ZA-CA} \cite{Cohen:2008}:
\be\label{eq:ZA-CA-subsets}
\xymatrix{
Z(A)\ar[d]_\zeta & \supset & \Rey(A) \ar[d]_\zeta & \supset& \Hig(A) \ar[d]_\zeta 
\\
C(A) & \supset & R(A) & \supset& I(A)
}
\ee
The individual subsets are defined as follows:
\begin{itemize}
\item
$\Rey(A) = \mathrm{ann}_{Z(A)}(\Jac(A))$ is the {\em Reynolds ideal}. (Symmetry of $A$ is not required for this definition.) Here, $\Jac(A)$ is the Jacobson radical of $A$. $\Rey(A)$ is an ideal in~$Z(A)$.
\item
$\Hig(A) = \mathrm{im}(\tau)$
 is the {\em Higman ideal} or {\em projective centre}\footnote{
The name ``projective centre'' can be motivated from the observation that $\Hig(A)$ can be described as all elements in $Z(A)$ such that the corresponding endomorphism of the $A$-$A$-bimodule $A$ factors through a projective 
$A$-$A$-bimodule
\cite[Prop.\,2.3\,\&\,2.4]{Liu:2012}. 
}.
It is an ideal in $Z(A)$ as $z\tau(a) = \tau(za)$ for $z\in Z(A)$, and by \cite[Lem.\,4.1]{Hethelyi:2005}\footnote{
In \cite[Sec.\,4]{Hethelyi:2005} the underlying field is assumed to be algebraically closed, but this is not used in the proof of Lemma 4.1
in \cite{Hethelyi:2005}.}
it is contained in $\Rey(A)$.
Note that even though $\tau$ depends on the choice of
	a
non-degenerate central form $\varepsilon$, $\mathrm{im}(\tau)$ does not. Indeed, any two $\varepsilon$, $\varepsilon'$ would be related by an invertible $z \in Z(A)$ as $\varepsilon'(a) = \varepsilon(za)$, and one checks that $\tau'(a) = \tau(z^{-1}a)$.
\item
$R(A) = \mathrm{span}_k\{ \,\chi_M \,|\, M \text{ right $A$-module} \}$, where $\chi_M(a) = \mathrm{tr}_M(a)$, the trace in $M$ over the linear map given by acting with $a\in A$. That $\zeta(\Rey(A))=R(A)$ is shown in \cite[Thm.\,1.6]{Lorenz:1997} (for finite-dimensional not necessarily symmetric algebras).
\item
$I(A) = \mathrm{span}_k\{\, \chi_P \,|\, P \text{ projective right $A$-module} \}$. In \cite[Prop.\,2.1]{Cohen:2008}
it is shown that $\zeta(\Hig(A))=I(A)$.
\end{itemize}

\begin{lemma}\label{lem:ZA=ReyA}
For an algebra $A$ over some field, $Z(A)=\Rey(A)$ if and only if $A$ is semisimple.
\end{lemma}

\begin{proof}
If $A$ is semisimple, $\Jac(A)=0$ and hence $Z(A) = \Rey(A)$. 
Conversely, if
$Z(A) = \Rey(A)$ then $1 \in \Rey(A)$ and since $\Jac(A)$ annihilates $\Rey(A)$ we must have $\Jac(A)=0$.
\end{proof}

The {\em Cartan matrix} of a finite-dimensional algebra $A$ over some field 
is the 
$\Irr(\rmod{A}) \times \Irr(\rmod{A})$ 
 matrix $\CM(A)$  with entries
\be
	\CM(A)_{U,V} ~=~ [P_U:V] \ ,
\ee
that is, the multiplicity of the simple $A$-module $V$ in the composition series of the projective cover $P_U$ of $U$.

Let $e_U$, $U \in \Irr(\rmod{A})$ be a choice of primitive idempotents such that $e_UA \cong P_U$ as right $A$-modules. If
  the underlying field is given by $k$, which we assumed to be algebraically closed in the outset of the paper,
the endomorphism spaces of simple $A$-modules are one-dimensional and we can write
\be
	\CM(A)_{U,V} ~=~ \dim_k \Hom_A(P_V,P_U) ~=~ \dim_k( \, e_U A e_V \, )  \ .
\ee
The last equality follows by noting that every right $A$-module map $f : e_VA \to e_UA$ is given by left multiplication with an element $e_U a e_V$, $a \in A$.
	Namely,
in one direction we set $f\mapsto f(e_V)e_V\in e_UA e_V$, 
while  in the other direction  $e_U a e_V \mapsto e_U a e_V \cdot (-)$ which is  in $\Hom_A(e_VA,e_UA)$, and one checks that both the maps are inverses of each other.

\begin{proposition}\label{prop:cartan-sym+rank}
Let $A$ be a symmetric algebra over 
	$k$.
Then:
\begin{enumerate}
\item The Cartan matrix of $A$ is symmetric, 
$\CM(A)_{U,V} = \CM(A)_{V,U}$.
\item Let $n = |\Irr(\rmod{A})|$ and 
let 
$\widehat\CM \in \mathrm{Mat}_n(k)$ 
be the image of $\CM(A)$ under the canonical homomorphism $\mathrm{Mat}_n(\Zb) \to \mathrm{Mat}_n(k)$. We have
\be
	\mathrm{rank}\big(\widehat\CM \big) 
	~=~ \dim_k \Hig(A) \ .
\ee
\end{enumerate}
\end{proposition}

	\begin{proof}
For part 1 one checks that the non-degenerate pairing on $A$ descends to a non-degenerate pairing between $e_UAe_V$ and $e_VAe_U$.
	Indeed, for every $a\in A$ such that $e_Uae_V \neq 0$ there exists $b\in A$ such that $(e_Uae_V, b) \neq 0$. But then also $(e_Uae_V, e_Vbe_U) \neq 0$ since $\eps(e_U ae_V  b ) = \eps(e_U e_Uae_V e_V b ) = \eps(e_Uae_V e_V b e_U)$. 
Hence, $\dim_k(e_U A e_V)=\dim_k(e_V A e_U)$.
The more surprising part 2 is proved in \cite[Cor.\,2.7]{Liu:2012}. 
	\end{proof}
	
Note that if $k$ has 
characteristic~$0$, then $\mathrm{rank}(\tilde \CM)$ is just the
  rank of $\CM(A)$ over~$\mathbb{Q}$.

\medskip

We now give a categorical formulation of the inclusions $Z(A) \supset \Rey(A) \supset \Hig(A)$ in \eqref{eq:ZA-CA-subsets}, which will be applicable in particular in the presence of a modified trace as in Section~\ref{sec:traces}. For the rest of this section, let us fix
\begin{itemize}
\item $\Ac$: a finite abelian category over 
	$k$,
\item $G$: a projective generator of $\Ac$,
\item $E = \End(G)$ the $k$-algebra of endomorphisms of $G$.
\end{itemize}

Recall the Hom-tensor adjoint equivalence 
\be\label{eq:C-modE-equiv}
\xymatrix@C=60pt@W=5pt@M=5pt{
\Ac \ar@/^/[r]^{\Ac(G,-)}
&\rmod{E} \ar@/^/[l]^{- \otimes_E G}
}
\ee
of $k$-linear
categories
between $\Ac$ and the category $\rmod{E}$ of finite-dimensional right $E$-modules  (see e.g.\ \cite[Sect.\,1.8]{EGNO-book}).

It is known that being symmetric is a Morita-invariant property of an algebra, as is the chain of ideals $Z(E) \supset \Rey(E) \supset \Hig(E)$ (see e.g.\ \cite{Hethelyi:2005,Broue:2009,Liu:2012} for related results). We can therefore use the equivalence in \eqref{eq:C-modE-equiv} and the ideals in \eqref{eq:ZA-CA-subsets}
for the choice $A=E$ and for a given projective generator $G$
 to define a chain of ideals in $\End(\Id_\Ac)$ as:
\be\label{eq:End-Rey-Hig}
\xymatrix@C=10pt@R=20pt@W=4pt@M=4pt{
\End(\Id_\Ac)\ar^{\cong}[d] & \supset & \Rey(\Ac) \ar^{\cong}[d] & \supset& \Hig(\Ac) \ar^{\cong}[d] 
\\
Z(E) &  & \Rey(E) & & \Hig(E)
}
\ee
Rather than showing directly that the above definition is independent of the choice of $G$, we will now give a different definition of $\Rey(\Ac)$ and $\Hig(\Ac)$ which does not require a choice of a projective generator. We show in Propositions~\ref{prop:Rey(Amod)=Rey(A)} and \ref{prop:HigA-HigAmod} that this definition agrees with the one above.

\subsubsection*{$\boldsymbol{\Rey(\Ac)}$} 

For each $P \in \Proj(\Ac)$ define the subspace $J_P \subset \End(P)$ as follows:
\be\label{eq:JP-def}
	J_P ~=~ \big\{ \, f : P \to P \,\big|\,
	\forall \, U \in \Ac \text{ simple},\, u \in \Ac(P,U) :
	u \circ f = 0 \, \big\} \ .
\ee
By definition, the Jacobson radical $\Jac(R)$ of a ring $R$ consists of all $r \in R$ which act as zero on all simple $R$-modules. 
Since  $G \in \Ac$ is a projective generator, $\Ac(G,U)$, where $U$ runs over all simple $U \in \Ac$, gives all simple right $E$-modules up to isomorphism (by the equivalence \eqref{eq:C-modE-equiv}). {}From~\eqref{eq:JP-def} we therefore get
\be\label{eq:JG-Jac}
	J_G ~=~ \Jac(E) \ .
\ee
In the case $\Ac = \rmod{A}$ we can take $G = A_A$, so that $\End(G)=A$ and hence $J_A = \Jac(A)$.

Using the subspaces \eqref{eq:JP-def}, we define
\be
	\Rey(\Ac) ~=~ \big\{ \, \eta \in \End(\Id_\Ac)  \,\big|\,
	\forall\, P \in \Proj(\Ac), f \in J_P : 
	\eta_P \circ f = 0 \, \big\} \ .
\ee
{}From this definition it is immediate that $\Rey(\Ac)$ is an ideal.

Recall that for a projective generator $G \in \Ac$,
the map
\be\label{eq:EndId-ZE-iso}
	\xi : \End(\Id_\Ac) \to Z(E) \quad , \quad \eta \mapsto \eta_G
\ee
is an algebra isomorphism
(that $\eta_G$ is central in $E$ follows from the naturality of $\eta$;	that \eqref{eq:EndId-ZE-iso} is a bijection is clear in the case $\Ac = \rmod{E}$,
	for general $\Ac$ use the equivalence~\eqref{eq:C-modE-equiv}). We can use this to relate $\Rey(\Ac)$ to the previous definition in terms of algebras:

\begin{proposition}\label{prop:Rey(Amod)=Rey(A)}
Let $A$ be a finite-dimensional $k$-algebra. Then $\xi$ in \eqref{eq:EndId-ZE-iso} restricts to an isomorphism of algebras $\xi : \Rey(\rmod{A}) \to \Rey(A)$.
\end{proposition}

\begin{proof}
We set $\Ac=\rmod{A}$ and choose the projective generator $G=A_A$ and thus $E=A$.
We have to show that
$\xi(\Rey(\Ac)) = \Rey(E)$.

\smallskip

\noindent
``$\subset$'':
Let $\eta \in \Rey(\Ac)$. Then $\xi(\eta) = \eta_G$. By definition of $\Rey(\Ac)$ we have in particular that for all $f \in J_G
=\Jac(E)$ (recall~\eqref{eq:JG-Jac}), 
$\eta_G \circ f=0$.
 Thus $\eta_G$ annihilates $\Jac(E)$,
  i.e.\ $\eta_G \in \Rey(E)$.

\smallskip

\noindent
``$\supset$'': Conversely, let $r \in \Rey(E)$ and let $\eta = \xi^{-1}(r) \in \End(\Id_\Cc)$. 
Let $P \in \Proj(\Cc)$, and $f \in J_P$. We need to show $\eta_P \circ f = 0$.
Pick a surjection $p : G^{\oplus n} \to P$ and a right-inverse $i : P \to G^{\oplus n}$. 
	We have $i \circ f \circ p \in J_{G^{\oplus n}}$ since for all $u : G^{\oplus n} \to U$, where $U$ is simple,
 $u \circ i \circ f \circ p = 0$ since $u \circ i : P \to U$ and $f \in J_P$.
Note then that $r^{\oplus n}: G^{\oplus n} \to G^{\oplus n}$ and $r^{\oplus n}\in \Rey(\End(G^{\oplus n}))$, therefore
$r^{\oplus n} \circ i \circ f \circ p = 0$. Since $r^{\oplus n} = \eta_{G^{\oplus n}}$ we get
\be
	0 =  \eta_{G^{\oplus n}} \circ i \circ f \circ p
	  = i \circ \eta_P \circ f \circ p \ .
\ee
But $i$ is injective and $p$ surjective, and so $\eta_P \circ f = 0$, as required.
\end{proof}

\subsubsection*{$\boldsymbol{\Hig(\Ac)}$} 

The condition replacing that the algebra is symmetric will be that $\Proj(\Ac)$ is Calabi-Yau (see e.g.\ \cite{Hesse:2016} for a recent treatment of the relation between symmetric Frobenius algebras and Calabi-Yau categories in the semisimple case):

\newcommand{\tCY}{\mathsf{t}}
\begin{definition}\label{def:CY-cat}
A $k$-linear additive category $\mathcal{K}$ with finite-dimensional Hom-spaces is called {\em Calabi-Yau} if it is equipped with a family of $k$-linear maps (the {\em trace maps})
\be
	\big(\,\tCY_M : \End(M) \to k \,\big)_{M \in \mathcal{K}}
\ee
such that for all $M,N \in \mathcal{K}$,
\begin{enumerate}
\item {(\em cyclicity)} for all $f: M\to N$, $g:N \to M$ we have $\tCY_M(g \circ f) = \tCY_N(f \circ g)$.
\item {(\em non-degeneracy)} the pairing $(-,-) : 
\mathcal{K}(M,N) \times \mathcal{K}(N,M) 
\to k$, $(f,g) = \tCY_N(f \circ g)$ is non-degenerate.
\end{enumerate}
We denote by $\mathrm{CY}(\mathcal{K})$ the set of all families of trace maps $(\tCY_M)_{M \in \mathcal{K}}$ which turn $\mathcal{K}$ into a Calabi-Yau category.
\end{definition}

A Calabi-Yau category with one object is the same as a symmetric Frobenius algebra. The  generalisation of the isomorphism \eqref{eq:zeta:ZA-CA} to the present setting is:

\begin{lemma}\label{lem:CY-traces_torsor_EndId*}
Let $\mathcal{K}$ be as in Definition~\ref{def:CY-cat} 
and let $\End(\Id_{\mathcal{K}})^\times$ be the subset of invertible endo-transformations. Then for each $(\tCY_M)_{M \in \mathcal{K}} \in \mathrm{CY}(\mathcal{K})$ 
the map
\be
	\End(\Id_{\mathcal{K}})^\times
	\longrightarrow \mathrm{CY}(\mathcal{K})
	\quad , \quad
	\eta  ~\longmapsto~ \big( f \mapsto \tCY_M(\eta_M \circ f) \big) 
\ee
is a bijection.
\end{lemma}

\begin{proof}
Clearly, $(\tCY_M(\eta_M \circ -))_{M \in \mathcal{K}}$ defines a family of cyclic and non-degenerate traces 
(for non-degeneracy, one notes that $\eta_M \circ g$ is non-zero for any non-zero $g$ because $\eta_M$ is invertible). 
Conversely, given 
	$(\tilde\tCY_M)_{M \in \mathcal{K}} \in \mathrm{CY}(\mathcal{K})$, 
by non-degeneracy for each $M \in \mathcal{K}$ there exists a unique $\eta_M \in \End(M)^\times$ such that
 $\tilde\tCY_M(-) = \tCY_M(\eta_M \circ -)$. 
Rewriting the cyclicity condition $\tilde\tCY_M(g \circ f) = \tilde\tCY_N(f \circ g)$ for all $f : M\to N$ and $g:N\to M$ in terms of 
the family $\tCY_M$
and using its cyclicity property shows $\tCY_M\bigl(g\circ (\eta_N \circ f - f \circ \eta_M)\bigr) = 0$ for any $g$,
and thus by non-degeneracy of $\tCY_M$ we have  $\eta_N \circ f = f \circ \eta_M$ for all $M$, $N$, and $f$.
\end{proof}

Suppose now that the finite abelian category $\Ac$ is such that $\Proj(\Ac)$ is Calabi-Yau with trace maps $(\tCY_P)_{P \in \Proj(\Ac)}$. For the pairings of morphisms between the projective objects $P,R \in \Proj(\Ac)$ we write
\be\label{eq:paring-projective}
	(-,-)_{PR} : \Ac(P,R) \times \Ac(R,P) \to k
	\quad , \quad
	(f,g)_{PR} = \tCY_{R}(f \circ g) \ .
\ee
Denote by $\gamma_{PR} \in \Ac(R,P) \otimes \Ac(P,R)$ the corresponding 
(unique) copairing, that is $\gamma_{PR} = \sum_{(\gamma_{PR})} \gamma_{PR}' \otimes \gamma_{PR}''$ and for all $x : P \to R$, $y : R \to P$, 
\be\label{eq:coparing-gamma-properties}
	\sum_{(\gamma_{PR})} \big(\,x\,,\,\gamma_{PR}'\,\big)_{\!PR} \,\, \gamma_{PR}'' ~=~ x
	\qquad , \qquad
	\sum_{(\gamma_{PR})} \gamma_{PR}' \,\, \big(\,\gamma_{PR}''\,,\,y\,\big)_{\!PR} ~=~ y
	\ .
\ee
We need the following lemma, whose proof is given in Appendix~\ref{app:ReyHigideal}.

\begin{lemma}\label{lem:tau-R-welldef}
For each $R \in \Proj(\Ac)$ and $y \in \End(R)$ there exists a unique $\eta \in \End(\Id_\Ac)$ such that $\eta_{P} = \sum_{(\gamma_{PR})} \gamma_{PR}' \circ y \circ \gamma_{PR}''$ for all $P \in \Proj(\Ac)$.
\end{lemma}

This lemma allows us to define, for each $R \in \Proj(\Ac)$,  a $k$-linear map $\tau_R : \End(R) \to \End(\Id_\Ac)$ by
\be\label{eq:tau_R-defn}
	\tau_R(y) = \eta
	\quad , \quad
	\text{where for all $P \in \Proj(\Ac)$} \quad
	\eta_P = \sum_{(\gamma_{PR})} \gamma_{PR}' \circ y \circ \gamma_{PR}''
	\ .
\ee
These maps are the analogue of the map $\tau$ from \eqref{eq:tau-map-A-ZA-def}. Indeed, $\Hig(\Ac)$ is now defined as the joint image of all the maps $\tau_R$:
\begin{align}\label{eq:HA-def-2}
	\Hig(\Ac) = \big\{\, \eta \in \End(\Id_\Ac) \,\big|\,
	&\exists\, R \in \Proj(\Ac)\,,\, y \in \End(R) \,:\, \eta = \tau_R(y) \,\big\} \ .
\end{align}
The point of the above definition is to avoid the choice of a projective generator. 
To actually compute $\Hig(\Ac)$, choosing a projective generator $G$ may nonetheless be useful.
For example, we show in Corollary~\ref{cor:HigA-from-projgen} that $\Hig(\Ac) = \mathrm{im}(\tau_G)$.

As was the case for the Higman ideal of a symmetric algebra, also $\Hig(\Ac)$ is independent of the choice of traces:

\begin{proposition}\label{prop:HigA-indep-tr}
Let $\Ac$ be a finite abelian category over $k$ such that $\Proj(\Ac)$ is Calabi-Yau. Then $\Hig(\Ac)$ is independent of the choice of the family of traces $(\tCY_P)_{P \in \Proj(\Ac)}$.
\end{proposition}

\begin{proof}
Let $\tCY, \tilde\tCY \in \mathrm{CY}(\Proj(\Ac))$. By Lemma~\ref{lem:CY-traces_torsor_EndId*} there is $\eta \in \End(\Id_{\Proj(\Ac)})^\times$ such that $\tilde\tCY_P(f) = \tCY_P(\eta_P \circ f)$ for all $P \in \Proj(\Ac)$ and $f \in \End(P)$.
Let $\gamma_{PR}$ and $\tilde\gamma_{PR}$ be the copairings arising from $\tCY$ and $\tilde\tCY$ as in \eqref{eq:coparing-gamma-properties}. If we write $\gamma_{PR} = \sum_{(\gamma_{PR})} \gamma_{PR}' \otimes \gamma_{PR}''$, it is easy to see that
\be\label{eq:HigA-indep-tr_aux1}
	\tilde\gamma_{PR} 
	~=~ \sum_{(\gamma_{PR})} (\gamma_{PR}' \circ \eta_R^{-1})
	 \otimes \gamma_{PR}''
	~=~ \sum_{(\gamma_{PR})} \gamma_{PR}' \otimes (\eta_R^{-1} \circ \gamma_{PR}'') \ .
\ee
Write $\tilde\tau_R$ for the map \eqref{eq:tau_R-defn} computed from $\tilde\gamma$. It is immediate from the above equality that for $y \in \End(R)$,
\be
	\tilde\tau_R(y) = \tau_R(\eta_R^{-1} \circ y) \ .
\ee
Thus indeed the joint image of the $\tau_R$ is independent of the choice of traces $\tCY_P$.
\end{proof}

The next proposition makes the connection between the definition \eqref{eq:HA-def-2} of $\Hig(\Ac)$ 
and $\Hig(A)$ for a symmetric algebra $A$ when taking $\Ac = \rmod{A}$. The proof has been relegated to Appendix~\ref{app:ReyHigideal}.

\begin{proposition}\label{prop:HigA-HigAmod}
Let $A$ be a finite-dimensional $k$-algebra. 
\begin{enumerate}
\item $A$ admits a central form turning it into a symmetric algebra if and only if $\Proj(\rmod{A})$ admits a family of traces turning it into a Calabi-Yau category.
\item
If $A$ is symmetric, the map $\xi$ from \eqref{eq:EndId-ZE-iso} 
for $G=A_A$
restricts to an isomorphism of algebras $\xi : \Hig(\rmod{A}) \to \Hig(A)$.
\end{enumerate}
\end{proposition}

Combining this proposition 
and Proposition~\ref{prop:Rey(Amod)=Rey(A)}
with the equivalence \eqref{eq:C-modE-equiv} and the inclusions \eqref{eq:ZA-CA-subsets} we get that we indeed have the inclusions \eqref{eq:End-Rey-Hig}.

\medskip

As for algebras, the Cartan matrix of $\Ac$ is the $\Irr(\Ac){\times}\Irr(\Ac)$-matrix $\CM(\Ac)$ with entries $\CM(\Ac)_{UV} = [P_U:V]$, i.e.\ in the Grothendieck group $\Gr(\Ac)$ we have
\be\label{eq:composition-series-proj-cover}
  [P_U] ~= \sum_{V \in \Irr(\Ac)} \!\! \CM(\Ac)_{UV} \, [V] \ .
\ee
	Since $k$ is algebraically closed, we can write
\be\label{eq:cartan-mat-def}
	\CM(\Ac)_{UV} ~=~ \dim_k \Ac(P_V,P_U) \ .
\ee
In view of Proposition~\ref{prop:HigA-HigAmod}
 and the equivalence \eqref{eq:C-modE-equiv},
  we have the following corollary to Proposition~\ref{prop:cartan-sym+rank}.

\begin{corollary}\label{cor:cartan-sym+rank_general}
Let $\Ac$ be a finite abelian category over
$k$ such that $\Proj(\Ac)$ is Calabi-Yau. 
Then
\begin{enumerate}
\item The Cartan matrix of $\Ac$ is symmetric, 
$\CM(\Ac)_{U,V} = \CM(\Ac)_{V,U}$.
\item Let $n = |\Irr(\Ac)|$ and 
let $\widehat\CM \in \mathrm{Mat}_n(k)$ be the image of $\CM(\Ac)$ under the canonical homomorphism $\mathrm{Mat}_n(\Zb) \to \mathrm{Mat}_n(k)$. We have
\be
	\mathrm{rank}\big(\widehat\CM \big) 
	~=~ \dim_k \Hig(\Ac) \ .
\ee
\end{enumerate}
\end{corollary}

In the context of this paper we are particularly interested in the case of factorisable and pivotal finite tensor categories $\Cc$. 
In this case we will later give a more direct description of $\Rey(\Cc)$ and $\Hig(\Cc)$ in terms of the $\phi_M$ defined in \eqref{eq:tildechi-def}, see Proposition~\ref{prop:ReyHig=span} below.
The next proposition shows in particular that Corollary~\ref{cor:cartan-sym+rank_general} is applicable to such $\Cc$.

\begin{proposition}\label{prop:cy-ssi-inclusion}
Let $\Cc$ be a factorisable and pivotal finite tensor category over $k$
	which satisfies Condition~P.
\begin{enumerate}
\item $\Proj(\Cc)$ admits a family of traces turning it into a Calabi-Yau category.
\item $\Cc$ is semisimple if and only if any one\ of the inclusions $\End(\Id_\Cc) \supset \Rey(\Cc) \supset \Hig(\Cc)$ is an equality. 
\end{enumerate}
\end{proposition}

\begin{proof}
Part 1 is immediate from Corollary~\ref{cor:factorisable-implies-trace}
by setting $\tCY_P:= t_P$ for any $P\in\Proj(\Cc)$, where $t$ is the modified trace on $\Proj(\Cc)$.
 For Part 2, the equivalence of semisimplicity to $\End(\Id_\Cc) = \Rey(\Cc)$ follows from Proposition~\ref{prop:Rey(Amod)=Rey(A)} and Lemma~\ref{lem:ZA=ReyA}. 

To show that $\Rey(\Cc) = \Hig(\Cc)$ is equivalent to semisimplicity of $\Cc$, we pick a projective generator $G$ and write $E=\End(G)$. 
Via the equivalence \eqref{eq:C-modE-equiv} to $\rmod{E}$, and by Propositions~\ref{prop:Rey(Amod)=Rey(A)} and \ref{prop:HigA-HigAmod}, 
the equality $\Rey(\Cc) = \Hig(\Cc)$ is equivalent to $\Rey(E)=\Hig(E)$, and by the isomorphisms in \eqref{eq:ZA-CA-subsets} in turn to $R(E)=I(E)$.

Let $\widehat\CM$ be the image of $\CM(\Cc)$ in $\mathrm{Mat}_n(k)$ as in Corollary~\ref{cor:cartan-sym+rank_general}. We now show:

\medskip

\noindent
{\em Claim:} $R(E)=I(E)$ iff $\widehat\CM$ is non-degenerate.

\medskip

This claim completes the proof of part 2, since by \cite[Thm.\,6.6.1]{EGNO-book}, a pivotal finite tensor category over $k$
is semisimple if and only if $\widehat\CM$ is non-degenerate.

\medskip

\noindent
{\em Proof of claim:}
Recall from \eqref{eq:ZA-CA-subsets} that the characters $\chi_M$ of $E$-modules $M$ span $R(E)$ (over~$k$) and those of projective $E$-modules span $I(E)$.
Using again the equivalence~\eqref{eq:C-modE-equiv}, we have $\chi_{\tilde P_U} = \sum_{V \in \Irr(\Cc)} \widehat\CM_{UV} \, \chi_{\tilde V}$, where $\tilde X := \Cc(G,X)$ is the image of $X$ under the equivalence~\eqref{eq:C-modE-equiv}. In particular, the image of the $k$-linear map described by the matrix~$\widehat\CM$ in the basis $\{ \chi_{\tilde V} \}_{V \in \Irr(\Cc)}$ of $R(E)$, is precisely $I(E)$.
We conclude that $R(E)=I(E)$ is equivalent to non-degeneracy of $\widehat\CM$, proving the claim.
\end{proof}

Part 2 of the proposition also follows from \cite[Cor.\,4.2\,\&\,Thm.\,5.12]{Shimizu:2015} together with Proposition~\ref{prop:ReyHig=span} below.
\footnote{
In \cite[Cor.\,2.3]{Cohen:2008} it is stated (among other things) that for a symmetric algebra $A$ over an  algebraically closed field of characteristic zero, $R(A)=I(A)$ implies that $A$ is semisimple. This would also imply part 2 of Proposition~\ref{prop:cy-ssi-inclusion}. However this statement is not true, as illustrated by the following counterexample (due to Ehud Meir): Take $A = k[X]/\langle X^2 \rangle$ and $\eps(1)=0$, $\eps(X)=1$. Then $A$ is a symmetric algebra. Its unique simple module is $k$ and the projective cover  of $k$ is $P_k=A$. We have $\chi_A = 2 \cdot \chi_k$. Thus $I(A)=R(A)$, but $A$ is not semisimple.

In \cite{Cohen:2008}, Corollary~2.3 is only used in Theorem~2.9, and there only the remaining parts of Corollary~2.3 are relevant, so that the above counterexample does not affect the other results in \cite{Cohen:2008} (we thank Miriam Cohen and Sara Westreich for correspondence on this point).

On the other hand, it is shown in \cite[Thm.\,2]{Lorenz:1997} that if $A$ is in addition a Hopf algebra such that the square of the antipode is inner, $R(A)=I(A)$ does indeed imply that $A$ is semisimple.
This is a special case of  \cite[Thm.\,6.6.1]{EGNO-book}.
}

\section{Characters and the modified trace}\label{sec:ch-tr}

In this section we fix 
\begin{itemize}
\item $\Cc$ : a factorisable and pivotal finite tensor category over $k$
	which satisfies Condition~P
	given in the beginning of Section~\ref{sec:projsimpleex},
\item $t$ : a choice of non-zero modified trace on $\Proj(C)$ via Corollary~\ref{cor:factorisable-implies-trace},
\item $G$ : a projective generator for $\Cc$,
\item $E := \End(G)$, a symmetric $k$-algebra via $t_G$.
\end{itemize}

Recall from \eqref{eq:tildechi-def} the natural endomorphism $\phi_M \in \End(\Id_\Cc)$ assigned to each $M \in \Cc$. Note that $\phi_M$ is uniquely determined by $(\phi_M)_G \in E$, and that in fact $(\phi_M)_G \in Z(E)$.

For a given $M \in \Cc$ we can now define two central forms on $E$. Let $x \in E$. The first central form uses the Hom-tensor equivalence \eqref{eq:C-modE-equiv} and is given by $x \mapsto tr_{\Cc(G,M)}(x)$, the trace in the right $E$-module $\Cc(G,M)$ over the right action of $x$. The second form is 
$x \mapsto t_G((\phi_M)_G \circ x)$. The aim of this section is to prove our second main result:

\begin{theorem}\label{thm:trace-vs-tG}
The modified trace $t_{P_\one}((\phi_\one)_{P_\one})$ is non-zero,
 and
for all $M \in \Cc$ the following equality of central forms on $E$ holds:
\be\label{eq:trace-vs-tG}
	tr_{\Cc(G,M)}(-) ~=~ \frac{t_G((\phi_M)_G \circ - )}{t_{P_\one}((\phi_\one)_{P_\one})} \ .
\ee
\end{theorem}

The proof relies on several intermediate results, which we present now.

\begin{lemma}\label{lem:transp-morph}
Let $f: A \to B$ and suppose that for all $X \in \Cc$ we have
\be\label{eq:transp-morph}
	(f \otimes \id_X) \circ c_{X,A} \circ c_{A,X} = f \otimes \id_X \ .
\ee
Then there is $m>0$ and $a : A \to \one^{\oplus m}$, $b:  \one^{\oplus m} \to B$ such that $f = b \circ a$.
\end{lemma}

\begin{proof}
Write $f = \big[A \xrightarrow{a} M \xrightarrow{b} B \big]$, where $a$ is epi and $b$ is mono. Rewrite \eqref{eq:transp-morph} as 
\be
	(b \otimes \id_X) \circ c_{X,M} \circ c_{M,X} \circ (a \otimes \id_X) = (b \otimes \id_X) \circ (a \otimes \id_X) \ .
\ee
By exactness of $\otimes$, also $a \otimes \id_X$ is epi and $b \otimes \id_X$ is mono. 
We conclude that $c_{X,M} \circ c_{M,X} = \id_{M \otimes X}$ for all $M$. 
{}From characterisation 1 of factorisability in Theorem~\ref{thm:factequiv}
it follows that $M \cong \one^{\oplus m}$ for some $m$.
\end{proof}

Recall the coend $\coend$
with its family $\iota_X$ of dinatural transformations
and its 
	non-zero
cointegral~$\coint_\coend$ from Section~\ref{sec:conventions}.

\begin{lemma}\label{lem:lambdaLoiota_X}
For each $X \in \Cc$ there exist $m>0$, 
$a : X \to \one^{\oplus m}$, $b:  \one^{\oplus m} \to X$ (possibly zero) such that
\begin{align}
	\coint_\coend \circ \iota_X &~=~ \ev_X \circ \big(\id_{X^*} \otimes (b \circ a)\big)
	\nonumber \\
	&~=~ \sum_{\alpha=1}^m 
	\big[ X^*X \xrightarrow{(b_\alpha)^* \otimes a_\alpha} \one^*\one
	\xrightarrow{\ev_\one} \one \big]
\end{align}
where $a_\alpha : X \to \one$, $b_\alpha:  \one \to X$ are the components of $a,b$.
\end{lemma}

\begin{proof}
By definition of the cointegral, we have 
\be
	\big[ \coend \xrightarrow{\Delta_\coend} \coend\coend \xrightarrow{\coint_\coend \otimes \id_\coend} \one \coend  \xrightarrow{\sim} \coend \big]
	~=~ \big[ \coend \xrightarrow{\coint_\coend} \one \xrightarrow{\eta_\coend} \coend \big] \ .
\ee
Combining this with the
pairing 
$\omega_\coend$ we obtain the following equality:
\be
	\big[ \coend\coend \xrightarrow{\Delta_\coend \otimes \id_\coend} (\coend\coend)\coend \xrightarrow{\sim} 
	\coend(\coend\coend) \xrightarrow{\coint_\coend \otimes \omega_\coend} \one \one \big]
	~=~ \big[ \coend\coend \xrightarrow{\coint_\coend \otimes \eps_\coend} \one\one \big] \ ,
\ee
where we used the property of the Hopf pairing that 
$\omega_\coend \circ (\eta_\coend\otimes \id_\coend) \circ \lambda_\coend^{-1} = \varepsilon_\coend$.
After pre-composing with $\iota_X \otimes \iota_Y$ and substituting the defining equations, a short calculation shows
\begin{align}
&
\big[ (X^*X)(Y^*Y) 
\xrightarrow{\sim} X^*((XY^*)Y)
\xrightarrow{\id \otimes (c_{Y^*,X} \circ c_{X,Y^*}) \otimes \id} X^*((XY^*)Y)
\nonumber\\
&
\hspace{15em} \xrightarrow{\sim} (X^*X)(Y^*Y) 
\xrightarrow{(\coint_\coend \circ \iota_X) \otimes \ev_Y} \one\one
\big]
\nonumber\\
&
=~
\big[ (X^*X)(Y^*Y) 
\xrightarrow{(\coint_\coend \circ \iota_X) \otimes \ev_Y} \one\one
\big]
\end{align}
Using the isomorphism $\Cc(X^*\otimes M,N) \xrightarrow{\sim}\Cc(M,X\otimes N)$  ,
\be
	h ~\longmapsto~
	\big[\, M \xrightarrow{\sim} \one M \xrightarrow{\coev_X \otimes \id} (XX^*)M \xrightarrow{\sim} X(X^*M) \xrightarrow{\id \otimes h} XN \, \big] \ ,
\ee
and similarly for $\Cc(M\otimes Y,N)\xrightarrow{\sim} \Cc(M,N\otimes Y^*)$,  the latter equality can be brought to the form 
 $(f \otimes \id_{Y^*}) \circ c_{Y^*,X} \circ c_{X,Y^*} = f \otimes \id_{Y^*}$ 
with
\be
	f ~=~
	\big[\, X \xrightarrow{\sim} \one X 
	\xrightarrow{\coev_X \otimes \id} (XX^*)X
	\xrightarrow{\sim}  X(X^*X)
	\xrightarrow{\id \otimes (\coint_\coend \circ \iota_X)  } X \one 
	\xrightarrow{\sim} X \,\big] \ .
\ee
Lemma~\ref{lem:transp-morph} now implies the claim.
\end{proof}

\begin{corollary}\label{cor:lambda-iota}
For $U \in \Cc$ simple, $U \ncong \one$, we have $\coint_\coend \circ \iota_{P_U}=0$. Furthermore, there exists a unique
 $\nu_\one : \one \to P_\one$ 
such that
\be\label{eq:lambda-iota}
	\coint_\coend \circ \iota_{P_\one} = \ev_\one \circ (\nu_\one^* \otimes \pi_\one) \ ,
\ee
where $\pi_\one : P_\one \to \one$ is the canonical projection of the projective cover.
\end{corollary}

\begin{proof}
This follows from Lemma~\ref{lem:lambdaLoiota_X} and the fact that 
$\Cc(P_U,\one)$ and $\Cc(\one,P_U)$ are both zero-dimensional for $U \ncong \one$ and one-dimensional for $U \cong \one$
(by unimodularity of~$\Cc$).
\end{proof}

Recall the natural endomorphism $\phi_M$ in~\eqref{eq:phi_M-explicit}.
The next proposition computes its values on projective objects
and also shows that $\nu_\one$ (and all $\nu_U$ defined there) are non-zero.

\begin{proposition}\label{prop:phi_U_PV}
Let $U,V \in \Irr(\Cc)$. Then  $(\phi_U)_{P_U} \neq 0$ and there exists a unique non-zero
$\nu_U : U \to P_U$ such that
\be\label{eq:phi_U_PV}
	(\phi_U)_{P_V}
	= \begin{cases} 0 &, \quad V \neq U \\
		\nu_U \circ \pi_U &, \quad  V = U 
		\end{cases}
\ee
where $\pi_U : P_U \to U$ is the canonical projection of the projective cover, and where $\nu_\one$ is the same as in Corollary~\ref{cor:lambda-iota}.
Moreover,
there exists a choice of $a: P_\one\to  P_U\otimes U^*$, $b: P_U\otimes U^* \to P_\one$ which realise $P_\one$ as a direct summand of $P_U\otimes U^*$ (i.e.\ $b\circ a  = \id_{P_\one}$), such that
\be\label{eq:phi_U_PV-diag}
	\pi_U 
	~= 
  \raisebox{-0.5\height}{\setlength{\unitlength}{.75pt}
  \begin{picture}(70,140)
   \put(0,0){\scalebox{.75}{\includegraphics{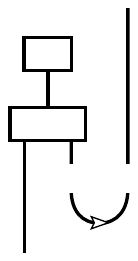}}}
   \put(0,0){
     \put(0,0){
     \put (8,0) {\scriptsize$ P_U $}
     \put (25,87) {\scriptsize$ P_\one $}
     \put (30,44) {\scriptsize$ U^* $}
     \put (58,44) {\scriptsize$ U $}
     \put (58,133) {\scriptsize$ U $}
     \put (20,71) {\scriptsize$ b $}
     \put (18,105) {\scriptsize$ \pi_\one $}
     }\setlength{\unitlength}{1pt}}
  \end{picture}}
  	\qquad,\qquad
	\nu_U 
	~= 
  \raisebox{-0.5\height}{\setlength{\unitlength}{.75pt}
  \begin{picture}(70,140)
   \put(0,0){\scalebox{.75}{\includegraphics{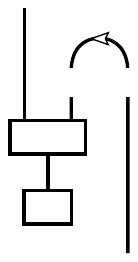}}}
   \put(0,0){
     \put(0,0){
     \put (8,133) {\scriptsize$ P_U $}
     \put (25,48) {\scriptsize$ P_\one $}
     \put (31,91) {\scriptsize$ U^* $}
     \put (59,91) {\scriptsize$ U $}
     \put (59,0) {\scriptsize$ U $}
     \put (20,65) {\scriptsize$ a $}
     \put (18,31) {\scriptsize$ \nu_\one $}
     }\setlength{\unitlength}{1pt}}
  \end{picture}}	\qquad .
\ee	
\end{proposition}

\begin{proof}
In the expression \eqref{eq:phi_M-explicit} for $(\phi_M)_X$  set $M=U$ and $X = P_V$. Write 
	$P_V \otimes  U^* \cong \bigoplus_{W \in \Irr(\Cc)} P_W^{\,\oplus n_W}$ 
and pick 
 $a_{W,\alpha} \in \Cc(P_W ,  P_V \otimes U^* )$ and $b_{W,\alpha} \in \Cc(P_V \otimes U^*, P_W)$,
  $\alpha = 1,\dots,n_W$, 
 which realise the direct sum decomposition:
\be
b_{W,\alpha} \circ a_{W',\alpha'} = \delta_{W,W'} \, \delta_{\alpha,\alpha'} \, \id_{P_W}
\quad,\quad
\sum_{W \in \Irr(\Cc)} \sum_{\alpha=1}^{n_W}  a_{W,\alpha} \circ b_{W,\alpha} =
 \id_{P_V \otimes U^*} \ .
\ee
For the dinatural transformation $\iota_{P_V \otimes U^*}$ in \eqref{eq:phi_M-explicit} consider the equalities
\begin{align}
\iota_{P_V \otimes U^*}
&=
\sum_{W \in \Irr(\Cc)} \sum_{\alpha=1}^{n_W} 
\iota_{P_V \otimes U^*} \circ (\id_{(P_V \otimes U^*)^*} \otimes (
a_{W,\alpha} \circ b_{W,\alpha}))
\nonumber \\
&=
\sum_{W \in \Irr(\Cc)} \sum_{\alpha=1}^{n_W} 
\iota_{P_W} \circ (a_{W,\alpha}^* \otimes b_{W,\alpha})
\label{eq:phi_U_PV-aux2}
\end{align}
{}From 
Corollary~\ref{cor:lambda-iota}, $\coint_\coend \circ \iota_{P_W} = 0$ for $W \ncong \one$, 
	and so
in the expression for $\coint_\coend  \circ \iota_{P_V \otimes U^*}$
  the sum over $W$ in \eqref{eq:phi_U_PV-aux2} reduces to $W=\one$. 
Since $\dim \Cc(P_W,R) = \delta_{W,R}$ for $W,R \in \Irr(\Cc)$, the multiplicities $n_W$ are given by $n_W = \dim \Cc(P_V \otimes U^*,W)$. In particular,
\be
n_\one = \dim \Cc(P_V \otimes U^*,\one) = \dim \Cc(P_V,U) = \delta_{U,V}\ .
\ee
 So in \eqref{eq:phi_U_PV-aux2} the sum over $\alpha$ is zero for $U \neq V$ and reduces to $\alpha=1$ for $U=V$. Together with \eqref{eq:lambda-iota} we obtain
\be
\coint_\coend \circ \iota_{P_V \otimes U^*}
~=~ \delta_{U,V}
\,
\ev_\one \circ 
\big(
 (a_{\one,1} \circ \nu_\one)^* \otimes (\pi_\one \circ b_{\one,1}) \big) \ .
\ee
Substituting this into \eqref{eq:phi_M-explicit} yields 
$(\phi_U)_{P_V}	= \delta_{U,V} \, \tilde\nu \circ \tilde\pi$
where
\be\label{eq:phi_U_PV-aux1}
	\tilde\nu 
	~= 
  \raisebox{-0.5\height}{\setlength{\unitlength}{.75pt}
  \begin{picture}(70,140)
   \put(0,0){\scalebox{.75}{\includegraphics{pics/pic02a.pdf}}}
   \put(0,0){
     \put(0,0){
     \put (8,133) {\scriptsize$ P_U $}
     \put (25,48) {\scriptsize$ P_\one $}
     \put (31,91) {\scriptsize$ U^* $}
     \put (59,91) {\scriptsize$ U $}
     \put (59,0) {\scriptsize$ U $}
     \put (14,65) {\scriptsize$ a_{\one,1} $}
     \put (18,31) {\scriptsize$ \nu_\one $}
     }\setlength{\unitlength}{1pt}}
  \end{picture}}
	\qquad,\qquad
	\tilde\pi 
	~= 
  \raisebox{-0.5\height}{\setlength{\unitlength}{.75pt}
  \begin{picture}(70,140)
   \put(0,0){\scalebox{.75}{\includegraphics{pics/pic02b.pdf}}}
   \put(0,0){
     \put(0,0){
     \put (8,0) {\scriptsize$ P_U $}
     \put (25,87) {\scriptsize$ P_\one $}
     \put (30,44) {\scriptsize$ U^* $}
     \put (58,44) {\scriptsize$ U $}
     \put (58,133) {\scriptsize$ U $}
     \put (15,71) {\scriptsize$ b_{\one,1} $}
     \put (18,105) {\scriptsize$ \pi_\one $}
     }\setlength{\unitlength}{1pt}}
  \end{picture}}
	\qquad .
\ee	

It remains to show that we can achieve $\tilde\pi = \pi_U$, the chosen projection $P_U \to U$, and that for $U=\one$, $\tilde\nu$ agrees with $\nu_\one$ in \eqref{eq:lambda-iota}.

By 	Theorem~\ref{thm:chiinj},
in particular $\chi_U \neq 0$.
Since $\psi$ and $\rho$ are isomorphisms, also $\phi_U \neq 0$, and thus $\tilde\nu, \tilde\pi \neq 0$
	(or otherwise $\phi_U$ would vanish on all indecomposable
	projectives and hence be zero).
Since $\Cc(P_U,U)$ is one-dimensional, 
by rescaling $b_{\one,1}$ if necessary (and thus rescaling  $a_{\one,1}$ accordingly by the inverse factor) one can achieve 
$\tilde\pi = \pi_U$. 
This finally gives~\eqref{eq:phi_U_PV-diag}.

For $U = V = \one$, one can alternatively compute $(\phi_\one)_{P_\one}$ directly from \eqref{eq:phi_M-explicit} and \eqref{eq:lambda-iota}. A short calculation confirms that $\nu_\one$ from \eqref{eq:lambda-iota} agrees with $\tilde\nu$.
\end{proof}

\begin{corollary}\label{cor:phi_U_V}
Let $U,V \in \Irr(\Cc)$, then $(\phi_U)_V  = 0$ unless $U=V$ and $U$ is projective.
\end{corollary}
\begin{proof}
 By naturality of $\phi_U$ we have $\pi_V\circ(\phi_U)_{P_V} = (\phi_U)_V\circ \pi_V$.  Applying~\eqref{eq:phi_U_PV} to LHS of the last equality and using $\pi_V \circ \nu_V = 0$ for $V$ non-projective, we get  the statement.
\end{proof}

\begin{proof}[Proof of Theorem~\ref{thm:trace-vs-tG}]
 We will show
  $t_G((\phi_M)_G \circ - ) = z \cdot tr_{\Cc(G,M)}(-)$
  and then determine $z \in k$ to be $t_{P_\one}((\phi_\one)_{P_\one})$.
Write $G = \bigoplus_{U \in \Irr(\Cc)} P_U^{\oplus n_U}$ and let $e_{U,\alpha}$, $U \in \Irr(\Cc)$, $\alpha = 1,\dots,n_U$ be the corresponding primitive orthogonal idempotents. 
Pick
$j_{U,\alpha} : P_U \to G$, $q_{U,\alpha} : G \to P_U$ such that 
\be\label{eq:qj-e}
q_{U,\alpha}  \circ j_{M,\beta} = \delta_{U,M}\,\delta_{\alpha,\beta} \, \id_{P_U} \qquad  \text{and} \qquad
j_{U,\alpha}  \circ q_{U,\alpha} = e_{U,\alpha}   \  .
\ee
As for any finite-dimensional associative algebra, the quotient $E/ \Jac(E)$ is a direct sum over 
	$U\in\Irr(\Cc)$ 
of $n_U{\times}n_U$-matrix algebras, and 
	the
$j_{U,\alpha} \circ q_{U,\beta}$ form a basis in these matrix algebras.
Therefore, the endomorphism algebra $E$ decomposes (as
a vector space) as
\be
	E ~=~ \bigoplus_{U,\alpha,\beta} k\,  j_{U,\alpha} \circ q_{U,\beta} ~\oplus~ \Jac(E) \ .
\ee

To show
 $t_G((\phi_M)_G \circ - ) = z \cdot \mathrm{tr}_{\Cc(G,M)}(-)$
  it is enough to consider $M \in \Irr(\Cc)$. 
In this case, $\tilde M := \Cc(G,M)$ is an irreducible $E$-module of dimension $n_M$. 
The trace $\mathrm{tr}_{\tilde M}(-)$ vanishes on $\Jac(E)$ while on semisimple part it gives
\be\label{eq:trace-vs-tG-aux1}
	\mathrm{tr}_{\tilde M}(j_{U,\alpha} \circ q_{U,\beta}) = \delta_{U,M} \,\delta_{\alpha,\beta} \ .
\ee

	It remains to check that evaluating $t_G((\phi_M)_G \circ - )$ on $j_{U,\alpha} \circ q_{U,\beta}$ leads to the same result, up to a factor of $t_{P_\one}((\phi_\one)_{P_\one})$. 
	We have, for all $U \in \Irr(\Cc)$,
\begin{align}
	(\phi_U)_G
	&\overset{\eqref{eq:qj-e}}=
	\sum_{V \in \Irr(\Cc)} \sum_{\alpha=1}^{n_V} 
	j_{V,\alpha} \circ q_{V,\alpha} \circ (\phi_U)_G 
	\overset{\text{$\phi_U$ nat.}}=
	\sum_{V \in \Irr(\Cc)} \sum_{\alpha=1}^{n_V} 
	j_{V,\alpha} \circ (\phi_U)_{P_V} \circ q_{V,\alpha}  
\nonumber\\
	&\overset{\text{Prop.~\ref{prop:phi_U_PV}}}=~
		\sum_{\alpha=1}^{n_U} 
	j_{U,\alpha} \circ \nu_U\circ \pi_U \circ q_{U,\alpha}   \ .
\label{eq:trace-vs-tG-aux2}
\end{align}
By~\eqref{eq:JG-Jac} the map $\pi_U\circ q_{U,\alpha}\colon G\to U$ annihilates $\Jac(E)$ and therefore
 $(\phi_U)_G \circ \Jac(E) = 0$.
Thus also $t_G((\phi_M)_G \circ - )$ vanishes on $\Jac(E)$. Furthermore,
\be\label{eq:trace-vs-tG-aux3}
t_G((\phi_M)_G \circ j_{U,\alpha} \circ q_{U,\beta} ) = \delta_{U,M}\,\delta_{\alpha,\beta} \, z
\quad ,  \quad \text{where} \quad 
z = t_{P_M}((\phi_M)_{P_M}) \ ,
\ee
and where we used  first cyclicity of $t$, then naturality of $\phi_M$ and then~\eqref{eq:qj-e} and~\eqref{eq:phi_U_PV}.
To arrive at $z = t_{P_\one}((\phi_\one)_{P_\one})$,
	we first observe that
\be\label{eq:proof-trace-vs-tG-aux}
 (\phi_M)_{P_M}  ~= 
  \raisebox{-0.5\height}{\setlength{\unitlength}{.75pt}
  \begin{picture}(70,251)
   \put(0,0){\scalebox{.75}{\includegraphics{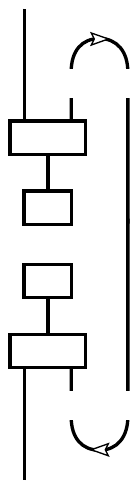}}}
   \put(0,0){
     \put(0,109){
     \put (8,135) {\scriptsize$ P_M $}
     \put (25,48) {\scriptsize$ P_\one $}
     \put (31,91) {\scriptsize$ M^* $}
     \put (59,91) {\scriptsize$ M^{**} $}
	     \put (20,65) {\scriptsize$ a $}
     \put (18,31) {\scriptsize$ \nu_\one $}
     }
     \put(0,0){
     \put (8,0) {\scriptsize$ P_M $}
     \put (25,87) {\scriptsize$ P_\one $}
     \put (30,44) {\scriptsize$ M^* $}
     \put (58,44) {\scriptsize$ M^{**} $}
	     \put (20,70) {\scriptsize$ b $}
     \put (18,105) {\scriptsize$ \pi_\one $}
     }
     \setlength{\unitlength}{1pt}}
  \end{picture}}
  \qquad .
\ee
	To see this, first use \eqref{eq:phi_U_PV} and substitute the explicit expressions for $\nu_M, \pi_M$ from \eqref{eq:phi_U_PV-diag}. 
Next use the pivotal structure to insert $\id_M = \big[ M \xrightarrow{\delta_M} M^{**} \xrightarrow{\delta_M^{-1}} M \big]$ 
and combine $\delta, \delta^{-1}$ with the evaluation and coevaluation maps as in~\eqref{eq:ev-coev-tilde}, thereby replacing $M$ by $M^{**}$ in the right-most strand of the above diagram. We can now compute
\begin{align}
	t_{P_M}\big( (\phi_M)_{P_M} \big)
	&	\overset{\text{(*)}}{=} 
	t_{P_M \otimes M^*}\big( a \circ \nu_\one \circ \pi_\one \circ b \big)
	\nonumber\\&
	\overset{\text{cycl.}}{=} 
	t_{P_\one}\big( \nu_\one \circ \pi_\one \circ b \circ a \big)
	\overset{\eqref{eq:phi_U_PV}}{=}
	t_{P_\one}((\phi_\one)_{P_\one}) \ ,
	\label{eq:tPMphiM=tP1phi1}
\end{align}
where in (*) we have rewritten \eqref{eq:proof-trace-vs-tG-aux} as a partial trace and used compatibility of the modified trace with partial traces
in~\eqref{eq:mod-tr_partial-trace}.

So far we proved 
$t_G((\phi_M)_G \circ - ) = z \cdot tr_{\tilde M}(-)$ with $z =t_{P_\one}((\phi_\one)_{P_\one})$. It remains to show that $z \neq 0$. 
By Theorem \ref{thm:chiinj} and \eqref{eq:tildechi-def}, $\phi_M \neq 0$ for $M \in \Irr(\Cc)$. Hence also $(\phi_M)_G \neq 0$. By non-degeneracy of the pairings \eqref{eq:pairingMP} (Corollary \ref{cor:factorisable-implies-trace}), there is an $f \in E$ such that $t_G((\phi_M)_G \circ f ) \neq 0$. In particular, $z \neq 0$.
\end{proof}

As a corollary to the proof of Theorem~\ref{thm:trace-vs-tG}, specifically 
	\eqref{eq:tPMphiM=tP1phi1} (using also Proposition~\ref{prop:phi_U_PV}),
 we obtain that, for $A,B \in \Irr(\Cc)$,
\be\label{eq:mod-trace-on-phiA}
	t_{P_A}\big(\, (\phi_B)_{P_A} \,\big) ~=~ \delta_{A,B}  \cdot t_{P_\one}((\phi_\one)_{P_\one}) \ . 
\ee

\section{Properties of the Reynolds and Higman ideals}

In this section $k$, $\Cc$, $t$, $G$ and $E$ have the same meaning as in the beginning of Section~\ref{sec:ch-tr}.
We give an
alternative description of the Reynolds and Higman ideals defined in Section~\ref{sec:sym-alg-ideal} in terms of the $\phi_U$. Then we describe the multiplication in these ideals.

\medskip

We start with the alternative  description of $\Rey(\Cc)$ and $\Hig(\Cc)$.

\begin{proposition}\label{prop:ReyHig=span}
We have
\be
\Rey(\Cc) = \mathrm{span}_k\big\{\,\phi_M \,\big|\, M\in\Cc \, \big\}
\quad , \quad
\Hig(\Cc) = \mathrm{span}_k\big\{\,\phi_P \,\big|\, P\in\Proj(\Cc) \, \big\} \ .
\ee
\end{proposition}

\begin{proof}
Abbreviate $S= \mathrm{span}_k\big\{\,(\phi_M)_G \,\big|\, M\in\Cc \, \big\}$ and $S'= \mathrm{span}_k\big\{\,(\phi_P)_G \,\big|\, P\in\Proj(\Cc) \, \big\}$.
By Propositions~\ref{prop:Rey(Amod)=Rey(A)} and \ref{prop:HigA-HigAmod} it is enough to show 
$\Rey(E) = S$
and
$\Hig(E) = S'$.

The image of $S$ under the map $\zeta$ from \eqref{eq:zeta:ZA-CA} 
(we take $\varepsilon=t_G$)
 is given by
$\zeta(S) = \mathrm{span}_k\big\{\,t_G
\big((\phi_M)_G\circ-\big) 
\,\big|\, M\in\Cc \, \big\}$. Theorem~\ref{thm:trace-vs-tG} allows us to rewrite this as
\be
\zeta(S) = \mathrm{span}_k\big\{\,tr_{\Cc(G,M)}(-) \,\big|\, M\in\Cc \, \big\} = R(E)\ ,
\ee
 where for the last equality we used the definition of $R(E)$ below \eqref{eq:ZA-CA-subsets} and that every $E$-module is isomorphic to $\Cc(G,M)$ for some $M$ (equivalence \eqref{eq:C-modE-equiv}). But by the isomorphisms in \eqref{eq:ZA-CA-subsets} the equality $\zeta(S)=R(E)$ implies $S=\Rey(E)$.

An analogous argument shows $S' = \Hig(E)$.
\end{proof}

The next proposition shows that the action of $\End(\Cc)$ on $\Rey(\Cc)$ is diagonal in the basis  
 $\{\, \phi_{U} \,|\, U \in \Irr(\Cc)\, \}$. 

\begin{proposition}\label{prop:phiU-End-diagonal}
For $\alpha \in \End(\Id_\Cc)$ and $U \in \Irr(\Cc)$ 
we have
\be\label{eq:phiU-End-diagonal}
	\alpha \circ \phi_{U}
	~=~
	\langle\alpha_U\rangle \cdot \phi_{U} \ ,
\ee 
where 
$\langle\alpha_U\rangle \in k$ is determined by $\alpha_U = \langle\alpha_U\rangle\,\id_U$. If $U,V \in \Irr(\Cc)$ lie in the same block of~$\Cc$, then $\langle\alpha_U\rangle = \langle\alpha_V\rangle$.
\end{proposition}

\begin{proof}
We pick a projective generator $G$ and use the equivalence in \eqref{eq:C-modE-equiv}. 
Let $e_1,\dots,e_b \in Z(E)$ be the primitive central idempotents. The $e_j$ decompose $E$ as an algebra into indecomposable direct summands, or, equivalently, via \eqref{eq:C-modE-equiv} they decompose $\Cc$ into blocks.

Abbreviate $z := \alpha_G \in Z(E)$. We can write 
	$z = \sum_{j=1}^b y_j e_j + n$ for some $y_j \in k$, $n \in \Jac(E)$. 
Suppose $U$ is in the $j$'th block of $\Cc$, that is, $z$ acts on the simple $E$-module $\Cc(G,U)$ by $y_j \,\id$. Then $\langle\alpha_U\rangle = y_j$, proving the second statement.

By Proposition~\ref{prop:phi_U_PV}, $(\phi_U)_G$ factors through the semisimple $E$-module $\Cc(G,U)^{\oplus m}$ for some $m$. Hence $z \,(\phi_U)_G = y_j \, (\phi_U)_G$, proving the first statement.
\end{proof}

\begin{remark}\label{rem:also-PU-diag}
Together with Proposition~\ref{prop:ReyHig=span},
Proposition~\ref{prop:phiU-End-diagonal} states that $\End(\Id_\Cc)$ acts diagonalisably on 
$\Rey(\Cc)$, with eigenbasis given by $\{ \phi_U \,|\, U \in \Irr(\Cc) \}$. 
Since the value of 
$\alpha \in \End(\Id_\Cc)$ on $\phi_U$ only depends on the block of $U$, and since all composition factors of an indecomposable projective belong to the same block, we equally have
\be\label{eq:phiPU-End-diagonal}
	\alpha \circ \phi_{P_U}
	~=~
	\langle\alpha_U\rangle \cdot \phi_{P_U} \ ,
\ee 
with $\langle\alpha_U\rangle$ as in Proposition~\ref{prop:phiU-End-diagonal}.
\end{remark}

By Proposition~\ref{prop:ReyHig=span}, $\Hig(\Cc)$ is spanned by the $\phi_P$, $P \in \Proj(\Cc)$. From this formulation, we see that $\Hig(\Cc)$ contains a natural subspace, namely the span of all $\phi_Q$ for $Q$ is both simple and projective.\footnote{
  Via the equivalence \eqref{eq:C-modE-equiv} to symmetric algebras, this subspace corresponds to the 
	subspace 
of the centre denoted by $Z_0$ in \cite{Cohen:2008}.} 
For each simple projective object $Q$, by
	Proposition~\ref{prop:phi_U_PV}
and since $\phi_Q \neq 0$, there is a non-zero $b_Q \in k$ such that
\be\label{eq:phi_U_PV_simproj}
	(\phi_Q)_{P_V}
	= \begin{cases} 0 &
	,\qquad
	 P_V \ncong Q \\
		b_Q \cdot \id_{P_V} &
		,\qquad
		 P_V \cong Q 
		\end{cases}
	\qquad .
\ee
	This allows for another expression of the normalisation constant $t_{P_\one}((\phi_\one)_{P_\one})$ in \eqref{eq:trace-vs-tG},
namely, for $Q$ a simple projective object,
\be\label{eq:tP1phiP1-via-tQidQ}
	t_{P_\one}\big((\phi_\one)_{P_\one}\big) 
	\overset{\eqref{eq:tPMphiM=tP1phi1}}=
	t_Q\big((\phi_Q)_Q\big)
	~=~ 
	b_Q \cdot t_Q(\id_Q) \ .
\ee
For later use we define
\be\label{eq:IrrProj-def}
	\mathrm{IrrProj}(\Cc) ~=~ \{\, U \in \Irr(\Cc) \,|\,
	\text{$U$ is projective } \} \ .
\ee
By Theorem~\ref{thm:main-simpleproj}, this set is not empty.
We can use this notation to describe the product in $\Rey(\Cc)$ and $\Hig(\Cc)$:

\begin{proposition}\label{prop:prod-Rey-Hig}
We have, for $U,V \in \Irr(\Cc)$,
\be
\phi_U \circ \phi_V
~=~
	\delta_{U,V} \, 
	\delta_{V \in \mathrm{IrrProj}(\Cc)} \,
	b_V  \, \phi_{V} 
\quad , \quad
	\phi_{P_U} 
	\circ
	\phi_{P_V}
	~=~
	\delta_{U,V} \, 
	\delta_{V \in \mathrm{IrrProj}(\Cc)} \,
	b_V  \, \phi_{P_V} \ .
\ee
\end{proposition}

\begin{proof}
By Proposition~\ref{prop:phiU-End-diagonal}, $\phi_U \circ \phi_V = \langle (\phi_U)_V \rangle \cdot \phi_V$.
By Corollary~\ref{cor:phi_U_V}, $\langle (\phi_U)_V \rangle = 0$ unless $U=V$ and $U$ is projective.
The definition of $b_Q$ in \eqref{eq:phi_U_PV_simproj} now
	implies
the first equality. 
For the second equality,
expand $\phi_{P_U} = \sum_{W \in \Irr(\Cc)} [P_U:W] \phi_W$, and similar for $\phi_{P_V}$. Since $P_U$ can only contain a composition factor from $\mathrm{IrrProj}(\Cc)$ if already $U \in \mathrm{IrrProj}(\Cc)$, the second equality follows from the first.
\end{proof}

\section{$\boldsymbol{S}$-invariance of projective characters}\label{sec:Sinv-proj}

In this section $k$, $\Cc$, $t$, $G$ and $E$ have the same meaning as in the beginning of Section~\ref{sec:ch-tr}.
Here we prove our third main result, namely that the Higman ideal of $\Cc$ gets mapped to itself under the action of $\mathscr{S}_\Cc$.
As applications, we 
investigate the action of the Grothendieck ring on $\Rey(\Cc)$ and $\Hig(\Cc)$, then give a variant of the Verlinde formula for projective objects and we show how matrix coefficients of $\mathscr{S}_\Cc$ on $\Hig(\Cc)$ are related to
the  modified trace.

\medskip

Recall the definition of $\smap : \Gr(\Cc)_k \to \End(\Id_\Cc)$ from 
\eqref{eq:sigma-Gr-EndId-def}.

\begin{lemma}\label{lem:sigP-ann-Jac}
Let $P \in \Cc$ be projective and let $f \in \Jac(E)$. Then $\smap([P])_G \circ f = 0$.
\end{lemma}

\begin{proof}
Since $t_G$ induces a non-degenerate pairing on $E$ (Corollary~\ref{cor:factorisable-implies-trace}), it is enough to show that for all $x \in E$, 
$t_G\big(\smap([P])_G \circ f \circ x\big) = 0$.
Combining the explicit expression of $\smap([P])_G$ as a partial trace in~\eqref{eq:eq:sigma-Gr-EndId-def_pic}
and  Lemma~\ref{lem:exchange-partial-trace} 
we see
\be\label{eq:sigP-ann-Jac-aux1}
  t_G\big(\smap([P])_G \circ f \circ x\big) 
  = 
  t_{P^*}\Big( tr^r_G\big(c_{G,P^*} \circ ((f\circ x) \otimes \id_{P^*}) 
  \circ c_{P^*,G}\big) \Big) \ .
\ee
As $f \in \Jac(E) = J_G$ (recall \eqref{eq:JG-Jac}), for all $U \in \Cc$ simple and all $u : G \to U$, we have $u \circ f = 0$, and hence also $u \circ f \circ x = 0$.
Corollary~\ref{cor:dinat-jac-0} now shows that the right hand side of \eqref{eq:sigP-ann-Jac-aux1} is zero.
\end{proof}

\begin{corollary}\label{cor:sig(P)inReyC}
Let $P \in \Cc$ be projective. Then 
$\smap([P]) \in \Rey(\Cc)$.
\end{corollary}

For each $U \in \Irr(\Cc)$ choose a primitive idempotent $e_U \in E$ corresponding to the projective cover 
of the irreducible representation $\Cc(G,U)$ of $E$. That is, 
	$e_UE \cong \Cc(G,P_U)$
as right $E$-modules. We remark here that by \eqref{eq:trace-vs-tG-aux3}
we have for $U,V \in \Irr(\Cc)$,
\be\label{eq:phiU-eV-ortho}
	t_G((\phi_U)_G  \circ e_V) ~=~ z\, \delta_{U,V} \ ,
\ee
where we set $z = t_{P_\one}((\phi_\one)_{P_\one})\ne 0$ as above.
{}From  this equality and 
Proposition~\ref{prop:ReyHig=span} it
follows that the pairing
\be\label{eq:Rey-idemp-pairing}
\Rey(\Cc) \times \mathrm{span}_k\big\{e_V \,\big|\, V \in \Irr(\Cc) \big\} \longrightarrow k
\quad , \quad
(r,u) \mapsto t_G(r \circ u)
\ee
is non-degenerate.

Let us abbreviate 
	the image of the Cartan matrix $\CM(\Cc)$ in $\mathrm{Mat}_n(k)$ by $\widehat\CM$, cf.\ Corollary~\ref{cor:cartan-sym+rank_general}. For $x \in \ker(\widehat\CM)$ write 
$\tilde x = \sum_{U \in \Irr(\Cc)} x_U e_U \in E$. 
We define
\be
	K^\perp := \big\{ \,r \in \Rey(\Cc) \,\big|\, t_G
	(r_G \circ \tilde x) =0 \text{ for all }
	x \in \ker(\widehat\CM) \big\} \ .
\ee

\begin{lemma}\label{lem:Hig-Kperp}
$\Hig(\Cc) = K^\perp$.
\end{lemma}

\begin{proof}
We first show that $\Hig(\Cc) \subset K^\perp$. By Proposition~\ref{prop:ReyHig=span} we need to show that for all $P \in \Proj(\Cc)$
and $x \in \ker(\widehat\CM)$, 
	$t_G\big( (\phi_P)_G  \circ \tilde x\big)=0$. 
It is enough to consider the projective covers $P_U$, $U \in \Irr(\Cc)$. 
By \eqref{eq:composition-series-proj-cover} we have 
	$\phi_{P_U} = \sum_{V \in \Irr(\Cc)} \widehat\CM_{UV} \, \phi_V$. 
Using this, we compute
\begin{align}
t_G\big((\phi_{P_U})_G  \circ \tilde x\big)
&~=
\hspace{-.4em} \sum_{V,W \in \Irr(\Cc)} \hspace{-.8em} \widehat\CM_{UV} \, x_W \, t_G\big((\phi_V)_G  \circ e_W\big)
\nonumber \\ &
	\overset{\eqref{eq:phiU-eV-ortho}}=
\hspace{-.4em} \sum_{V,W \in \Irr(\Cc)} \hspace{-.8em} \widehat\CM_{UV}\,  x_W \, \delta_{V,W} \, z
\overset{x \in \ker(\widehat\CM)}= 
0 \ .
\end{align}
	By non-degeneracy of the pairing \eqref{eq:Rey-idemp-pairing},
we have $\dim_k K^\perp = \mathrm{rank}(\widehat\CM)$, the rank of the Cartan matrix,
	seen as a matrix over $k$. 
Hence by Corollary~\ref{cor:cartan-sym+rank_general}, $\dim_k K^\perp = \dim_k \Hig(\Cc)$.
\end{proof}

The following theorem generalises results of~\cite{Lachowska-center,Cohen:2008} as mentioned in the introduction.

\begin{theorem}\label{thm:projective-S-inv}
$\mathscr{S}_\Cc(\Hig(\Cc)) = \Hig(\Cc)$.
\end{theorem}

\begin{proof}
We will show $\mathscr{S}_\Cc(\Hig(\Cc)) \subset K^\perp$. 
	Let $P \in \Proj(\Cc)$. 
By Corollary~\ref{cor:sig(P)inReyC}
	and \eqref{eq:smap-via-SC-phiM},
$\mathscr{S}_\Cc(\phi_P) \in \Rey(\Cc)$. It remains to check that $\mathscr{S}_\Cc(\phi_P) \in K^\perp$, i.e.\ that $t_G\big((\mathscr{S}_\Cc(\phi_P))_G \circ \tilde x\big)=0$ for all 
$x \in \ker(\widehat\CM)$. 
We compute
\begin{align}
& t_G\big((\mathscr{S}_\Cc(\phi_P))_G \circ \tilde x\big)
\overset{\text{as in \eqref{eq:sigP-ann-Jac-aux1}}}=
  t_{P^*}\big( tr_G^r(c_{G,P^*} \circ (\tilde x \otimes \id_{P^*}) 
  \circ c_{P^*,G}) \big)
\nonumber \\ &
\overset{\text{(*)}}=
\sum_{U \in \Irr(\Cc)} x_U \, t_{P^*}\big( \sigma([P_{U^*}])_{P^*} \big)
\overset{\eqref{eq:composition-series-proj-cover}}=
\sum_{U,V \in \Irr(\Cc)} x_U \, \widehat\CM_{U^*V}\, t_{P^*}\big( \sigma([V])_{P^*} \big)
\nonumber \\ &
\overset{\text{(**)}}=
\sum_{U,V \in \Irr(\Cc)} \widehat\CM_{V^*U}\, x_U \, t_{P^*}\big( \sigma([V])_{P^*} \big)
\overset{x \in \ker(\widehat\CM)}=
0 \ .
\end{align}
In (*) we inserted $\tilde x = \sum_{U} x_U \, e_U$ and used
that $e_U=a\circ b$, for $a\colon P_U\to G$, $b\colon G\to P_U$ realising the direct summand $P_U$ in $G$,  
and moving $b$ around the loop we get  $b\circ a = \id_{P_U}$.
The reason that $U^*$ appears instead of $U$ is the extra dual in \eqref{eq:eq:sigma-Gr-EndId-def_pic},
  together with $(P_U)^* \cong P_{U^*}$ (as follows from uni-modularity of $\Cc$)
\footnote{Applying  \eqref{eq:eq:sigma-Gr-EndId-def_pic} for $\sigma(V^*)_{X}$, then  inserting $\delta_V \circ \delta_V^{-1}\colon V^{**}\to V\to V^{**}$ and  moving $\delta_V^{-1}$ once around the loop, we get a loop coloured by $V$ instead of $V^{**}$.}.
	This isomorphism, combined with \eqref{eq:cartan-mat-def}, shows $\CM_{U,V} = \CM_{V^*,U^*}$
	(and hence $\widehat\CM_{U,V} = \widehat\CM_{V^*,U^*}$),
which explains (**).

The inclusion $\mathscr{S}_\Cc(\Hig(\Cc)) \subset K^\perp$ implies the theorem since $\mathscr{S}_\Cc$ is an isomorphism 
	(Remark~\ref{rem:SC-iso-iff-fact})
and so together with Lemma~\ref{lem:Hig-Kperp} we have $\dim_k(\mathscr{S}_\Cc(\Hig(\Cc))) = \dim_k \Hig(\Cc) = \dim_k K^\perp$. Thus $\mathscr{S}_\Cc(\Hig(\Cc)) = K^\perp = \Hig(\Cc)$, once more appealing to Lemma~\ref{lem:Hig-Kperp}.
\end{proof}

Recall the projective $SL(2,\Zb)$-action on $\End(\Id_\Cc)$ from Theorem~\ref{thm:proj-sl2Z-EndId} in case $\Cc$ is in addition ribbon.

\begin{corollary}\label{cor:Hig-SL2Z-sub}
Let $\Cc$ be in addition ribbon. Then $\Hig(\Cc)$ is a submodule for the projective $SL(2,\Zb)$-action on $\End(\Id_\Cc)$.
\end{corollary}

\begin{proof}
By Theorem~\ref{thm:projective-S-inv}, $\Hig(\Cc)$ is stable under the action of the $S$-generator of $SL(2,\Zb)$.
The action of $T$ is given by $\theta \circ (-)$, and since $\Hig(\Cc)$ is an ideal, it is stable under the $T$-action as well.
\end{proof}

\medskip

We now give applications of Theorem~\ref{thm:projective-S-inv} to the
problem of diagonalisability of the
Grothendieck ring of $\Cc$ (Proposition \ref{prop:GrC-diagonalisability}), 
to the Verlinde formula (Proposition~\ref{prop:M-via-S}) 
and to the modified trace (Proposition~\ref{prop:mod-tr-S}). 
Namely, we ask how much can be learned about the tensor products of projective objects from knowing the restriction of $\mathscr{S}_\Cc$ to $\Hig(\Cc)$,
and give a formula for the modified trace of the ``open Hopf-link operator'' $\sigma([P])$ in terms of the $S$-matrix elements.
See e.g.\ \cite{Fuchs:2003yu,Fuchs:2006nx,Gainutdinov:2016qhz,Creutzig:2016fms} 
for discussions of the Verlinde formula in the finitely non-semisimple setting and for more references.

\medskip

We start with the application to the Grothendieck ring.
Recall from Corollary \ref{cor:inj-alg-hom} that $[M] \mapsto \smap([M])$ is an injective ring homomorphism 
from $\Gr(\Cc)$ to $\End(\Id_\Cc)$. For $M \in \Cc$ denote by 
\be
\rho^L([M])\colon\; \End(\Id_\Cc) \to \End(\Id_\Cc)
\ee
 the left multiplication by $\smap([M])$.
Then $\rho^L([A])  \rho^L([B]) 
= \rho^L([A]\cdot [B])$ and $\rho^L([\one]) = \id$, i.e.\ $\rho^L\colon \Gr(\Cc) \to \End_k(\End(\Id_\Cc))$ defines a 
	faithful
representation. Since $\Rey(\Cc)$ and $\Hig(\Cc)$ are ideals, they are submodules for the representation $\rho^L$.

We can now state the following result, part 1 of which is a generalisation of (part of) \cite[Thm.\,3.14]{Cohen:2008}.

\begin{proposition}\label{prop:GrC-diagonalisability}
Consider the action $\rho^L : \Gr(\Cc) \to \End_k(\End(\Id_\Cc))$.
\begin{enumerate}
\item 
The restriction of $\rho^L$ to the submodule $\Rey(\Cc)$, respectively $\Hig(\Cc)$, is diagonalisable
with action
\be
	\rho^L([M]) \, \phi_{V}
	\,=~
	\ell^M_V
	\, \phi_{V} \ ,
\qquad \text{resp.} \qquad
	\rho^L([M]) \, \phi_{P_V}
	\,=~
	\ell^M_V 
	\, \phi_{P_V} \ ,
\ee
where 
	$V \in \Irr(\Cc)$ and
	$\ell^M_V  := \big\langle \smap([M])_V \big\rangle \in k$ (see Proposition~\ref{prop:phiU-End-diagonal}).
\item If $\Cc$ is semisimple, $\rho^L$ can be diagonalised on $\End(\Id_\Cc)$.
\item If $\mathrm{char}(k)=0$ and $\rho^L$ can be diagonalised on $\End(\Id_\Cc)$, then $\Cc$ is semisimple.
\end{enumerate}
\end{proposition}

\begin{proof}
Part 1 is immediate from Proposition~\ref{prop:phiU-End-diagonal}
and Remark~\ref{rem:also-PU-diag}.
Part~2 follows from part~1 together with Proposition~\ref{prop:cy-ssi-inclusion}, which states that for $\Cc$ semisimple we have $\Hig(\Cc)=\End(\Id_\Cc)$.

Part 3 we prove by contradiction. Suppose that $\Cc$ is not semisimple, so that not every indecomposable projective object is simple. 
Let $P$ be such a non-simple indecomposable projective object. 
Since $\mathrm{char}(k)=0$, the map $\Gr(\Cc) \to \Gr_k(\Cc)$ is injective, and hence by Theorem \ref{thm:chiinj} (and the definition in \eqref{eq:tildechi-def}), $\phi_P \neq 0$.
Furthermore, by Proposition~\ref{prop:prod-Rey-Hig}, $\phi_P$ is nilpotent, in fact $\phi_P \circ \phi_P=0$. 

Since $\phi_P \in \Hig(\Cc)$ (Proposition~\ref{prop:ReyHig=span}), by Theorem~\ref{thm:projective-S-inv}, we can write $\phi_P = \mathscr{S}_\Cc^{-1}(\sigma([P]))  = \sum_{U \in \Irr(\Cc)} x_U \, \sigma([P_U])$ and this linear combination is nilpotent because $\phi_P \circ \phi_P = 0$. Then, since  $\sigma$ is an injective 
	algebra map
(Corollary~\ref{cor:inj-alg-hom}), $\Gr_k(\Cc)$ contains a non-zero nilpotent element, and so $\rho^L$ (which factors through $\Gr_k(\Cc)$) cannot be diagonalised.
\end{proof}
	
In particular, we have the following consequence for the linearised Grothendieck ring:

\begin{corollary}\label{cor:Grss-Css}
If $\mathrm{char}(k)=0$,
the linearised Grothendieck ring $\Gr_k(\Cc)$ is semisimple iff $\Cc$ is semisimple.
\end{corollary}

\medskip

Next we turn to an application of Theorem \ref{thm:projective-S-inv} to a Verlinde-type formula.

\medskip

Write $\{M\}$ for the isomorphism class of an object $M \in \Cc$ (rather than its class $[M]$ in $\Gr(\Cc)$). Define the ring $K_0(\Cc)$ as the free abelian group spanned by the isomorphism classes of indecomposable projective objects
	(see e.g.\ \cite[Sec.\,1.8]{EGNO-book}),
\be
K_0(\Cc) ~:= \bigoplus_{U \in \Irr(\Cc)} \Zb\, \{P_U\} \ .
\ee
The product of $\{P_U\}$ and $\{P_V\}$ is defined by the decomposition of $P_U \otimes P_V$ into indecomposable projectives.
	Thus $K_0(\Cc)$ encodes the tensor products of projective objects in $\Cc$ up to isomorphism. 
We denote the structure constants of $K_0(\Cc)$ by $M_{UV}^{~W}$, that is,
\be\label{eq:proj-M-def}
	\{P_U\} \{P_V\} ~= \sum_{W \in \Irr(\Cc)} M_{UV}^{~W} \, \{P_W\} \ ,
\ee
or, equivalently,
\be
	P_U \otimes P_V ~\cong \bigoplus_{W \in \Irr(\Cc)} (P_W)^{\oplus M_{UV}^{~W}} \ .
\ee
The map $K_0(\Cc) \to \Gr(\Cc)$, $\{P_U\} \mapsto [P_U]$ is a ring homomorphism. The matrix elements of this map are precisely the entries of the Cartan matrix of $\Cc$.

For the application to the Verlinde formula, 
we will need two bases of $\Hig(\Cc)$. For the first basis, choose a subset $J \subset \Irr(\Cc)$ such that
\be\label{eq:choice-Hig-basis1}
	\big\{\, \phi_{P_U} \,\big|\, U \in J \,\big\}
\ee
is a basis of $\Hig(\Cc)$. 
	Recall that $\widehat\CM$ denotes the image of $\CM(\Cc)$ in $\mathrm{Mat}_n(k)$.
By Corollary~\ref{cor:cartan-sym+rank_general} we have 
	$|J| = \mathrm{rank}(\widehat\CM)$. 
By Theorem~\ref{thm:projective-S-inv}, $\mathscr{S}_\Cc$ is an invertible endomorphism of $\Hig(\Cc)$. The second basis then is
\be\label{eq:choice-Hig-basis2}
	\big\{\, \mathscr{S}_\Cc(\phi_{P_U}) \,\big|\, U \in J \,\big\} \ .
\ee
Let $\tilde B$ be the $\Irr(\Cc) \times J$-matrix expressing $\phi_{P_U}$ in the first basis,
\be\label{eq:tildeB-expansion}
	\phi_{P_U}
	=
	\sum_{A \in J}
	\tilde B_{UA} \, \phi_{P_A} 
	\qquad , \quad \text{for} ~~U\in\Irr(\Cc) \ .
\ee
The change of basis map between \eqref{eq:choice-Hig-basis1} and \eqref{eq:choice-Hig-basis2} will be denoted by the $J \times J$-matrix~$\tilde S$, that is
\be\label{eq:tildeS-in-basis-def}
	\mathscr{S}_\Cc(\phi_{P_A})
	=
	\sum_{B \in J}
	\tilde S_{AB} \,\phi_{P_B} 
	\qquad , \quad \text{for} ~~A\in J \ .
\ee
Finally, let 
$\tilde C$ be the $J \times J$-matrix which describes the effect of taking duals,
\be\label{eq:tildeC-def}
	\phi_{(P_A)^*}
	=
	\sum_{B \in J}
	\tilde C_{AB}\,\phi_{P_B} 
	\qquad , \quad \text{for} ~~A\in J \ .
\ee
By construction, $\tilde S$ and $\tilde C$ are invertible, while $\tilde B$ is invertible if and only if $\Cc$ is semisimple  (Proposition~\ref{prop:cy-ssi-inclusion}).

Recall the definition of the constants $b_Q \in k^\times$ from \eqref{eq:phi_U_PV_simproj} and of the subset $\mathrm{IrrProj}(\Cc) \subset \Irr(\Cc)$ from \eqref{eq:IrrProj-def}. We remark that in order to get a basis in \eqref{eq:choice-Hig-basis1} one necessarily has $\mathrm{IrrProj}(\Cc) \subset J$.

\begin{proposition}\label{prop:M-via-S}
For all $U,V \in\Irr(\Cc)$ and $X \in J$
	we have the following identity in $k$,
\be\label{eq:M-via-S}
\sum_{W \in \Irr(\Cc)} \hspace{-.5em} M_{UV}^{~W} \, \tilde B_{WX}
~=
 \hspace{-.1em}
\sum_{Q \in \mathrm{IrrProj}(\Cc)} \hspace{-.9em}
b_Q\,  (\tilde B\cdot \tilde S)_{UQ} \,  (\tilde B\cdot \tilde S)_{VQ}  
\, (\tilde S \cdot \tilde C)_{QX}  \ .
\ee
\end{proposition}

\begin{proof}
{}From \eqref{eq:proj-M-def} we get the identity
\be
	\mathscr{S}_\Cc(\phi_{P_U}) \circ  \mathscr{S}_\Cc(\phi_{P_V})
	~= \sum_{W \in \Irr(\Cc)} M_{UV}^{~W} \, \mathscr{S}_\Cc(\phi_{P_W})
	\ .
\ee
Substituting \eqref{eq:tildeB-expansion} and \eqref{eq:tildeS-in-basis-def} gives
\begin{align}
\sum_{A,B,X,Y \in J} \hspace{-.8em}
\tilde B_{UA} \, \tilde B_{VB} 
\tilde S_{AX} \, \tilde S_{BY} \, \phi_{P_X} \circ \phi_{P_Y}
&~=
\sum_{W \in \Irr(\Cc)} \sum_{R,Z \in J} 
M_{UV}^{~W}  \, \tilde B_{WR} \, \tilde S_{RZ} \, \phi_{P_Z}
\nonumber\\
\overset{\text{Prop.\,\ref{prop:prod-Rey-Hig}}}=
\sum_{Z \in \mathrm{IrrProj}(\Cc)}
(\tilde B \cdot \tilde S)_{UZ}
(\tilde B \cdot \tilde S)_{VZ}
\,
b_Z
\,
\phi_{P_Z} \ .
&
\label{eq:M-via-S_aux1}
\end{align}
Next note that $\mathscr{S}_\Cc(\mathscr{S}_\Cc(\phi_M)) = \phi_{M^*}$ for all $M \in \Cc$ (as the square of $\mathscr{S}_\Cc$ results in the action of the antipode of  $\coend$,
 see \cite{Lyubashenko:1995} and 
 e.g.\ \cite[Lem.\,5.7]{FGRprep1}). 
 This results in the matrix identity $\tilde S \cdot \tilde S = \tilde C$.
 Hence $\tilde S$ and $\tilde C$ commute and $\tilde S^{-1} = \tilde S \cdot \tilde C$, as $\tilde C^2$ is the identity matrix.
Multiplying both sides of \eqref{eq:M-via-S_aux1} by $\tilde S^{-1}$ yields the claim.
\end{proof}

Note that \eqref{eq:M-via-S} does not determine $M_{UV}^{~W}$ uniquely unless $\Cc$ is semisimple 
	and $k$ is of characteristic zero.
Indeed, otherwise $\tilde B$ does not have a right inverse
	and $M_{UV}^{~W}$ is considered as an element of $k$, i.e.\ modulo $\mathrm{char}(k)$.

\begin{remark}\label{rem:VOA-S-trans-proj}
~\\[-1.5em]
\begin{enumerate}\setlength{\leftskip}{-1em}
\item
Recall from Remark~\ref{rem:motivation}\,(2) that one of the motivations for this paper comes from vertex operator algebras. Let $V$ be a vertex operator algebra as in  Remark~\ref{rem:motivation}\,(2). 
In \cite[Conj.\,5.10]{Gainutdinov:2016qhz}, a precise conjecture is made on how the behaviour of so-called pseudo-trace functions of $V$ under the modular 
$S$-transformation is related to $\mathscr{S}_\Cc$ for $\Cc = \Rep(V)$. In this relation, $\phi_M \in \End(\Id_\Cc)$ corresponds to the character of the $V$-module $M$ (and not to a pseudo-trace function). Characters and their modular properties are much more accessible than
	those of
pseudo-trace functions. That is, from a vertex operator algebra perspective, the following pieces of information are typically more accessible than the structure constants $M_{UV}^{~W}$ themselves:
\begin{itemize}
\item $\tilde B_{UA}$, $\tilde C_{AB}$ (find the projective modules of $V$ and analyse their composition series as well as those of their contragredient modules),
\item $\tilde S_{AB}$ (compute the modular $S$-transformation of characters of projective modules),
\item $b_Q$ (compute the modular $S$-transformation of the vacuum character and take the inverse of the coefficient of the character of $Q$).
\end{itemize}
Here the text in brackets indicates the computation one needs to do for $V$-modules, provided one assumes the conjectures in \cite{Gainutdinov:2016qhz}. 
For the last point note that
by \eqref{eq:eq:sigma-Gr-EndId-def_pic} and \eqref{eq:smap-via-SC-phiM} we have 
\be\label{eq:b-from-S1Q}
  \mathscr{S}_\Cc(\phi_\one) = \id = b_Q^{-1} \phi_Q + (\text{terms that vanish when evaluated on $Q$}) \ .
\ee
Thus one can expect all the data on the right hand side of \eqref{eq:M-via-S} to be relatively accessible for a vertex operator algebra $V$ (as compared to computing fusion of $V$-modules). We will see an example of this in Section~\ref{sec:sf-example}.
\item
If $\Cc$ is semisimple, then Proposition~\ref{prop:M-via-S} is precisely the categorical Verlinde formula:
Firstly, $K_0(\Cc) = \Gr(\Cc)$, so that $M_{UV}^{~W}$ are the structure constants of $\Gr(\Cc)$. 
Next, by \cite[Rem.\,3.10\,(3)]{Gainutdinov:2016qhz} we have
$\mathscr{S}_\Cc(\phi_U) = (\mathrm{Dim}\,\Cc)^{-\frac12}
 \sum_{X \in \Irr(\Cc)} 
s_{U^*,X}\, \phi_X$, where $\mathrm{Dim}\,\Cc = \sum_{X \in \Irr(\Cc)} \dim(X)^2$ and $s_{A,B} \,\id_\one = \mathrm{tr}(c_{B,A} \circ c_{A,B})$. Thus $\tilde S_{AB} = s_{A^*,B} / \sqrt{\mathrm{Dim}\,\Cc}$. Finally, $\tilde B_{AB} = \delta_{A,B}$, $\tilde C_{AB} = \delta_{A,B^*}$ and from \eqref{eq:b-from-S1Q} we see $b_Q= s_{\one,Q} / \sqrt{\mathrm{Dim}\,\Cc}$. Altogether,
\eqref{eq:M-via-S} becomes (using $s_{X,Y} = s_{Y,X} = s_{X^*,Y^*}$)
\be
M_{UV}^{~W} 
~=~ \frac{1}{\mathrm{Dim}\,\Cc}
\sum_{Q \in \mathrm{Irr}(\Cc)} 
\frac{s_{U^*,Q}\,  s_{V^*,Q} \,  
	s_{Q,W}
}{s_{\one,Q}}
~=~ \frac{1}{\mathrm{Dim}\,\Cc}
\sum_{Q \in \mathrm{Irr}(\Cc)} 
\frac{s_{U,Q}\,  s_{V,Q} \,  s_{W^*,Q}}{\dim(Q)} \ ,
\ee
	where $M_{UV}^{~W}$ is considered as an element of $k$. This is the semisimple categorical Verlinde formula, see \cite[Thm.\,4.5.2]{tur}.

\item A variant of the fusion algebra of projective characters was studied~\cite[Cor.\,4.4]{Lachowska-Verlinde} in the context of the small quantum groups for  simply laced simple Lie algebras. For  factorisable ribbon Hopf algebras, a Verlinde-like formula for the action of the Grothendieck ring on the Higman ideal is given in~\cite[Thm.\,3.14]{Cohen:2008}.
\end{enumerate}
\end{remark}

It is an interesting problem to relate the modified trace on ``open Hopf links'' and properties of the modular $S$-transformation \cite{Creutzig:2016fms}. The above results allow us to do this in the case the open Hopf link is coloured by
 projective objects:

\begin{proposition}\label{prop:mod-tr-S}
For all $A \in J$, $X \in \Irr(\Cc)$ and
 $Q \in \mathrm{IrrProj}(\Cc)$ we have
\be\label{eq:mod-tr-S}
\frac{t_{P_X}\bigl(\sigma([P_A])_{P_X}\bigr)}{t_Q(\id_Q)} ~=~ b_Q \sum_{B\in J} \tilde S_{AB} \, 
	\widehat\CM_{BX} \ .
\ee
\end{proposition}

\begin{proof}
Apply $t_{P_X}(\cdots)$ to both sides of \eqref{eq:tildeS-in-basis-def}, expand $\phi_{P_B}$ as $\phi_{P_B} = \sum_{W \in \Irr(\Cc)} [P_B:W] \phi_W$ and use \eqref{eq:mod-trace-on-phiA} and \eqref{eq:tP1phiP1-via-tQidQ}.
\end{proof}

In particular, for $X=Q\in \mathrm{IrrProj}(\Cc)$ we have
\be\label{eq:modtr-modS-simpleproj}
\frac{t_{X}(	\sigma([P_A])_{X}
)}{t_X(\sigma([\one])_X)} ~=~ \frac{\tilde S_{AX}}{ b_X^{-1}}   \ .
\ee
Moreover, by Lemma~\ref{lem:exchange-partial-trace} we have $t_{P_B}(\sigma([P_A])_{P_B}) = t_{P_A^*}(\sigma([P_B^*])_{P_A^*})$, so that \eqref{eq:mod-tr-S} implies that for $A,B \in J$,
\be
	\big(\tilde S \cdot \widehat\CM \big)_{AB}
	~=~
	\big(\tilde S \cdot \widehat\CM \big)_{B^*A^*} \ ,
\ee
where by abuse of notation we restrict 
	$\widehat\CM$ 
to be a $J \times J$-matrix.

\begin{remark}
Equation \eqref{eq:modtr-modS-simpleproj} is a special case of a relation observed in the example of $W_p$-models in \cite[Sec.\,3.3]{Creutzig:2016fms}. The notation used there is related to ours as $S^{\circ\!\!\circ;\mathcal{P}}_{M,P} = t_P(\sigma([M])_P)$, where $M \in \Cc$, $P \in \Proj(\Cc)$.
In \cite{Creutzig:2016fms}, the matrix of the modular $S$-transformation of pseudo-trace functions in some basis is denoted by $S^\chi$. 
Using these notations,
\eqref{eq:modtr-modS-simpleproj} reads $S^{\circ\!\!\circ;\mathcal{P}}_{P_A,X} / S^{\circ\!\!\circ;\mathcal{P}}_{\one,X} = S^\chi_{P_A,X} / S^\chi_{\one,X}$ for $A \in \Irr(\Cc)$, $X \in \mathrm{IrrProj}(\Cc)$ (assuming the conjectures in 
\cite[Sec.\,5]{Gainutdinov:2016qhz} 
so that we can express
$\tilde S_{AX}$ and $b_X^{-1}$ in terms of modular properties of pseudo-trace functions).
\end{remark}

\section{Example: symplectic fermions}\label{sec:sf-example}

In this section we consider a family of examples of 
factorisable finite ribbon categories, namely the so-called symplectic fermion 
categories $\SF(\h,\beta)$. The aim is to compute the modified trace for this class of examples
 and to illustrate the use of Proposition~\ref{prop:M-via-S}.

\medskip

Let $\h$ be a non-zero finite-dimensional symplectic vector space over $\Cb$ and let $\beta \in \Cb$ satisfy $\beta^4 = i^{\dim\h}$. The $\Cb$-linear finite ribbon category $\SF(\h,\beta)$ was introduced in \cite{Davydov:2012xg,Runkel:2012cf}. Conjecturally, it is ribbon-equivalent to the representations of the even part of the symplectic fermion vertex operator super-algebra constructed from $\h$ (hence the name), see \cite[Conj.\,7.3]{Davydov:2016euo} for the precise statement and further references.

One can show that $\SF(\h,\beta)$ is ribbon-equivalent to a certain finite-dimensional quasi-triangular quasi-Hopf algebra \cite{Gainutdinov:2015lja,FGRprep2}, but we will not make
use of this here.
Instead we sketch the construction in \cite{Davydov:2012xg}. We refer to \cite[Sec.\,5.2]{Davydov:2012xg} or to e.g.\ \cite{Davydov:2013ty,FGRprep2} 
for further details.

\medskip
\newcommand{\algS}{\mathbb{\Lambda}}

Denote the symmetric tensor category of finite-dimensional complex super-vector spaces  by $\sVect$. 
We define
\be
\algS = \Lambda(\h) \ ,
\ee 
the exterior algebra of $\h$, and consider it as an algebra in $\sVect$ by taking the $\Zb$-grading of $\Lambda(\h)$ modulo 2.
In particular, $\dim(\algS) = 2^{\dim\h}$ and $\mathrm{sdim}(\algS)=0$, 
where $\mathrm{sdim}(-)$ denotes the super-dimension. $\algS$ carries the structure of a 
	commutative and cocommutative
quasi-triangular Hopf algebra in $\sVect$.
 We denote the  multiplication, comultiplication, unit and counit by 
$\mu_\algS$, $\Delta_\algS$, $\eta_\algS$, and $\eps_\algS$, respectively
(see \cite[Eqn.\,(5.4)]{Davydov:2012xg} for the coalgebra structure).
 
We will make use of a specific cointegral $\coint_\algS : \algS \to \Cb$ for $\algS$
(see \cite[Eqn.\,(5.16)]{Davydov:2012xg}). 
To describe it, let $d = \dim\h$ and $a_1,\dots,a_d$ be a symplectic basis of $\h$ such that $(a_1,a_2)=1$, $(a_3,a_4)=1$, etc. Then $\coint_\algS$ is non-zero only on the top component of $\algS$, where it takes the value
\be
	\coint_\algS(a_1a_2\cdots a_d) ~=~ \beta^{-2} \ .
\ee

The category $\SF(\h,\beta)$ is the direct sum of two
full subcategories,
\be
	\SF(\h,\beta) ~=~
	\SF_0 ~\oplus~ \SF_1 
	\quad , \qquad
	\SF_0 = \Rep_{\sVect}(\algS)
	~~,~~
	\SF_1 = \sVect \ ,
\ee
where $\Rep_{\sVect}(\algS)$ denotes the category of left $\algS$-modules in $\sVect$. 
$\SF_0$ and $\SF_1$ each contains two simple objects (up to isomorphism), which we denote as
\be\label{eq:proj-covers-SF}
	\underbrace{\one = \Cb^{1|0}~,~
	\Pi\one = \Cb^{0|1}}_{\in\SF_0}
	\quad , \quad
	\underbrace{T = \Cb^{1|0}~,~
	\Pi T = \Cb^{0|1}}_{\in\SF_1} \ .
\ee
Here, $\Pi$ is the parity-flip functor. The $\algS$-action on $\one$ and $\Pi\one$ is trivial.

Since $\SF_1$ is semi-simple, $T$ and $\Pi T$ are also projective. The projective cover of $\one$ is $\algS$ itself, with $\pi_\one : \algS \to \one$ being the projection on the top component of $\algS$ (together with a choice of isomorphism to $\Cb^{1|0}$). Altogether, the projective covers are
\be\label{eq:SF-proj-cov}
	P_\one = \algS
	~~,\quad
	P_{\Pi\one} = \Pi \algS
	~~,\quad
	P_T = T
	~~,\quad
	P_{\Pi T} = \Pi T \ .
\ee
In particular, the Cartan matrix in \eqref{eq:composition-series-proj-cover} reads, in the ordering \eqref{eq:proj-covers-SF} of simple objects,
\be
	\CM(\SF(\h,\beta)) ~=~
	\begin{pmatrix}
	2^{2N-1} & 2^{2N-1} & 0 & 0  \\
	2^{2N-1} & 2^{2N-1} & 0 & 0  \\
	0 & 0 & 1 & 0 \\
	0 & 0 & 0 & 1
	\end{pmatrix}
	\qquad , \quad
	\text{where} \quad
	N = \tfrac12 \dim\h \ . 
\ee

We write 
$* : \SF(\h) \times \SF(\h) \to \SF(\h)$ 
for the tensor product functor and will reserve the notation $\otimes$ for the tensor product in $\sVect$ and for that of $\algS$-modules in $\sVect$. 
The tensor product $*$ is $\Zb_2$-graded, and depending on which sector $X,Y \in \SF(\h)$ are chosen from, it is defined to be:
\begin{equation}\label{eq:*-tensor}
X \ast Y ~=~
\left\{\rule{0pt}{2.8em}\right.
\hspace{-.5em}\raisebox{.7em}{
\begin{tabular}{ccll}
   $X$ & $Y$ & $X \ast Y$ &
\\
$0$ & $0$ & $X \otimes Y$ & $\in~\SF_0$
\\
 $0$ & $1$ & $X \otimes Y$ & $\in~\SF_1$
\\
 $1$ & $0$ & $X \otimes Y$ & $\in~\SF_1$
\\
$1$ & $1$ & $\algS \otimes X \otimes Y$ & $\in~\SF_0$
\end{tabular}}
\end{equation}
In sector $00$, the tensor product is that of $\Rep_{\sVect}(\algS)$, with $\algS$-action given by the coproduct (and using the symmetric braiding of $\sVect$).
In sectors $01$ and $10$,  the tensor product is that of the underlying super-vector spaces.
 In sector $11$, the $\algS$-action is given by the left regular action on $\algS$.

We refer to \cite{Davydov:2012xg,FGRprep2} for the associator and braiding isomorphisms in $\SF(\h,\beta)$, as well as for the ribbon twist and for the rigid and the pivotal structures. The $\beta$-dependence enters into these structure maps. We have:

\begin{proposition}[{\cite{Davydov:2012xg,FGRprep2}}]
$\SF(\h,\beta)$ is factorisable.
\end{proposition}

The embedding $\sVect \hookrightarrow \SF_0$, $X \mapsto X$, where $X$ is endowed with the trivial $\algS$-action, is an embedding of ribbon categories. Every projective object in $\SF(\h,\beta)$ is isomorphic to $\algS * X \oplus T * Y$ for some $X,Y \in \sVect$, embedded in $\SF_0$. Note that by definition,
\be
	\End_{\SF}(\algS*X) = \End_{\sVect,\algS}(\algS \otimes X) \quad , \quad
	\End_{\SF}(T*X) = \End_{\sVect}(X) \ ,
\ee
where $\End_{\sVect,\algS}(\algS \otimes X)$ are all even endomorphisms of the super-vector space $\algS \otimes X$ which commute with the $\algS$-action and $\End_{\sVect}(X)$ are all even endomorphisms of the super-vector space $X$.
Write $\mathrm{str}(-)$ for the super-trace.

\begin{proposition}\label{prop:SF-modified-tr}
Let $X$ be a finite-dimensional super-vector space
and $f \in \End_{\SF}(\algS * X)$, $g \in \End_{\SF}(T * X)$. Then, for each $t_0 \in \Cb^\times$,
\be\label{eq:SF-modified-tr}
	t_{\algS*X}(f) = t_0 \, \mathrm{str}\big( (\coint_\algS \otimes \id) \circ f \circ (\eta_\algS \otimes \id)\,\big)
	\quad , \quad
	t_{T*X}(g) = t_0 \, \mathrm{str}(g)
\ee
defines a modified trace on $\Proj(\SF(\h,\beta))$.
\end{proposition}

\begin{proof}
The proof makes use of the explicit associator and rigid structure of $\SF(\h,\beta)$. Since we only need them for this proof, we prefer to refer to 
	\cite{Davydov:2012xg,FGRprep2} 
for the explicit formulas, rather than repeat them here.
We  use the simple projective 
	object
$Q=T$ to start with as in~\eqref{eq:tQ-choice} and set 
$t_T(\id_T) = t_0$.

To compute $t_{\algS*X}(f)$, we note that by~\eqref{eq:*-tensor}, $\algS=T*T^*$ and that the associator $\alpha_{T,T^*,X}: T*(T^**X) \to (T*T^*)*X$ is the identity map for the trivial $\algS$-module $X$. 
Here, the dual $T^*$ is as in $\sVect$.
We define $\tilde f \in \End_{\SF}(T*(T^**X))$ as
\be
  \raisebox{-0.45\height}{\setlength{\unitlength}{.75pt}
  \begin{picture}(55,90)
   \put(0,10){\scalebox{0.95}{\includegraphics{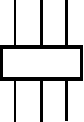}}}
     \put (5,88) {\scriptsize$ T $}    
     \put (21,88) {\scriptsize$ T^* $}     
      \put (37,88) {\scriptsize$ X $}  
     \put (20,42) {\scriptsize$ \tilde f $}
     \put (5,0) {\scriptsize$ T $}    
     \put (21,0) {\scriptsize$ T^* $}      
     \put (37,0) {\scriptsize$ X $}  
     \setlength{\unitlength}{1pt}
  \end{picture}
  \raisebox{60pt}{\framebox{$\SF$}}
  } 
  \quad := \qquad 
    \raisebox{-0.45\height}{\setlength{\unitlength}{.75pt}
  \begin{picture}(73,90)
   \put(0,10){\scalebox{0.95}{\includegraphics{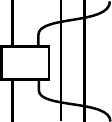}}}
     \put (5,88) {\scriptsize$ \algS $}     
     \put (35,88) {\scriptsize$ T $}    
     \put (48,88) {\scriptsize$ T^* $}      
     \put (63,88) {\scriptsize$ X $}  
     \put (12,43) {\scriptsize$ f $}
     \put (5,0) {\scriptsize$ \algS $}     
     \put (35,0) {\scriptsize$ T $}    
     \put (48,0) {\scriptsize$ T^* $}      
     \put (63,0) {\scriptsize$ X $}  
     \setlength{\unitlength}{1pt}
  \end{picture}
  \raisebox{60pt}{\framebox{$\sVect$}}
  }
  \quad .
\ee
Here, the boxes indicate in which category the string diagram is to be evaluated. In particular, the right hand side is given in $\sVect$ and  the crossings are the flips of super-vector spaces.
	We can now write
\be\label{eq:tSX-1}
	t_{\algS* X}(f) ~=~  
	t_{T*(T^* * X)}(\tilde f)
	~=~
	t_T(F)
	\qquad \text{where} \quad
	F~=
  \raisebox{-0.5\height}{\setlength{\unitlength}{.75pt}
  \begin{picture}(73,119)
   \put(0,0){\scalebox{.85}{\includegraphics{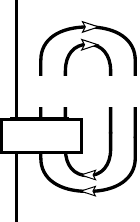}}}
     \put(0,0){
     \put (5,125) {\scriptsize$ T $}
     \put (18,67) {\scriptsize$ T^* $}      \put (33,67) {\scriptsize$ X $}    \put (55,67) {\scriptsize$ X^* $}     \put (73,67) {\scriptsize$ T^{**} $}
     \put (18,43) {\scriptsize$ \tilde f $}
          \put (7,-10) {\scriptsize$ T $}
     }
     \setlength{\unitlength}{1pt}
  \end{picture} 
  \raisebox{80pt}{\framebox{$\SF$}}
  }
\quad .
\ee
In the last equality we used the partial trace property~\eqref{eq:mod-tr_partial-trace},
	and we implicitly used that $\alpha_{T,T^*,X}$ is the identity map.
Next we write out $F$
in terms of the structure morphisms (recall that we omit the tensor product `\,$*$\,' between objects for better readability):
\begin{align}
F &~=~ \big[\,
T \xrightarrow{\sim} T\one
\xrightarrow{\id\otimes\coev_{T^*  X}}
T((T^*X)(T^* X)^*)
\xrightarrow{\alpha^{-1}\circ (\tilde f * \id) \circ \alpha }
T((T^*X)(T^* X)^*)
\nonumber\\
& \hspace{6em}
\xrightarrow{\id \otimes \widetilde\ev_{T^*X}}
T\one
\xrightarrow{\sim}
T
\,\big] \ ,
\label{eq:SF-modified-tr_aux1}
\end{align}
where $\alpha^{-1}\circ (\tilde f * \id)\circ \alpha$ stands for the composition $\alpha^{-1}_{T,T^*X,(T^* X)^*}\circ (\tilde f * \id)\circ \alpha_{T,T^*X,(T^* X)^*}$.
Finally, we compute $F$
using the explicit associator in the 111-sector (see \cite[Eqn.\,(2.27)]{Davydov:2012xg})
  and the duality maps from
  \cite[Eqns.\,(3.59), (4.77), (4.78)]{Davydov:2012xg}:
\be
F ~=
  \raisebox{-0.5\height}{\setlength{\unitlength}{.75pt}
  \begin{picture}(110,225)
   \put(0,10){\scalebox{.95}{\includegraphics{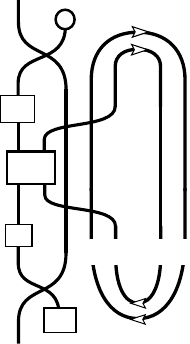}}}
     \put(0,10){
     \put (8,213) {\scriptsize$ T $}
        \put (22,200) {\scriptsize$ \varepsilon_\algS $}
                \put (1,140) {\scriptsize$ \phi^{-1}_\algS $}
     \put (15,105) {\scriptsize$  f $}
         \put (6,64) {\scriptsize$ \phi_\algS $}
     \put (51,53) {\scriptsize$ T^* $}      \put (66,53) {\scriptsize$ X $}    \put (88,53) {\scriptsize$ X^* $}     \put (106,53) {\scriptsize$ T^{**} $}
          \put (29,12) {\tiny$ \coint_\algS $}
          \put (8,-10) {\scriptsize$ T $}
     }
     \setlength{\unitlength}{1pt}
  \end{picture} 
  \raisebox{140pt}{\framebox{$\sVect$}}
  }
   ~~\overset{(*)}=~~   
     \raisebox{-0.5\height}{\setlength{\unitlength}{.75pt}
  \begin{picture}(75,130)
   \put(0,0){\scalebox{.95}{\includegraphics{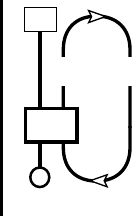}}}
     \put(0,0){
     \put (-3,135) {\scriptsize$ T $}
        \put (17,118) {\tiny$ \coint_\algS $}
     \put (28,52) {\scriptsize$  f $}
 \put (35,85) {\scriptsize$ X $}    \put (72,85) {\scriptsize$ X^* $} 
          \put (10,14) {\scriptsize$ \eta_\algS $}
          \put (-3,-10) {\scriptsize$ T $}
     }
     \setlength{\unitlength}{1pt}
  \end{picture}
  \;\raisebox{90pt}{\framebox{$\sVect$}}
  }
 \quad ,
\ee
where, $\phi_\algS\colon \algS \to \algS$ appears in the associator in the 111-sector. In the equality (*) we used the identities
$\eps_\algS \circ \phi_\algS^{-1} = \coint_\algS$ 
and 
$\phi_\algS \circ \coint_\algS = \eta_\algS$
(for the first identity, use \cite[Eqns.\,(3.45), (2.29)]{Davydov:2012xg} and the fact that in the present case the distinguished group-like element $g$ of \cite{Davydov:2012xg} is just $\eta_\algS$, for the second identity use \cite[Eqns.\,(4.84), (3.45)]{Davydov:2012xg}).
Altogether, $t_{\algS* X}(f) = t_T(F)$ gives
the expression in~\eqref{eq:SF-modified-tr}.

The computation of $t_{T*X}(g)$ also uses the partial trace property~\eqref{eq:mod-tr_partial-trace}: $t_{T*X}(g) = t_T(tr^r_X(g)) =  t_0 \, \mathrm{str}(g)$.
\end{proof}

Let us evaluate the modified trace from Proposition~\ref{prop:SF-modified-tr} for the four indecomposable projectives in \eqref{eq:proj-covers-SF}. Note that all (left) $\algS$-module maps $\algS \to \algS$ in $\sVect$ are given by right multiplication $R_a$ with an even element $a\in \algS$. The same holds for $\algS$-module maps $\Pi \algS \to \Pi \algS$.
\be
	t_{\algS}(R_a) = t_0\,\coint_\algS(a) 
	~~,\quad
	t_{\Pi \algS}(R_a) = -t_0\,\coint_\algS(a) 
	~~,\quad
	t_T(\id_T) = t_0
	~~,\quad
	t_{\Pi T}(\id_{\Pi T}) = -t_0
	\ . 
\ee

\medskip

After the computation of the modified trace, we now turn to an application of Proposition~\ref{prop:M-via-S}. Recall that $\Irr(\SF(\h,\beta)) = \{ \one, \Pi\one,T,\Pi T\}$. From \eqref{eq:SF-proj-cov} one sees that $[P_\one] = [P_{\Pi\one}]$, so that we can choose $J = \{ \one,T,\Pi T\}$. 
For the matrices $\tilde B$, $\tilde S$ and $\tilde C$ from \eqref{eq:tildeB-expansion}--\eqref{eq:tildeC-def} one finds, for the above order of basis vectors,
\be\label{eq:SF-proj-verlinde-data}
\tilde B
=
\begin{pmatrix}
1 & 0 & 0 \\
1 & 0 & 0 \\
0 & 1 & 0 \\
0 & 0 & 1 
\end{pmatrix}
\quad , \quad
\tilde S
=
\begin{pmatrix}
0 & 2^{N-1} & -2^{N-1} \\[.3em]
2^N & \tfrac12 & \tfrac12 \\[.3em]
-2^N & \tfrac12 & \tfrac12 
\end{pmatrix}
\quad , \quad
\tilde C
=
\begin{pmatrix}
1 & 0 & 0 \\
0 & 1 & 0 \\
0 & 0 & 1 
\end{pmatrix}
\quad , \quad
\ee
where $\tilde S$ has been computed in \cite{FGRprep2} (for $\beta=\mathrm{e}^{-\mathrm{i} N\pi/4}$). 
The constants $b_Q$ are (using again the $S$-action from \cite{FGRprep2} and \eqref{eq:b-from-S1Q})
\be
	b_T = 2^{-N-1} \quad,\quad
	b_{\Pi T} = -2^{-N-1}  \ .
\ee

\begin{remark}
~\\[-1.5em]
\begin{enumerate}\setlength{\leftskip}{-1em}
\item
The matrix $\tilde S$ can be read off directly from the modular 
$S$-transformation of characters for projective modules of the even part of the symplectic fermion vertex operator algebra (see \cite{Kausch:1995py,Gaberdiel:1996np,Abe:2005} and e.g.\ \cite[Sec.\,2.1]{Davydov:2016euo} for the modular properties of the characters of simple modules).
\item
Using Proposition~\ref{prop:mod-tr-S}, $\tilde S$ can be used to compute the modified trace of open Hopf link operators for projective objects. For example, \eqref{eq:mod-tr-S} with 
$A=Q=T$, $X=\one$ gives
\be
t_0^{-1} \, t_{P_\one}\big( \sigma([T])_{P_\one} \big)
~=~
b_T
\sum_{B \in J} \tilde S_{TB} \, \CM(\SF)_{B\one}
~=~ 2^{-N-1} \cdot 2^N \cdot 2^{2N-1} ~=~ 4^{N-1} \ .
\ee
\end{enumerate}
\end{remark}

Starting from the data in \eqref{eq:SF-proj-verlinde-data},
Equation \eqref{eq:M-via-S} allows one to compute, for all $U,V \in \Irr(\SF(\h,\beta))$,
\be
	M_{UV}^{~\one} + M_{UV}^{~\Pi\one}
	~~,\quad
	M_{UV}^{~T}
	~~,\quad
	M_{UV}^{~\Pi T} \ .
\ee 
Because $\tilde B$ is degenerate, we only obtain the sum $M_{UV}^{~\one} + M_{UV}^{~\Pi\one}$. However, in the present case we can do better. 
Namely, for finite tensor categories we have
\be
	M_{UV}^{~\one}
	= \dim_k\Cc(P_U \otimes P_V,\one)
	= \dim_k\Cc(P_U,P_V^*)
	= [P_V^*:U] \overset{\eqref{eq:cartan-mat-def}}= 
	\CM(\SF(\h,\beta))_{V^*U} \ .
\ee
So in the present example, knowledge of the composition series of the $P_U$ together with the data \eqref{eq:SF-proj-verlinde-data} actually fixes $M_{UV}^{~W}$ completely.

\appendix

\section{Proofs of Lemma~\ref{lem:dinat-Gr} and Corollary~\ref{cor:dinat-jac-0}}\label{app:dinat-exact}

\begin{proof}[Proof of Lemma~\ref{lem:dinat-Gr}]
Consider the product complex
\be
\xymatrix{
& 0 \ar[d] & 0 \ar[d] & 0 \ar[d] &
\\
0 \ar[r] & {}^*C \otimes A \ar[r]^{\id \otimes f} \ar[d]^{{}^*\!g \otimes \id} & {}^*C \otimes B \ar[r]^{\id \otimes g} \ar[d]^{{}^*\!g \otimes \id} & {}^*C \otimes C \ar[r] \ar[d]^{{}^*\!g \otimes \id} & 0
\\
0 \ar[r] & {}^*B \otimes A \ar[r]^{\id \otimes f} \ar[d]^{{}^*\!f \otimes \id} & {}^*B \otimes B \ar[r]^{\id \otimes g} \ar[d]^{{}^*\!f \otimes \id} & {}^*B \otimes C \ar[r] \ar[d]^{{}^*\!f \otimes \id} & 0
\\
0 \ar[r] & {}^*\!A \otimes A \ar[r]^{\id \otimes f} \ar[d] & {}^*\!A \otimes B \ar[r]^{\id \otimes g} \ar[d] & {}^*\!A \otimes C \ar[r] \ar[d] & 0
\\
& 0 & 0 & 0 &
 }
\ee
Since the tensor product functor is biexact by assumption, this complex has exact rows and columns.

Let $K \xrightarrow{\kappa} {}^*B \otimes B$ be the kernel of ${}^*\!f \otimes g$.
It is shown in the proof of \cite[Lem.\,2.5.1]{Geer:2010} that 
\begin{enumerate}
\item
the morphism 
$(\id \otimes f) \circ \pi_1 + ({}^*\!g \otimes \id) \circ \pi_2 : 
{}^*B \otimes A \,\oplus\, {}^*C \otimes B \to {}^*B \otimes B$, where $\pi_{1,2}$ are the projections on the first and second direct summand, factors through $K$ as $\kappa \circ u$ with $u : {}^*B \otimes A \,\oplus\, {}^*C \otimes B \to K$ surjective.
\item
there are (necessarily unique) morphisms
$\alpha : K \to {}^*\!A \otimes A$ and $\gamma : K \to {}^*C \otimes C$ such that
\be\label{eq:dinat-Gr_aux1}
	(\id \otimes f) \circ \alpha ~=~ ({}^*\!f \otimes \id) \circ \kappa
	\quad , \quad
	({}^*\!g \otimes \id) \circ \gamma ~=~ (\id \otimes g) \circ \kappa \ .	
\ee
\end{enumerate}
Write $\hat b := (\id_{{}^*\!B} \otimes b) \circ \widetilde\coev_B$. Note that
\be
	({}^*\!f \otimes g) \circ \hat b
	~=~ (\id \otimes (g \circ b \circ f)) \circ \widetilde\coev_A
	~=~ 0 \ ,
\ee
where in the last step we used the commuting diagram \eqref{eq:dinat-Gr_diag}. Thus $\hat b : \one \to {}^*B \otimes B$ factors through $K$, 
$\hat b = \big[ \one \xrightarrow{\tilde b} K \xrightarrow{\kappa} {}^*B \otimes B \big]$. We have
\begin{align}
	& (id \otimes f) \circ \alpha \circ \tilde b
	\overset{\eqref{eq:dinat-Gr_aux1}}= 
	({}^*\!f \otimes \id) \circ \kappa \circ \tilde b
	= 
	({}^*\!f \otimes \id) \circ (\id \otimes b) \circ \widetilde\coev_B
	\nonumber \\
	&=
	(\id \otimes (b\circ f)) \circ \widetilde\coev_A
	\overset{\eqref{eq:dinat-Gr_diag}}= 
	(id \otimes f) \circ (\id \otimes a) \circ \widetilde\coev_A \ .
\end{align}
Since $id \otimes f$ is injective, we get $\alpha \circ \tilde b = (\id \otimes a) \circ \widetilde\coev_A$. Along similar lines 
	-- using injectivity of ${}^*\!g \otimes \id$ --
one shows that $\gamma \circ \tilde b = (\id \otimes c) \circ \widetilde\coev_C$.
Define the morphism $\Gamma :\one \to X$ as
\be
	\Gamma ~:=~ \eta_A \circ \alpha \circ \tilde b ~+~ \eta_C \circ \gamma \circ \tilde b \ .
\ee
By the above calculation we see that $\Gamma$ is equal to the right hand side of \eqref{eq:dinat-Gr_rel}. To see that $\Gamma$ it also equal to the left hand side of \eqref{eq:dinat-Gr_rel}, we make use of the surjection $p : P \to \one$, where $P \in \Mc$ is projective. We will now show that $\Gamma \circ p = \eta_B \circ \hat b \circ p$, which completes the proof.

Since by point 1 above, $u: {}^*B \otimes A \,\oplus\, {}^*C \otimes B \to K$ is surjective (and since $P$ is projective), we can find $s$ such that the left square in the following commuting diagram commutes:
\be\label{eq:dinat-Gr_aux2}
\xymatrix{
P \ar[rr]^s \ar[d]^p && {}^*B \otimes A \,\oplus\, {}^*C \otimes B \ar[dl]^u \ar[d]^{(\id \otimes f) \circ \pi_1 + ({}^*\!g \otimes \id) \circ \pi_2}
\\
\one \ar[r]^{\tilde b} \ar@/_1pc/[rr]_{\hat b} & K \ar[r]^\kappa & {}^*B \otimes B
}
\ee
Now compute
\begin{align}
	(\id \otimes f) \circ \alpha \circ \tilde b \circ p
	&\underset{\eqref{eq:dinat-Gr_aux2}}{\overset{\eqref{eq:dinat-Gr_aux1}}=}
	({}^*\!f \otimes \id) \circ \kappa \circ u \circ s
	\overset{\text{(*)}}=
	({}^*\!f \otimes \id) \circ (id \otimes f) \circ \pi_1 \circ s
	\nonumber\\
	&=
	(id \otimes f) \circ ({}^*\!f \otimes \id) \circ \pi_1 \circ s \ ,
\end{align}
where in (*) we used the commuting diagram \eqref{eq:dinat-Gr_aux2} and the fact that ${}^*f \circ {}^*g = 0$. Since $id \otimes f$ is injective, we can conclude that $\alpha \circ \tilde b \circ p = ({}^*\!f \otimes \id) \circ \pi_1 \circ s$. Hence
\be\label{eq:dinat-Gr_aux3}
	\eta_A \circ \alpha \circ \tilde b \circ p
	~=~ \eta_A \circ ({}^*\!f \otimes \id) \circ \pi_1 \circ s
	\overset{\text{dinat.}}= \eta_B \circ (\id \otimes f) \circ \pi_1 \circ s \ .
\ee
Similarly one can show $\gamma \circ \tilde b \circ p = (\id \otimes g) \circ \pi_2 \circ s$ and
\be\label{eq:dinat-Gr_aux4}
	\eta_C \circ \gamma \circ \tilde b \circ p
	~=~ \eta_B \circ ({}^*g \otimes \id) \circ \pi_2 \circ s \ .
\ee
Combining \eqref{eq:dinat-Gr_aux3} and \eqref{eq:dinat-Gr_aux4} we finally find
\be
	\Gamma \circ p
	~=~
	\eta_B \circ (\id \otimes f) \circ \pi_1 \circ s
	+
	\eta_B \circ ({}^*g \otimes \id) \circ \pi_2 \circ s
	\overset{\eqref{eq:dinat-Gr_aux2}}=
	\eta_B \circ \hat b \circ p \ .
\ee
\end{proof}

\begin{proof}[Proof of Corollary~\ref{cor:dinat-jac-0}]
The proof of Part 1 follows that of \cite[Cor.\,2.5.2]{Geer:2010}.
	We give the details for the first equality, the second one is shown analogously using Remark~\ref{rem:all-endo-are-id}\,(1).

Let $A = \bigoplus_m A_m$ be a decomposition of $A$ into indecomposable objects $A_m \in \Cc$, and write 
	$j_m : A_m \to A$ 
and $p_m : A \to A_m$ for the corresponding embedding and projection maps. 
Writing $\id_A = \sum_m j_m \circ p_m$ and applying dinaturality of $\eta$ to $j_m$, we have
\begin{align}
\eta_A \circ (\id \otimes f) \circ \widetilde\coev_A
&~=~
\sum_m \eta_A \circ (\id \otimes (j_m \circ p_m \circ f)) \circ \widetilde\coev_A
\nonumber \\
&~=~
\sum_m \eta_{A_m} \circ (\id \otimes (p_m \circ f \circ j_m)) \circ \widetilde\coev_{A_m} \ .
\end{align}
Now for all simple $U \in \Cc$ and $u : A_m \to U$ we have $u \circ p_m \circ f \circ j_m = 0$ by assumption on $f$ (as $u \circ p_m : A \to U$). It is therefore enough to prove the corollary for indecomposable $A$. 

Let $A$ be indecomposable. By Fitting's Lemma, every element of $\End(A)$ is either invertible or nilpotent. By assumption, $f$ cannot be invertible, hence it is nilpotent, say $f^n = 0$. Let $K = \ker(f)$ and consider the commutative diagram
\be
\xymatrix{
 0 \ar[r] & K \ar[r]^\kappa \ar[d]^0 & A \ar[r]^\gamma \ar[d]^f & C \ar[r] \ar@{-->}[d]^{\exists!\,c} & 0
\\
 0 \ar[r] & K \ar[r]^\kappa & A \ar[r]^\gamma & C \ar[r] & 0
 }
\ee
with exact rows,
i.e.\ $(C,\gamma)$ 
is the cokernel of the kernel of $f$ aka.\ the image of $f$.
	The map~$c$ exists and is unique by the universal property of the cokernel.
By Lemma~\ref{lem:dinat-Gr} we have
\be\label{eq:induction-step-aux}
\eta_A \circ (\id \otimes f) \circ \widetilde\coev_A
~=~
0 + 	\eta_C \circ (\id \otimes c) \circ \widetilde\coev_C \ .
\ee

We proceed by induction on the degree $n$ of nilpotency. For $n=1$ the claim of part 1 of the corollary is clear. Suppose now the claim holds for all maps whose degree of nilpotency is less or equal to $n-1$.
Since $f \circ f^{n-1} = 0$, there is 
$h : A \to K$ such that $f^{n-1} = \kappa \circ h$. But then 
$c^{n-1} \circ \gamma = 
\gamma \circ f^{n-1} = \gamma \circ \kappa \circ h = 0$, and since $\gamma$ is surjective, $c^{n-1} = 0$. 
By our induction assumption, \eqref{eq:induction-step-aux} is zero.

\smallskip

For part 2, consider the dinatural transformation $\xi$ from  $(-) \otimes (-)^*$ to $B^* \otimes B$ with
\be
	\xi_A ~=~ 
	\raisebox{-0.5\height}{\setlength{\unitlength}{.75pt}
  \begin{picture}(76,135)
   \put(0,0){\scalebox{.75}{\includegraphics{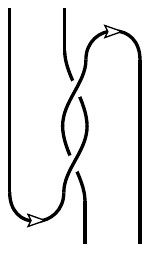}}}
   \put(0,0){
     \put(0,0){
     \put (40,0) {\scriptsize$ A $}
     \put (67,0) {\scriptsize$ A^* $}
     \put (4,128) {\scriptsize$ B^* $}
     \put (30,128) {\scriptsize$ B $}
     }\setlength{\unitlength}{1pt}}
  \end{picture}}
  \qquad .
\ee
To obtain \eqref{eq:f-loop-zero}, set $H := \xi_A \circ (f \otimes \id_{A^*}) \circ \coev_A$ and
use the second equality in part 1 to see $H=0$.
Then 
use $\widetilde\ev_B$ to turn $H : \one \to B^* \otimes B$ into an endomorphism of $B$, and after using the naturality of the braiding we get~\eqref{eq:f-loop-zero}.
\end{proof}

\section{Proofs of Lemma~\ref{lem:tau-R-welldef} and Proposition~\ref{prop:HigA-HigAmod}}\label{app:ReyHigideal}

For this appendix, we fix
\begin{itemize}
\item $\Ac$: a finite abelian category over	$k$,
\item $G$: a projective generator of $\Ac$,
\item $E = \End(G)$ the $k$-algebra of endomorphisms of $G$.
\end{itemize}

Write $\Id_{\Proj(\Ac)}$ for the identity functor on the projective ideal of $\Ac$, and $\End(\Id_{\Proj(\Ac)})$ for
the algebra of
 its natural endomorphisms.
The following lemma is standard.

\begin{lemma}\label{lem:nat-end-via-proj}
The map $\End(\Id_\Ac) \to \End(\Id_{\Proj(\Ac)})$, $\eta\mapsto\eta\raisebox{-.3em}{$|_{\Proj(\Ac)}$}$ 
is a $k$-algebra isomorphism.
\end{lemma}

\begin{proof}
Since $\eta \in \End(\Id_\Ac)$ is zero iff its values on all projective objects are zero,
 the map is injective. For surjectivity, let $A \in \Ac$ and pick an exact sequence $Q \xrightarrow{x} P \xrightarrow{y} A \to 0$ with $P,Q \in \Proj(\Ac)$. Consider the diagram
\be
\xymatrix{
Q \ar[r]^x \ar[d]^{\eta_Q} & P \ar[r]^y \ar[d]^{\eta_P} & A \ar[r] \ar@{-->}[d]^{\exists!\eta_A} & 0
\\
Q \ar[r]^x  & P \ar[r]^y & A \ar[r] & 0
}
\ee 
The left square commutes by naturality of $\eta$ on projectives. 
The dashed morphism $\eta_A$ exists by the universal property of the cokernel $(A,y)$ of $x : Q \to P$. 
It is straightforward to verify that $\eta_A$ is independent of the choice of $P,Q,x,y$ and that the family $\{\eta_A\}_{A \in \Ac}$ is a 
 natural endomorphism of $\Id_\Ac$ extending $\eta$ from $\Proj(\Ac)$ to all of $\Ac$.
\end{proof}

\begin{proof}[Proof of Lemma~\ref{lem:tau-R-welldef}]
Let $y \in \End(R)$ and 
define the family $(\eta_P)_{P \in \Proj(\Ac)}$ as in \eqref{eq:tau_R-defn}:
\be
	\eta_P := \sum_{(\gamma_{PR})} \gamma_{PR}' \circ y \circ \gamma_{PR}'' \ .
\ee
We will show that $\eta \in \End(\Id_{\Proj(\Ac)})$. The statement of the lemma then follows from Lemma~\ref{lem:nat-end-via-proj}.

Let $P,Q \in \Proj(\Ac)$ and $f : P\to Q$ be given. 
We need the following auxiliary result. Let $X$ and $Y$ in $\Ac(R,Q) \otimes \Ac(P,R)$ be defined as
\be\label{eq:tau-R-welldef-aux1}
	X = \sum_{(\gamma_{QR})} \gamma_{QR}' \otimes (\gamma_{QR}'' \circ f)
	\quad , \quad
	Y = \sum_{(\gamma_{PR})} (f \circ \gamma_{PR}') \otimes \gamma_{PR}''  \ .
\ee
Write $X = \sum_{(X)} X' \otimes X''$ and dito for $Y$.
Then for all $z : Q \to R$ we have
\begin{align}
\sum_{(X)} (z,X')_{QR} \, X''
&=
\sum_{(\gamma_{QR})}  (z, \gamma_{QR}')_{QR} \, \gamma_{QR}'' \circ f
\overset{\eqref{eq:coparing-gamma-properties}}=
z \circ f \ ,
\nonumber\\
\sum_{(Y)} (z,Y')_{QR} \, Y''
&=
\sum_{(\gamma_{PR})} (z,f \circ \gamma_{PR}')_{QR} \, \gamma_{PR}''
\overset{\eqref{eq:paring-projective}}=
\sum_{(\gamma_{PR})} (z \circ f , \gamma_{PR}')_{PR} \, \gamma_{PR}''
\overset{\eqref{eq:coparing-gamma-properties}}=
z \circ f \ .
\end{align}
By non-degeneracy of the pairings we obtain $X=Y$.
Using this identity, we get
\be
f \circ \eta_P = \sum_{(Y)} Y' \circ y \circ Y''
= \sum_{(X)} X' \circ y \circ X'' = \eta_Q \circ f \ ,
\ee
in other words, the family $(\eta_P)_{P \in \Proj(\Ac)}$ is natural in $\Proj(\Ac)$.
\end{proof}

To prepare the proof of Proposition~\ref{prop:HigA-HigAmod}, we need some further properties of the maps $\tau_R$ in \eqref{eq:tau_R-defn}, and we need a  description of Calabi-Yau trace maps in terms of their restriction to indecomposable projectives.

\begin{lemma}\label{lem:tau-and-splittings}
Let $R,S \in \Proj(\Ac)$ and let $a : S \to R$, $b: R \to S$. 
Then $\tau_R(a \circ b) = \tau_S(b \circ a)$.
\end{lemma}

\begin{proof}
By an argument similar to that showing $X=Y$ in \eqref{eq:tau-R-welldef-aux1} above, one obtains that for all $P,R,S \in \Proj(\Ac)$ and 
$b : R \to S$,
\be\label{eq:tau-and-splittings-aux1}
	\sum_{(\gamma_{PR})} \gamma_{PR}' \otimes (	b
	 \circ \gamma_{PR}'')
	~=~
	\sum_{(\gamma_{PS})} (\gamma_{PS}' \circ 
	b) \otimes \gamma_{PS}''  \ .
\ee
Then
\be
	\tau_R(a \circ b)_P
	\overset{\eqref{eq:tau_R-defn}}=
	\sum_{(\gamma_{PR})} \gamma_{PR}' \circ a \circ (b \circ \gamma_{PR}'')
	\overset{\eqref{eq:tau-and-splittings-aux1}}=
	\sum_{(\gamma_{PS})}
	 (\gamma_{PS}' \circ b) \circ a \circ \gamma_{PS}''
	\overset{\eqref{eq:tau_R-defn}}=
	\tau_S(b \circ a)_P \ ,
\ee
where we used~\eqref{eq:tau-and-splittings-aux1} with both sides composed with $a\otimes \id$.
\end{proof}

\begin{corollary}\label{cor:HigA-from-projgen}
$\Hig(\Ac) = \mathrm{im}(\tau_G)$.
\end{corollary}

\begin{proof}
Let $R \in \Proj(\Ac)$ and $h \in \End(R)$ be given. Pick an $m>0$ such that there is a splitting $p : G^{\oplus m} \to R$, $i : R \to G^{\oplus m}$, $p \circ i = \id_R$. Let $p_j : G \to R$, $i_j : R \to G$ be the components maps, that is, $\id_R = \sum_{j=1}^m p_j \circ i_j$. Then
\be
	\tau_R(h)
	=
	\sum_{j=1}^m \tau_R(p_j \circ i_j \circ h)
	\overset{\text{Lem.\,\ref{lem:tau-and-splittings}}}=
	\sum_{j=1}^m \tau_G(i_j \circ h \circ p_j) \ .
\ee
Thus $\tau_R(h)$ equals the image of $\sum_{j=1}^m i_j \circ h \circ p_j \in \End(G)$ under $\tau_G$.
\end{proof}

For $T \in \prod_{U \in \Irr(\Ac)} (\End(P_U)\to k)$ denote by $T_{U}$ the component of $T$ in $\End(P_U)\to k$. Write
\begin{align}\label{eq:Tc}
\Tc ~=~
\Big\{ T \in \hspace{-.5em}\prod_{U \in \Irr(\Ac)} \hspace{-.5em} (\End(P_U)\to k) ~\Big|~
&\text{$T$ makes the full subcategory
 with}
\nonumber \\[-1.5em]
& \hspace{2em}
 \text{objects $P_U$, $U \in \Irr(\Ac)$, Calabi-Yau } \Big\} \ .
\end{align}
In other words, $T \in \Tc$ iff for all $V,W \in \Irr(\Ac)$, the pairings $(-,-)_{P_VP_W} : \Ac(P_V,P_W) \times \Ac(P_W,P_V) \to k$, $(f,g)_{P_VP_W} = T_W(f \circ g)$ are non-degenerate and symmetric in the sense that $(f,g)_{P_VP_W} = (g,f)_{P_WP_V}$.

	Recall the notation $\mathrm{CY}(\Proj(\Ac))$ for the set of trace maps $(\tCY_P)_{P \in \Proj(\Ac)}$ which turn $\Proj(\Ac)$ into a Calabi-Yau category (Definition~\ref{def:CY-cat}). 

\begin{lemma}\label{lem:CY-structures-vs-setT}
The map
\be\label{eq:CY-structures-vs-setT}
	\mathrm{CY}(\Proj(\Ac)) \to \Tc
	\quad , \quad
	(\tCY_P)_{P \in \Proj(\Ac)} \mapsto (\tCY_{P_U})_{U \in \Irr(\Ac)} 
\ee
is a $k$-linear isomorphism.
\end{lemma}

\begin{proof}
	~\\
	{\em Injectivity:}
Let $P \in \Proj(\Ac)$ be given. Write $P = \bigoplus_{U \in \Irr(\Ac)} P_U^{\oplus n_U}$ and let $j_{U\alpha} : P_U \to P$, 
$p_{U\alpha} : P \to P_U$, 
$U \in \Irr(\Ac)$ and $\alpha = 1,\dots,n_U$, be the embedding and projection maps of the individual summands. By cyclicity
and linearity of $\tCY_{P_U}$, we have the decomposition
\be\label{eq:CY-structures-vs-setT_aux1}
	\tCY_P(f)
	~= \sum_{U \in \Irr(\Ac)} \sum_{\alpha=1}^{n_U}
	\tCY_{P_U}(p_{U\alpha} \circ f \circ j_{U\alpha}) \ ,
\ee
where we used $\sum_{U,\alpha} e_{U\alpha} = \id_P$ for the primitive idempotents $e_{U\alpha} =  j_{U\alpha} \circ p_{U\alpha}$.
Thus $(\tCY_P)_{P \in \Proj(\Ac)}$ is uniquely determined by the values $(\tCY_{P_U})_{U \in \Irr(\Ac)}$, showing that the map \eqref{eq:CY-structures-vs-setT} is injective.

\medskip
	
	\noindent
	{\em Surjectivity:}
Given $(\tCY_{P_U})_{U \in \Irr(\Ac)} \in \Tc$, we can define all $\tCY_P$ via \eqref{eq:CY-structures-vs-setT_aux1}. One checks that $\tCY_P$ is independent of the choice of embedding and projection maps. It remains to verify that $\tCY_P$ is non-degenerate and symmetric. 
Let $P,Q \in \Proj(\Ac)$ and let $j_{U\alpha}, p_{U\alpha}$ be a realisation of $P$ as a direct sum of indecomposable projectives
as above,
 and $j_{U\alpha}', p_{U\alpha}'$ one of $Q$.
\begin{itemize}\setlength{\leftskip}{-1em}
\item {\em (Symmetry)} Let $f : P \to Q$, $g: Q \to P$. 
We compute
\begin{align}
	\tCY_Q(f \circ g)
	&~=~
	\sum_{U,V,\alpha,\beta}
		\tCY_{P_U}
	(p_{U\alpha}' \circ f \circ j_{V\beta} \circ p_{V\beta} \circ g \circ j_{U\alpha}') 
	\nonumber \\
	&\overset{\text{sym.}}=
	\sum_{U,V,\alpha,\beta}
	\tCY_{P_V}( p_{V\beta} \circ g \circ j_{U\alpha}' \circ p_{U\alpha}' \circ f \circ j_{V\beta}) 
	~=~
	\tCY_P(g \circ f) \ .
\end{align}
\item
{\em (Non-degeneracy)}  Let $f : P \to Q$ be non-zero.
Then there exist $U,V,\alpha,\beta$ such that $p'_{V\beta} \circ f \circ j_{U\alpha} \neq 0$. By non-degeneracy of the $\tCY_{P_U}$ there is $\tilde g : P_V \to P_U$ such that
\be
0 \neq \tCY_{P_V}(p'_{V\beta} \circ f \circ j_{U\alpha} \circ \tilde g)
\overset{\text{sym.}}=
\tCY_{Q}(f \circ j_{U\alpha} \circ \tilde g \circ p'_{V\beta})
\ee
Thus for $g = j_{U\alpha} \circ \tilde g \circ p'_{V\beta}$ we get $\tCY_Q(f \circ g) \neq 0$.
\end{itemize}
This shows that the map  \eqref{eq:CY-structures-vs-setT} is surjective.
\end{proof}

\begin{proof}[Proof of Proposition~\ref{prop:HigA-HigAmod}]
Recall that $G$ denotes a choice of projective generator of $\Ac$ and $E = \End(G)$.
We will show the following statements, which are equivalent to points 1 and~2 in the proposition
	upon
setting $\Ac = \rmod{A}$ and $G=A_A$ as the right regular representation.
\begin{enumerate}
\item $E$ admits a central form turning it into a symmetric algebra if and only if $\Proj(\Ac)$ admits a Calabi-Yau structure.
\item
If $E$ is symmetric, the map $\xi$ from \eqref{eq:EndId-ZE-iso} restricts to an isomorphism of algebras $\xi : \Hig(\Ac) \to \Hig(E)$.
\end{enumerate}

\noindent
{\em Part 1:} Clearly, if $\tCY \in \mathrm{CY}(\Proj(\Ac))$, then $\tCY_G : E \to k$ turns $E$ into a symmetric algebra
with the central form $\varepsilon = \tCY_G$.

Suppose conversely that $E$ is a symmetric algebra with respect to the central form $\eps : E \to k$. Write $G = \bigoplus_{U \in \Irr(\Ac)} P_U^{\oplus n_U}$ and write $j_{U\alpha} : P_U \to G$, 
$p_{U\alpha} : G \to P_U$,
$\alpha = 1,\dots,n_U$, for the embedding and projection maps of the individual summands. For $U \in \Irr(\Ac)$ define
\be
	T_U : \End(P_U) \to k
	\quad , \quad
	f \mapsto \eps(j_{U1} \circ f \circ p_{U1}) \ ,
\ee
i.e.\ we only use the component with $\alpha=1$ for each summand.
We now claim that $(T_U)_{U \in \Irr(\Ac)} \in \Tc$
(recall the definition~\eqref{eq:Tc}). 
By Lemma~\ref{lem:CY-structures-vs-setT} this will prove part 1.

\smallskip

\noindent
{\em Non-degeneracy:} Let $f : P_U \to P_V$ be non-zero. We need to find $g: P_V \to P_U$ such that $T_V(f \circ g) \neq 0$. Since $f$ is non-zero, so is $j_{V1} \circ f \circ p_{U1}$. By non-degeneracy of $\eps$, there is $\tilde g \in E$ such that $\eps(j_{V1} \circ f \circ p_{U1} \circ \tilde g) \neq 0$. Set $g := p_{U1} \circ \tilde g \circ j_{V1}$ and compute
\begin{align}
	T_V(f \circ g)
	&~=~
	\eps(j_{V1} \circ f \circ p_{U1} \circ \tilde g \circ j_{V1} \circ p_{V1})
	\nonumber \\
	&\overset{\text{cycl.}}=
	\eps(j_{V1} \circ p_{V1} \circ j_{V1} \circ f \circ p_{U1} \circ \tilde g)
	\nonumber \\
	&~=~
	\eps(j_{V1} \circ f \circ p_{U1} \circ \tilde g)
	~\neq~ 0 \ .
\end{align}

\noindent
{\em Symmetry:} Let $f : P_U \to P_V$, $g: P_V \to P_U$. We compute
\begin{align}
	T_V(f \circ g)
	&~=~
	\eps(j_{V1} \circ f \circ g \circ p_{V1})
	~=~
	\eps(j_{V1} \circ f \circ p_{U1} \circ j_{U1} \circ g \circ p_{V1})
	\nonumber \\
	&\overset{\text{cycl.}}=
	\eps(j_{U1} \circ g \circ p_{V1} \circ j_{V1} \circ f \circ p_{U1})
	~=~
	\eps(j_{U1} \circ g \circ f \circ p_{U1})
	\nonumber \\
	&~=~ T_U(g \circ f) \ .
\end{align}

\smallskip

\noindent
{\em Part 2:}
By the assumption that $E$ is symmetric and by part 1, 
$\Proj(\Ac)$ 
is Calabi-Yau via some $\tCY\in \mathrm{CY}(\Proj(\Ac))$. By Proposition~\ref{prop:HigA-indep-tr} (and by a similar argument for $\Hig(E)$), neither $\Hig(\Ac)$ nor $\Hig(E)$ depend on the choice of traces. In particular, we may replace whatever central form $E$ was originally equipped with by $\tCY_G$.

In terms of the central form $\tCY_G$, the map $\tau : E \to Z(E)$ from \eqref{eq:tau-map-A-ZA-def} reads
\be\label{eq:HigA-HigAmod_aux1}
	\tau(f)
	=
	\sum_{(\gamma_{GG})} \gamma_{GG}' \circ f \circ \gamma_{GG}'' 
	\overset{\eqref{eq:tau_R-defn}}= 
	(\tau_G(f))_G\ .
\ee
Using this, we compute
\be
	\xi(\Hig(\Ac))
	\overset{\text{Cor.\,\ref{cor:HigA-from-projgen}}}=
	\xi(\tau_G(E))
	\overset{\eqref{eq:EndId-ZE-iso}}= 
	(\tau_G(E))_G
	\overset{\eqref{eq:HigA-HigAmod_aux1}}= 
	\tau(E)
	= \Hig(E) \ .
\ee
\end{proof}

\newcommand\arxiv[2]      {\href{http://arXiv.org/abs/#1}{#2}}
\newcommand\doi[2]        {\href{http://dx.doi.org/#1}{#2}}
\newcommand\httpurl[2]    {\href{http://#1}{#2}}

\end{document}